\documentclass[12pt]{article}
\usepackage{amsfonts, amssymb, epsfig,subfigure}
\usepackage{mathrsfs}
\usepackage{stmaryrd}
\usepackage{graphicx,amsmath}
\usepackage{latexsym}
\usepackage{verbatim}
\usepackage{diagbox}

\usepackage{color}
\usepackage{subeqnarray}
\usepackage{cases}
\usepackage{amsmath}
\usepackage{lineno}

\newtheorem{expl}{Example}[section]
\newtheorem{remark}{Remark}[section]

\newtheorem{thm}{Theorem}[section]

\newtheorem{lem}[thm]{Lemma}
\newtheorem{rem}[thm]{Remark}

\makeatletter
\@addtoreset{equation}{section}

\newcommand{\beq}[1]{\begin{equation} \label{#1}}
\newcommand{\eeq}{\end{equation}}
\newcommand{\bed}{\begin{displaymath}}
\newcommand{\eed}{\end{displaymath}}
\newcommand{\bea}{\bed\begin{array}{rl}}
\newcommand{\eea}{\end{array}\eed}

\newcommand{\barray}{\begin{array}{ll}}
\newcommand{\earray}{\end{array}}

\newcommand{\Spvek}[2][r]{%
  \gdef\@VORNE{1}
  \left(\hskip-\arraycolsep%
    \begin{array}{#1}\vekSp@lten{#2}\end{array}%
  \hskip-\arraycolsep\right)}

\def\vekSp@lten#1{\xvekSp@lten#1;vekL@stLine;}
\def\vekL@stLine{vekL@stLine}
\def\xvekSp@lten#1;{\def\temp{#1}%
  \ifx\temp\vekL@stLine
  \else
    \ifnum\@VORNE=1\gdef\@VORNE{0}
    \else\@arraycr\fi%
    #1%
    \expandafter\xvekSp@lten
  \fi}
\allowdisplaybreaks[2]
\makeatletter
  \newcommand\figcaption{\def\@captype{figure}\caption}
  \newcommand\tabcaption{\def\@captype{table}\caption}
\makeatother

\def\sqr#1#2{{\vcenter{\vbox{\hrule height.#2pt
\hbox{\vrule width.#2pt height#1pt \kern#1pt
\vrule width.#2pt} \hrule height.#2pt}}}}

\newcommand{\E}{\mathbb{E}}

\newcommand{\RR}{\mathbb{R}}

\def\X{\Xi}

\def\<{\langle} \def\>{\rangle}

\def\1{\oslash} \def\2{\oplus} \def\3{\otimes} \def\4{\ominus}
\def\5{\circ} \def\6{\odot} \def\7{\backslash} \def\8{\infty}
\def\9{\bigcap} \def\0{\bigcup} \def\+{\pm} \def\-{\mp}
\def\[{\langle} \def\]{\rangle}

\def\nn{\nonumber}

\def\bc{\begin{center}}       \def\ec{\end{center}}
\def\ba{\begin{array}}        \def\ea{\end{array}}
\def\be{\begin{equation}}     \def\ee{\end{equation}}
\def\bea{\begin{eqnarray}}    \def\eea{\end{eqnarray}}
\def\beaa{\begin{eqnarray*}}  \def\eeaa{\end{eqnarray*}}

\begin{document}
\title{Explicit multiscale  numerical method for super-linear slow-fast stochastic differential equations}

\author{Yuanping Cui  \thanks{School of Mathematical Sciences, Tiangong University, Tianjin, 300387, China. Research
		of this author  was supported by the National Natural Science Foundation of China (No. 12401216). }
\and Xiaoyue Li \thanks{School of Mathematical Sciences, Tiangong University, Tianjin, 300387, China.
Research
of this author  was supported by the National Natural Science Foundation of China (No. 12371402 , 11971096
), the Tianjin Natural Science Foundation (24JCZDJC00830), the National Key R$\&$D Program of China (2020YFA0714102) and  the Natural Science Foundation of Jilin Province (No. YDZJ202101ZYTS154).}
\and Xuerong Mao
\thanks{Department of Mathematics and Statistics, University of Strathclyde, Glasgow G1 1XH, UK. Research of this author was supported by
the Royal Society (No. WM160014, Royal Society Wolfson Research Merit Award),
the Royal Society of Edinburgh (No. RSE1832).  }
}

\date{}

\maketitle

\begin{abstract}
This manuscript is dedicated to the numerical approximation of super-linear slow-fast stochastic differential equations (SFSDEs). Borrowing the heterogeneous multiscale idea, we propose an explicit multiscale Euler-Maruyama scheme suitable for SFSDEs with locally Lipschitz coefficients using an appropriate truncation technique. By the averaging principle, we establish the strong convergence of the numerical solutions to the exact solutions in the pth moment. Additionally, under lenient conditions on the coefficients, we also furnish a strong error estimate. In conclusion, we give two illustrative examples and accompanying numerical simulations to affirm the theoretical outcomes.
\end{abstract}

 \vspace{3mm}
 \noindent {\bf Keywords.}
 Slow-fast stochastic differential equations; Super-linearity; Explicit multiscale scheme; $p$th moment; Strong convergence.

\section{Introduction}\label{Tntr}
Stochastic modelling plays an essential role in many branches of science and industry. Especially, super-linear stochastic differential equations (SDEs) are usually used to describe  real-world systems  in various applications, for examples, the stochastic Lotka-Volterra  model in biology  for   the population growth (see e.g. \cite{MR2380366}), the elasticity of volatility model  arising  in finance   for the  asset price  (see e.g. \cite{MR1742310}) and the  stochastic Ginzburg-Landau equation stemming from statistical physics
in the study of phase transitions  (see e.g. \cite{Kloeden}). In many fields, various factors  change at different rates: some vary rapidly whereas others evolve slowly. As a result the  separation of fast and slow time scales arises in chemistry, fluid dynamics, biology, physics, finance and other fields (see e.g. \cite{MR2830582,Glimm, Harvey, MR3309627}). Stochastic systems with this characteristic are  studied extensively (see e.g. \cite{MR4286678, MR2409418, Grig, Fuke2016}) and are often modeled by the  slow-fast SDEs (SFSDEs)
\begin{equation}\label{e1}
	\begin{cases}
		\mathrm{d}x^{\varepsilon}(t)= b(x^{\varepsilon}(t),y^{\varepsilon}(t))\mathrm{d}t +\sigma(x^{\varepsilon}(t))\mathrm{d}W^1(t),
		\\ \mathrm{d}y^{\varepsilon}(t)=\displaystyle \frac{1}{\varepsilon}f(x^{\varepsilon}(t), y^{\varepsilon}(t))\mathrm{d}t +\frac{1}{\sqrt{\varepsilon}}g(x^{\varepsilon}(t),y^{\varepsilon}(t))\mathrm{d}W^{2}
		(t)
	\end{cases}
\end{equation}
with initial value $(x^{\varepsilon}(0),y^{\varepsilon}(0))=(x_{0},y_{0})\in \mathbb{R}^{n_1}\times \mathbb{R}^{n_2}$.
Here,   coefficients
\begin{align*}
	\displaystyle &b: \mathbb{R}^{n_{1}} \times \mathbb{R}^{n_{2}} \rightarrow\mathbb{R}^{n_{1}},~~
	\sigma: \mathbb{R}^{n_{1}} \rightarrow \mathbb{R}^{n_{1}\times d_{1}},~~f:\mathbb{R}^{n_{1}}\times \mathbb{R}^{n_{2}}\rightarrow \mathbb{R}^{n_{2}},~~
	g: \mathbb{R}^{n_{1}}\times \mathbb{R}^{n_{2}}\rightarrow\mathbb{ R}^{n_{2}\times d_{2}} 
\end{align*}
are  continuous, while $\{W^1(t)\}_{t\geq 0}$ and $\{W^2(t)\}_{t\geq 0}$ represent mutually independent  $d_{1}$-dimensional and $d_{2}$-dimensional Brownian motions, respectively. The  parameter $\varepsilon>0$ represents the ratio of nature time scales between   $x^{\varepsilon}(t)$ and $y^{\varepsilon}(t)$.  Especially,  as $\varepsilon\ll 1$, $x^{\varepsilon}(t)$ and $y^{\varepsilon}(t)$ are  called the slow component and fast component, respectively.
{ In various applications the time evolution of the slow component $x^{\varepsilon}(t)$ is under the spotlight. 
	Hence, our main aim is  to construct an appropriate  numerical scheme to approximate the slow component of \eqref{e1} with superlinear coefficients.}

{ The averaging principle provides
	a substantial simplification of the original system.   It essentially describes} the asymptotic behavior of the slow component as $\varepsilon\rightarrow 0$. Precisely,  the slow component $x^{\varepsilon}(t)$  converges to $\bar{x}(t)$ in strong or weak sense, which is the solution of
\begin{equation}\label{eq2.4}
	\mathrm{d}\bar{x}(t)=\bar{b}(\bar{x}(t))\mathrm{d}t+\sigma(\bar{x}(t))\mathrm{d}W^{1}(t),~~
	\bar{x}(0)=x_{0},
\end{equation}
where 
\begin{equation}\label{f11}
	\bar{b}(x)=\int_{\RR^{n_2}}b(x,y)  \mu^{x}(\mathrm{d}y).\end{equation}
Here
$\mu^{x}(\cdot)$ denotes
the unique invariant probability measure independent of $y_0$  of the transition semigroup of $y^{x, y_0 }(t)$  
satisfying equation  with  frozen slow component
\begin{align}\label{eq2.5}
	\mathrm{d}y^{x, y_0 }(t)=f(x, y^{x, y_0 }(t))\mathrm{d}t+g(x,y^{x, y_0 }(t))\mathrm{d}W^{2}(t)
\end{align}
with initial value $y^{x, y_0}(0)=y_0$, where $x$ is regarded as a parameter,  under assumptions ensuring its existence.  This paper emphasizes the strong averaging principle, namely, the convergence of  $x^{\varepsilon}(t)$  to $\bar{x}(t)$ is in the $p$th moment for some $p>0$.

 The averaging principle was originally developed by Khasminskii \cite{MR260052}. {Subsequently,  fruitful results on the averaging principle have been developed for SFSDEs with the linear growth coefficients (see e.g. \cite{MR1046602, A.Y., MR1328391,MR2264816, MR2165382,MR1050462,MR2744917}. Recently, growing interests have been drawn to the study of  the averaging principle for SFSDEs with super-linear growth coefficients.  Liu et al. \cite{MR4047972} proved the strong convergence of the averaging principle as the drift coefficients are locally Lipschitz continuous with respect to the slow and fast variables. Hong et al. \cite{MR4374850} gave the 1/6-order strong convergence rate  for a class of nonlinear stochastic partial differential equations (SPDEs).  Shi et al. \cite{20232011} obtained  the optimal convergence rate for SFSDEs driven by L\'evy processes, which slow drift coefficient satisfies the monotonicity condition and grows  polynomially. Furthermore, the strong averaging principles have  been developed for various kinds of  slow-fast stochastic systems, such as jump-diffusion processes (see e.g.  \cite{MR2338495, MR2831776})  and SPDEs (see e.g.  \cite{MR3040961, MR2822851, MR3679916}).

\par
The averaging principle is one of the key techniques in the theoretical analysis of SFSDEs. 
However, it is almost impossible  to  get the explicit  form  of the invariant measure $\mu^{x}$ in the averaged equation \eqref{eq2.4} due to the complicated
dynamics of  the frozen equation \eqref{eq2.5}. Thus, the form of the averaged equation  \eqref{eq2.4} is almost unknown, which becomes a major obstacle to solving or approximating it directly.  
Fortunately, the  heterogeneous multiscale  method (see e.g. \cite{MR1979846,MR2314852}) (HMM) was proposed   to approximate the  averaged equation numerically. This facilitated the development of  the numerical approximation theory for the SFSDEs. In 2003,  Vanden-Eijnden
\cite{MR1980483} proposed a numerical scheme for the deterministic multi-scale  system without rigorous  analysis. E et al. \cite{MR2165382} provided a thorough analysis of the convergence and efficiency of the HMM scheme for SFSDEs without slow diffusion term, where the slow drift and fast diffusion coefficients are bounded and the fast drift coefficient is a smooth function with bounded derivatives of any order. In 2006, Givon et al. \cite{MR2264816} developed the projective integration schemes for SFSDEs in which the slow drift and diffusion coefficients satisfy the Lipschitz condition and the fast drift and diffusion coefficients are bounded. In 2008, Givon et al. \cite{MR2443000} went a further step to extend  the projective integration schemes for  jump-diffusion systems. In 2010, Liu 
\cite{MR2644318} 
established the HMM numerical theory for the fully coupled SFSDEs, where the slow drift and diffusion coefficients are bounded, all  coefficients are smooth and have bounded derivatives with any order. Br\'ehier   developed the HMM scheme for the slow-fast parabolic stochastic partial differential equations (see e.g. \cite{MR3040961, MR4092407}).

All of the above studies  were carried out {under the linear growth condition}, so the Euler-Maruyama (EM) scheme is used as a macro solver to simulate the evolution of the slow component owing
to its simple algebraic structure and the cheap computational cost.
{The  super-linear growth coefficients of the slow-fast stochastic systems  bring   the  super-linear structure  to the averaged equation.} {For an example, consider a  SFSDE with a super-linear slow drift
	\begin{equation}\label{ex1}
		\begin{cases}
			\mathrm{d}x^{\varepsilon}(t)=\big(-(x^{\varepsilon}(t))^3-y^{\varepsilon}(t)
			\big)\mathrm{d}t+x^{\varepsilon}(t)\mathrm{d}W^{1}(t),\\
			\mathrm{d}y^{\varepsilon}(t)=\displaystyle \frac{1}{\varepsilon}\big(x^{\varepsilon}(t)
			-y^{\varepsilon}(t)\big)\mathrm{d}t+\frac{1}{\sqrt{\varepsilon}}\mathrm{d}W^2(t)
		\end{cases}
	\end{equation}
	with $(x^{\varepsilon}(0),y^{\varepsilon}(0))=(x_0,y_0).$ The corresponding equation with frozen slow component is described by
	\begin{align}\label{f14}
		\mathrm{d}y^{x, y_0}(s)=(x-y^{x, y_0}(s))\mathrm{d}s+\mathrm{d}W^{2}(s)
	\end{align}
	with initial value $y^{x, y_0}(0)=y_0$.
	By solving the Fokker-Planck equation,  the  invariant probability density of \eqref{f14} is $ \mu^{x}(dy)= \frac{e^{-(y-x)^2}}{\sqrt{\pi}}\mathrm{d}y.$
	Then the averaged equation is described by
	\begin{equation}\label{ex2}
		\mathrm{d}\bar{x}(t)=\big(-\bar{x}^{3}(t)-\bar{x}(t)\big)\mathrm{d}t+\bar{x}(t)\mathrm{d}W^{1}(t) 
	\end{equation}
	with $\bar{x}(0)=x_0$.
	As  pointed out by \cite{MR2795791}   the  EM approximation error of  \eqref{ex2} diverges to infinity  in $p$th moment for any $p\geq1$.  {In fact, the numerical solutions generated by the Projective Integration (PI) scheme with the EM scheme as the macro-solver,  as detailed in {{\cite[(4.1)-(4.4)]{MR2264816}}}, to predict the averaged equation \eqref{ex2}  blows up quickly, see Figure \ref{figN1}. So the dynamics of the numerical solutions by the PI scheme is completely different from those of the underlying exact ones.} Therefore,  using the  EM scheme as the macro solver to simulate the averaged equation of SFSDE with super-linear coefficients may lead to divergence. On the other hand,
	\cite{MR2165382} pointed out that although implicit numerical methods are feasible as the macro solver for the super-linear averaged equation,  the algorithm and implementation requires more computation costs. As a consequence,  to  construct an appropriate explicit numerical scheme for super-linear SFSDEs to overcome the numerical stiffness becomes an urgent target. 
	\begin{figure}[htp]
  \begin{center}
\includegraphics[width=10cm,height=6cm]{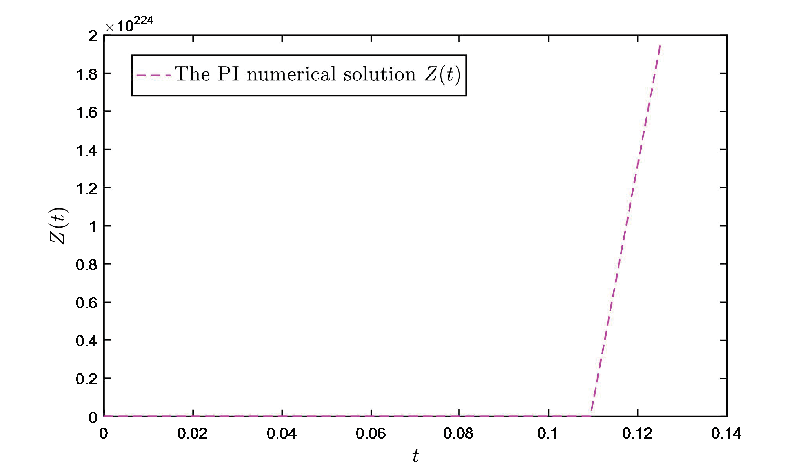}
   \end{center}
   \caption{The sample paths  of the PI numerical solution $Z(t)$  on $t\in [0,3]$ with $\Delta_1=2^{-6}$, $\Delta_2=2^{-6}$ and $M=2^{18}$. }
		\label{figN1}
\end{figure}}
\par Significant advancements have been made in the field of explicit numerical methods for super-linear SDEs, such as the development of the tamed Euler-Maruyama (EM) scheme (see e.g. \cite{MR3364862,MR2985171,MR3070913,MR3543890}), the tamed Milstein scheme (see e.g. \cite{MR3037286}), the stopped EM scheme (see e.g. \cite{MR3116272}), the truncated EM scheme (see e.g. \cite{MR4305371,MR3941887,MR3370415}) and references therein. So far the ability of these modified EM methods to approximate the solutions of super-linear diffusion systems has  been displayed comprehensively. 

Inspired by the references mentioned above, we construct an explicit multiscale numerical method to approximate super-linear SFSDEs and  obtain its strong convergence. We overcome two major obstacles:  the unknown  form  and  super-linear structure of  $\bar{b}(\cdot)$. Borrowing the HMM idea, our  multiscale  scheme  involves three subroutines as follows:
\begin{itemize}
	\item[1.]  Macro solver.  To avoid the excessive deviation caused by  the averaged coefficient  $\bar{b}(\cdot)$, a truncation mapping $T_{\Delta_1}: \mathbb{R}^{n_1}\rightarrow \mathbb{R}^{n_1}$ is proposed in \eqref{eqN3.2} to modify  $\bar{b}(\cdot)$ as  $\bar{b}(T_{\Delta_1}(\cdot))$.  Due to the unknown form of   $\bar{b}(\cdot)$, we  then propose an estimator $\tilde{b}(\cdot)$ to approximate it. Then  we  use the EM scheme to evolve  the  modified averaged equation described by
\begin{align*}
X_{n+1}=X_{n}+\tilde{b}(T_{\Delta_1}(X_{n}))\Delta_1+\sigma(X_{n})\Delta W^{1}_{n},
\end{align*}
where $\Delta_1$ is the macro time step size and $\Delta W^{1}_{n}=W^{1}((n+1)\Delta_1)-W^1(n\Delta_1)$. 
	\item[2.]  Micro solver. To produce the data for generating  $\tilde{b}(T_{\Delta_1}(X_{n}))$, we use the EM method to solve the equation \eqref{eq2.5} with frozen parameter $x=T_{\Delta_1}(X_{n})$ described by
\begin{equation*}
\begin{cases}
\!\!Y^{T_{\Delta_1}(X_{n}),y_0}_{m+1}\!=\!Y^{T_{\Delta_1}(X_{n}),y_0}_{m} \!+\!f\big(T_{\Delta_1}(X_n)\!,\!Y^{T_{\Delta_1}(X_n),y_0}_{m}\big)\Delta_2\!+\!g\big(T_{\Delta_1}(X_n)\!,\!Y^{T_{\Delta_1}(X_n),y_0}_{m}\big)\Delta W^{2}_{n,m},\\
\!\!Y^{T_{\Delta_1}(X_{n}),y_0}_{0}=y_0, ~~m=0,1,\cdot\cdot\cdot,
\end{cases}
\end{equation*}where $\{W^{2}_{n}(\cdot)\}_{n\geq 0}$ is a mutually independent Brownian motion sequence and also independent of $W^{1}(t)$, and $\Delta W^{2}_{n,m}=W^{2}_n((m+1)\Delta_{2})-W^{2}_n(m\Delta_{2})$.
	\item[3.]   Estimator. The  approximated coefficient $\tilde{b}(T_{\Delta_1}(X_n))$ can be given by the   time averaging 
\begin{align}\label{eqN1.8}
\tilde{b}(T_{\Delta_1}(X_n))=\frac{1}{M}\sum_{m=1}^{M}b\big(T_{\Delta_1}(X_{n}),Y^{T_{\Delta_1}(X_{n}),y_0}_{m}\big),
\end{align}
 where $M$ denotes the number of  micro time steps used for this approximation. 
\end{itemize}
Following this line,  an easily implementable multiscale truncated EM (MTEM) scheme described  by \eqref{e5} for a class of super-linear SFSDEs is established. By constructing the truncation mapping,  super-linear growth coefficients of the slow drift are modified.
 The corresponding estimator is obtained from the data generated by the  micro solver with the macro grid points as frozen parameters, which approximates to the underlying superlinear $\bar{b}(\cdot)$ in an efficient manner as $\Delta_1\rightarrow 0$.  Thus, this scheme avoids possible large excursion from the super-linearity of $\bar{b}(\cdot)$.

The main contribution of our paper is the construction of an explicit multiscale numerical scheme for a broad class of super-linear SFSDEs, along with a rigorous proof of its strong convergence. Leveraging the ergodicity theory of the exact  and numerical solutions of the  equation \eqref{eq2.5} with frozen slow component, we initially estimated the error between the averaged coefficient $\bar{b}(\cdot)$ and its estimator (see Lemma \ref{Lb.7}).  Employing this result, and incorporating the analysis technique of the stopping time, we demonstrate the strong convergence and its strong convergence rate between the exact solution of the averaged equation and the numerical solution generated by the MTEM scheme (see Theorems \ref{Lb.14} and \ref{L5.1}). Furthermore, by drawing on the conclusions of the averaging principle, which establishes the strong convergence  between the exact solutions of the slow component and averaged equation, the strong convergence  between the exact solution of the slow component and the numerical solution generated by MTEM scheme also can  be derived (see Theorem \ref{th2}).

The rest of this paper is organized as follows. Section \ref{s-c2} gives some notations, hypotheses and preliminaries. Section \ref{s-c3}  proposes an explicit multiscale numerical method. Section \ref{s-c4} provides some important pre-estimates. Section \ref{s-c5} yields the  strong convergence of MTEM scheme.  Section \ref{s-c6} focuses on the error analysis of the explicit MTEM scheme and presents an important example. Section \ref{s-c7} gives several numerical examples to validate our theoretical findings. Section \ref{s-c8} concludes this paper.

\section{Preliminary results}\label{s-c2}
Throughout this paper,  we  use the following  notation. Let $( \Omega,\cal{F},\mathbb{P})$ be a complete  probability space with a natural filtration $\{\mathcal{F}_{t}\}_{t\geq 0}$ satisfying the usual conditions (i.e. it is right continuous and increasing while $\mathcal{F}_0$ contains all $\mathbb{P}$-null sets), and $\mathbb{E}$ be the expectation corresponding to $\mathbb{P}$. Let $|\cdot|$ denote the Euclidean norm in $\RR^n$ and the trace norm in $\RR^{n\times d}.$
If $A$ is a vector or matrix, we denote its transpose by $A^T$.
For a set $\mathbb{D}$, let $I_{\mathbb{D}}(x)=1$ if $x\in \mathbb{D}$ and $0$ otherwise. We set $\inf\emptyset=\infty$, where $\emptyset$ is empty set. Moreover, for any $a,b\in \mathbb{R}$, we define $a\vee b=\max\{a,b\}$ and $a\wedge b=\min\{a,b\}$. We  use $C$  and $C_{l}$ to denote the generic positive constants,  which may take different values at different appearances, where the subscript  $l$  in $C_{l}$  is used to
 to highlight that this constant  depends on the parameter $l$. In addition, $C, C_{l}$ are independent of parameters $\Delta_1$, $\Delta_2$, $n$ and $M$ that
occur in the next section.  In particular, $C_{R}$ usually denotes some positive  function increasing  with respect to $R$.

 Let  $\mathcal{P}(\mathbb{R}^{n_2})$ denote the family  of  all probability measures on $\mathbb{R}^{n_2}$.
For any $p\geq1$, let $\mathcal{P}_{p}(\mathbb{R}^{n_2})$ be the set in  $\mathcal{P}(\mathbb{R}^{n_2})$  with
finite $p$-th moment, i.e.,
\begin{align*}
\mathcal{P}_{p}(\mathbb{R}^{n_2}):=\Big\{\mu\in \mathcal{P}(\mathbb{R}^{n_2}): \int_{\mathbb{R}^{n_2}}|y|^{p}\mu(\mathrm{d}y)<\infty\Big\},
\end{align*}
which is a Polish space under the Wasserstein distance
\begin{align*}
\mathbb{W}_{p}(\mu_1,\mu_2)=\inf_{\pi\in \mathcal{C}(\mu_1,\mu_2)}\Big(\int_{\mathbb{R}^{n_2}\times \mathbb{R}^{n_2}}|y_1-y_2|^{p}\pi(\mathrm{d}y_1,\mathrm{d}y_2)\Big)^{\frac{1}{p}},
\end{align*}
where $\mathcal{C}(\mu_1,\mu_2)$ stands for the set of all probability measures on $\mathbb{R}^{n_2}\times \mathbb{R}^{n_2}$ with marginals $\mu_1$ and
$\mu_2$, respectively.
\par To state the main  results, we first impose some hypotheses on the coefficients $b$ and  $\sigma$ of slow equation and  the coefficients $f$ and $g$  of the fast equation, denoted by $({\bf{S\cdot}})$ and $({\bf{F\cdot}})$, respectively.
\begin{itemize}
\item[({\bf{S1}})]There exists a  constant  $\theta_{1}\geq1$  such that for any $R>0$,  $x_{1}, x_2\in \mathbb{R}^{n_{1}}$ with $|x_{1}|\vee|x_2|\leq R$ { and $y\in \RR^{n_2}$,}
\begin{align}
|b(x_{1},y)-b(x_{2},y)|+|\sigma(x_{1})-\sigma(x_{2})|\leq L_{R}|x_{1}-x_{2}| (1+|y|^{\theta_{1}}),\nn\
\end{align}
where $L_{R}$ is a positive constant dependent on $R$.
\item[({\bf{S2}})]There exist constants {$\theta_2>0$} and $K_{1}> 0$ such that for any $x\in \mathbb{R}^{n_1}$ and $y_1,y_2\in \mathbb{R}^{n_2}$.
 \begin{align*}
|b(x,y_{1})-b(x,y_{2})|\leq K_{1}|y_{1}-y_{2}|\big(1+|x|^{\theta_2}
+|y_{1}|^{\theta_2}+|y_{2}|^{\theta_2}\big). 
\end{align*}
\item[({\bf{S3}})]There exists a constant $K_2>0$ such that for any $x\in \mathbb{R}^{n_1}$,
\begin{align*}
|\sigma(x)|\leq K_{2}(1+|x|).
\end{align*}
\item[({\bf{S4}})] There exist constants $\theta_3, \theta_4\geq1$ and $K_{3}>0$ such that for any $ x\in \mathbb{R}^{n_{1}}, y\in\mathbb{R}^{n_{2}}$,
\begin{align*}
&|b(x,y)|\leq K_{3}(1+|x|^{\theta_3}+|y|^{\theta_4}).
\end{align*}
\item[({\bf{S5}})] There exist constants $K_4>0$ and $\lambda>0$ such that for any $x\in \mathbb{R}^{n_1}, y\in \mathbb{R}^{n_2}$,
\begin{align}
x^{T}b(x,y)\leq K_4(1+|x|^2)+\lambda|y|^2.\nn\
\end{align}
\end{itemize}
\begin{itemize}
\item[({\bf F1})]The functions $f$ and $g$ are globally Lipschitz continuous, namely, for any $x_1,x_2\in \mathbb{R}^{n_1}$ and $y_1,y_2 \in \mathbb{R}^{n_2}$, there exists a  positive constant $L$ such that
\begin{align*}
|f(x_{1},y_{1})-f(x_{2},y_{2})|\vee|g(x_{1},y_{1})-g(x_{2},y_{2})|\leq L(|x_{1}-x_{2}|+|y_{1}-y_{2}|).
\end{align*}
\item[({\bf F2})] There exists a constant $\beta>0$ such that for any $ x\in \mathbb{R}^{n_1}$ and $y_{1}, y_{2}\in \mathbb{R}^{n_{2}} $,
\begin{align*}
2(y_{1}-y_{2})^{T}\big(f(x,y_{1})-f(x,y_{2})\big) +|g(x,y_{1})-g(x,y_{2})|^{2}
 \leq -\beta|y_{1}-y_{2}|^{2}.
\end{align*}
\item[$({\bf F3 })$] There exist constants $k\geq 2$, $\alpha_{k}>0$ and $L_k>0$ such that for any  $x\in \mathbb{R}^{n_1}$, $y\in \mathbb{R}^{n_2}$,
\begin{align*}
y^{T}f(x,y)+\frac{k-1}{2}|g(x,y)|^2\leq-\alpha_k|y|^{2}+L_k(1+|x|^2).
\end{align*}
\end{itemize}
\begin{remark}
 {\rm Referring to \cite[Theorem 2.2]{MR4047972}, system (\ref{e1}) admits a unique global solution $(x^{\varepsilon}(t),y^{\varepsilon}(t))$ under $({\bf S1})$-$({\bf S5})$ and $({\bf F1})$-$({\bf F3})$. Obviously, $({\bf{F1}})$ guarantees
  that the frozen equation \eqref{eq2.5}} has a unique global solution $y^{x,y_0}(s)$, which is a time homogeneous Markov process.
\end{remark}
\begin{lem}[{{\cite[Lemma 3.6]{MR4047972}}}]\label{LN3.6}
{\rm If $({\bf F3})$ hold, then for any $x\in \RR^{n_1}$, $y_0\in \RR^{n_2}$}, there exist  positive constants $\tilde{\alpha}_k$ and $C_k$ such that
\begin{align*}
\sup_{t\geq 0}\mathbb{E}|y^{x,y_0}(t)|^{k}\leq e^{-\tilde{\alpha}_{k}t}|y_0|^{k}+C_{k}(1+|x|^{k}).
\end{align*} 
\end{lem}
\begin{lem}[{{\cite[Lemma 3.7]{MR4047972}}}]\label{LN3.7}
{\rm If $({\bf F2})$ hold, then for any $x\in \RR^{n_1}$, $y_1, y_2\in \RR^{n_2}$}, 
\begin{align*}
\sup_{t\geq 0}\mathbb{E}|y^{x,y_1}(t)-y^{x,y_2}(t)|^{2}\leq e^{-\beta t}|y_1-y_2|^2.
\end{align*} 
\end{lem}

By a  argument similar to that of \cite[Lemma 3.10]{MR4047972},  the upper bound of $\mathbb{E}|y^{x_1,y_0}-y^{x_2,y_0}|^2$ for any $x_1,x_2\in \RR^{n_1}$ is given as follows.
\begin{lem}\label{LN2.3}
{\rm If $({\bf F1})$-$({\bf F3})$ hold, then for any $x_1, x_2 \in \RR^{n_1}$}, $y_0\in \RR^{n_2}$,
\begin{align*}
\sup_{t\geq 0}\mathbb{E}|y^{x_1,y_0}(t)-y^{x_2,y_0}(t)|^2\leq C|x_1-x_2|^2.
\end{align*}
\end{lem}
\begin{proof}\noindent\textbf {Proof.}
Note that for any $t\geq 0$,
\begin{align*}
y^{x_1,y_0}(t)-y^{x_2,y_0}(t)&=\int_{0}^{t}\big(f(x_1,y^{x_1,y_0}(s))-f(x_2,y^{x_2,y_0}(s))\big)\mathrm{d}s\nn\
\\&~~~+\int_{0}^{t}\big(g(x_1,y^{x_1,y_0}(s))-g(x_2,y^{x_2,y_0}(s))\big)\mathrm{d}W^{2}(s).
\end{align*} 
Then using the It\^o formula and {\bf (F2)} one derives that
\begin{align}\label{NN2.1}
\mathbb{E}|y^{x_1,y_0}(t)-y^{x_2,y_0}(t)|^2&= \mathbb{E}\int_{0}^{t}2(y^{x_1,y_0}(s)-y^{x_2,y_0}(s))^{T}\big(f(x_1,y^{x_1,y_0}(s))-f(x_2,y^{x_2,y_0}(s))\big)\nn\
\\&~~~+\big|g(x_1,y^{x_1,y_0}(s))-g(x_2,y^{x_2,y_0}(s))\big|^2\mathrm{d}s\nn\
\\&= \mathbb{E}\int_{0}^{t}2(y^{x_1,y_0}(s)-y^{x_2,y_0}(s))^{T}\big(f(x_1,y^{x_1,y_0}(s))-f(x_1,y^{x_2,y_0}(s))\big)\nn\
\\&~~~+\big|g(x_1,y^{x_1,y_0}(s))-g(x_1,y^{x_2,y_0}(s))\big|^2\mathrm{d}s+\mathcal{C}\nn\
\\&\leq -\beta\int_{0}^{t}\mathbb{E}|y^{x_1,y_0}(s)-y^{x_2,y_0}(s)|^2+\mathcal{C},
\end{align}
where \begin{align*}
\mathcal{C}&=2\mathbb{E}\int_{0}^{t}(y^{x_1,y_0}(s)-y^{x_2,y_0}(s))^{T}\big(f(x_1,y^{x_2,y_0}(s))-f(x_2,y^{x_2,y_0}(s))\big)\mathrm{d}s\nn\
\\&~~~+\mathbb{E}\int_{0}^{t}\big|g(x_1,y^{x_2,y_0}(s))-g(x_2,y^{x_2,y_0}(s))\big|^2\mathrm{d}s\nn\
\\&~~~+ \mathbb{E}\int_{0}^{t}2\big|g(x_1,y^{x_1,y_0}(s))-g(x_1,y^{x_2,y_0}(s))\big|\big|g(x_1,y^{x_2,y_0}(s))-g(x_2,y^{x_2,y_0}(s))\big|\mathrm{d}s.
\end{align*}
Then using {\bf (F1)} implies that
\begin{align*}
\mathcal{C}\leq (2L+2L^2)\mathbb{E}\int_ {0}^{t}|y^{x_1,y_0}(s)-y^{x_2,y_0}(s)||x_1-x_2|\mathrm{d}s+L^2|x_1-x_2|^2t.
\end{align*}
Furthermore, employing the Young inequality one obtains that 
\begin{align*}
\mathcal{C}\leq \frac{\beta}{2}\int_{0}^{t}\mathbb{E}|y^{x_1,y_0}(s)-y^{x_2,y_0}(s)|^2\mathrm{d}s+C|x_1-x_2|^2t.
\end{align*}
Inserting the above inequality into \eqref{NN2.1} shows that
\begin{align*}
\mathbb{E}|y^{x_1,y_0}(t)-y^{x_2,y_0}(t)|^2&\leq -\frac{\beta}{2}\int_{0}^{t}\mathbb{E}|y^{x_1,y_0}(s)-y^{x_2,y_0}(s)|^2\mathrm{d}s+Ct|x_1-x_2|^2,
\end{align*}
which implies the desired result.
\end{proof}

Building on Lemmas \ref{LN3.6}-\ref{LN2.3}, using the method of synchronous coupling we establish the existence and uniqueness of the invariant probability measure $\mu^{x}$ for equation \eqref{eq2.5}. Furthermore, we  estimate the deviation of  invariant probability  measures $\mu^{x_1}$ and $\mu^{x_2}$ for equation \eqref{eq2.5} with different frozen parameters $x_1,x_2 \in \RR^{n_1}$. For the completeness of this paper, we present the conclusions and the corresponding proofs needed in our paper.
\begin{lem}\label{Lcyp2.1}
{\rm If $({\bf F1})$-$({\bf F3})$ hold  {with some $k\geq2$}, then for any fixed $x\in\RR^{n_1}$,  the  transition {semigroup $\{\mathbb{P}^{x}_{t}\}_{t\geq0}$} of equation \eqref{eq2.5} has a unique invariant probability measure $\mu^{x}\in \mathcal{P}_{k}(\mathbb{R}^{n_2})$, {which satisfies that}
\begin{align}\label{cyp2.4}
\int_{\mathbb{R}^{n_2}}|y|^{k}\mu^{x}(\mathrm{d}y)\leq C(1+|x|^{k}).
\end{align} Furthermore, for any $x_1,x_2\in \mathbb{R}^{n_1}$,
\begin{align}\label{cyp2.5}
\mathbb{W}_{2}(\mu^{x_1},\mu^{x_2})\leq C|x_1-x_2|.
\end{align}}
\end{lem}
\begin{proof}\textbf {Proof.}
For any fixed $x\in \mathbb{R}^{n_1}$ and {$y_0\in \mathbb{R}^{n_2}$}, under $({\bf F3})$,  it follows from the  result of Lemma \ref{LN3.6} that
\begin{align*}
\sup_{t\geq 0}\mathbb{E}|y^{x,y_0}(t)|^{k}\leq C_{y_0}(1+|x|^{k}),
 \end{align*}
 which implies that $\delta_{y_0}\mathbb{P}^{x}_{t}\in \mathcal{P}_{k}(\mathbb{R}^{n_2})\subset  \mathcal{P}_{2}(\mathbb{R}^{n_2})$. It is well known  that for any $\mu_1,\mu_2\in \mathcal{P}_{2}(\mathbb{R}^{n_2})$
\begin{align*}
\mathbb{W}_{2}(\mu_1\mathbb{P}^{x}_{t},\mu_2\mathbb{P}^{x}_{t})&\leq\int_{\mathbb{R}^{n_2}\times \mathbb{R}^{n_2}}\mathbb{W}_{2}(\delta_{y_1}\mathbb{P}^{x}_{t},\delta_{y_2}\mathbb{P}^{x}_{t})\pi(\mathrm{d}y_1,\mathrm{d}y_2)
\\& \leq \int_{\mathbb{R}^{n_2}\times \mathbb{R}^{n_2}} \big(\mathbb{E}|y^{x,y_1}(t)-y^{x,y_2}(t)|^2\big)^{\frac{1}{2}}\pi(\mathrm{d}y_1,\mathrm{d}y_2),
\end{align*}
where $\pi\in \mathcal{C}(\mu_1,\mu_2)$, and $\delta_{y_0}$ is the Dirac measure with mass at point $y_0\in \RR^{n_2}$. Then under ({\bf F2}), using Lemma \ref{LN3.7}, one deduces that
\begin{align*}
\mathbb{W}_{2}(\mu_1\mathbb{P}^{x}_{t},\mu_2\mathbb{P}^{x}_{t})&\leq Ce^{-\frac{\beta t}{2}}\int_{\mathbb{R}^{n_2}\times \mathbb{R}^{n_2}}|y_1-y_2|\pi(\mathrm{d}y_1,\mathrm{d}y_2)\nn\
\\&\leq Ce^{-\frac{\beta t}{2} }\Big(\int_{\mathbb{R}^{n_2}\times \mathbb{R}^{n_2}}|y_1-y_2|^2\pi(\mathrm{d}y_1,\mathrm{d}y_2)\Big)^{\frac{1}{2}}.\nn\
\end{align*}
Since $\pi$ is arbitrary, we have
\begin{align*}
\mathbb{W}_{2}(\mu_1\mathbb{P}^{x}_{t},\mu_2\mathbb{P}^{x}_{t})\leq Ce^{-\frac{\beta t}{2} }\mathbb{W}_{2}(\mu_1,\mu_2),
\end{align*}
which yields the uniqueness of invariant probability measure if it exists. Next  we shall prove the existence of invariant probability measure. In fact, it is sufficient to prove that for any fixed $x\in \mathbb{R}^{n_1}$ and $y_0\in \mathbb{R}^{n_2}$, $\{\delta_{y_0} \mathbb{P}^{x}_{t}\}_{t\geq 0}$ is a $\mathbb{W}_2$-Cauchy family due to the completeness of $\mathcal{P}_{2}(\mathbb{R}^{n_2})$ space. Using the Kolmogorov-Chapman equation, Lemmas \ref{LN3.6} and \ref{LN3.7}, one derives that for any $t, s>0$,
\begin{align*}
\mathbb{W}_{2}(\delta_{y_0}\mathbb{P}^{x}_{t},\delta_{y_0}\mathbb{P}^{x}_{t+s})&= \mathbb{W}_{2}(\delta_{y_0}\mathbb{P}^{x}_{t},\delta_{y_0} \mathbb{P}^{x}_{s}\mathbb{P}^{x}_{t})
\leq Ce^{-\frac{\beta t}{2} }\mathbb{W}_{2}(\delta_{y_0},\delta_{y_0}\mathbb{P}^{x}_{s})\nn\
\\&\leq Ce^{-\frac{\beta t}{2} }\big(|y_0|^{2}+\mathbb{E}|y^{x,y_0}(s)|^{2}\big)^{\frac{1}{2}}
\leq Ce^{-\frac{\beta s}{2} }(1+|x|+|y_0|),
\end{align*}
which implies that as $t\rightarrow \infty$, $\{\delta_{y_0} \mathbb{P}^{x}_{t}\}_{t\geq 0}$ is a $\mathbb{W}_2$-Cauchy family whose limit is denoted by $\mu^{x}$. Furthermore, in view of the continuity of $\mathbb{W}_{2}$-distance (see, \cite[Corollary 6.1]{Villani}) we derive for any $t>0$
\begin{align*}
\mathbb{W}_{2}(\mu^{x}\mathbb{P}^{x}_{t},\mu^{x})=\lim_{s\rightarrow\infty}\mathbb{W}_{2}(\delta_{y_0}\mathbb{P}^{x}_{s+t},\delta_{y_0}\mathbb{P}^{x}_{s})=0,
\end{align*}
which implies that $\mu^{x}$ is indeed an invariant probability measure of $y^{x}(s)$. Furthermore, applying the invariance  of $\mu^{x}$ and Lemma \ref{LN3.6} yields that
		\begin{align*}
			\int_{\mathbb{R}^{n_2}}(|y|^{k}\wedge N)\mu^{x,\Delta_2}(\mathrm{d}y)&=\int_{\mathbb{R}^{n_2}}\mathbb{E}(|Y^{x,y}_{n,m}|^{k}\wedge N)\mu^{x,\Delta_2}(\mathrm{d}y)\notag
			\\&\leq \int_{\mathbb{R}^{n_2}}(\mathbb{E}|Y^{x,y}_{n,m}|^{k}\wedge N)\mu^{x,\Delta_2}(\mathrm{d}y)\notag
			\\&\leq\int_{\mathbb{R}^{n_2}}\Big(|y|^{k} e^{-\tilde{\alpha}_{k}t}\wedge N\Big)\mu^{x,\Delta_2}(\mathrm{d}y)+C(1+|x|^{k}),
		\end{align*}
		where  the first inequality holds  by Jensen's inequality since $x\mapsto N\wedge x,x\in \mathbb{R}$ is a convex  function. Then, taking $t\rightarrow\infty$ and using the dominated convergence theorem imply that
		\begin{align*}
			\int_{\mathbb{R}^{n_2}}(|y|^{k}\wedge N)\mu^{x,\Delta_2}(\mathrm{d}y)\leq C(1+|x|^k).
		\end{align*}
		Furthermore, letting $N\rightarrow\infty$  and applying the monotone convergence theorem, one gets
		\begin{align*}
			\int_{\mathbb{R}^{n_2}}|y|^{k}\mu^{x,\Delta_2}(\mathrm{d}y)\leq C(1+|x|^k).
		\end{align*}
In addition, using the continuity of $\mathbb{W}_{2}$ again yields that
\begin{align*}
\mathbb{W}^2_{2}(\mu^{x_1},\mu^{x_2})&=\lim_{t\rightarrow\infty}\mathbb{W}^2_{2}(\delta_{y_0}\mathbb{P}^{x_1}_{t},\delta_{y_0}\mathbb{P}^{x_2}_{t})\nn\
\\&\leq \lim_{t\rightarrow\infty}\mathbb{E}|y^{x_1,y_0}(t)-y^{x_2,y_0}(t)|^{2}\leq C|x_1-x_2|^2,
\end{align*}
where the last step follows from the  result of  Lemma \ref{LN2.3}. The proof is complete.
\end{proof}
\par  The averaging principle substantially reduces the complexity of the original system \eqref{e1}, which becomes our theoretic base to develop the explicit multiscale numerical method. To facilitate this, we cite some known results on the averaging principle  and the averaged equation.
\begin{lem}[{{\cite[Theorem 2.3]{MR4047972}}}]\label{L1}
{\rm If $({\bf S1})$-$({\bf S5})$ and $({\bf F1})$-$({\bf F3})$ hold with $k> 4\theta_{1}\vee2(\theta_2+1)\vee2\theta_3\vee2\theta_4$, then for any $0<p<k$ and $T>0$,
\begin{align*}
\lim\limits_{\varepsilon\rightarrow 0}\mathbb{E}\Big(\sup\limits_{t\in[0, T]}|x^{\varepsilon}(t)-\bar{x}(t)|^{p}\Big)=0,
\end{align*}
where $x^{\varepsilon}(t)$ and $\bar{x}(t)$ are the solutions of \eqref{e1} and \eqref{eq2.4}, respectively.}
\end{lem}
\par By virtue of Lemma \ref{Lcyp2.1}, we show that  the drift term $\bar{b}(\cdot)$ of the averaged equation \eqref{eq2.4} inherits the locally Lipschitz continuity like ({\bf{S1}}).
\begin{lem}\label{L3.3}
	{\rm Under $({\bf S1})$, $({\bf S2})$, $({\bf S4})$ and $({\bf F1})$-$({\bf F3})$ with $k\geq{2}\vee\theta_1\vee2\theta_2\vee\theta_4$,
		for any $R>0$, there exists a constant $\bar{L}_{R}$ such that for any $x_1, x_2\in \mathbb{R}^{n_1}$ with $|x_1|\vee|x_2|\leq R$,
		\begin{align*}
			|\bar{b}(x_1)-\bar{b}(x_2)|\leq \bar{L}_{R}|x_1-x_2|.
	\end{align*}}
\end{lem}
\begin{proof}\noindent\textbf {Proof.}
Under $({\bf S4})$ and  $({\bf F1})$-$({\bf F3})$ with $k\geq \theta_4$, it follows from \eqref{f11}  that 
\begin{align}\label{eqN2.4}
|\bar{b}(x)|\leq\int_{\RR^{n_2}} |b(x,y)|\mu^{x}(\mathrm{d}y)\leq K_3\int_{\RR^{n_2}}(1+|x|^{\theta_3}+|y|^{\theta_4})\mu^{x}(\mathrm{d}y)<\infty,~~~\forall x\in \RR^{n_1}.
\end{align}
	For any $x_1,x_2\in \mathbb{R}^{n_1}$, according to \eqref{f11} we have
	\begin{align*}
		|\bar{b}(x_1)-\bar{b}(x_2)|&=\Big|\int_{\mathbb{R}^{n_1}\times\mathbb{R}^{n_2}}(b(x_1,y_1)-b(x_2,y_2))\pi(\mathrm{d}y_1,\mathrm{d}y_2)\Big|\notag
		\\&\leq\int_{\mathbb{R}^{n_1}\times\mathbb{R}^{n_2}}\big|b(x_1,y_1)-b(x_2,y_2)|\pi(\mathrm{d}y_1,\mathrm{d}y_2)\notag
		\\&\leq \int_{\mathbb{R}^{n_1}\times\mathbb{R}^{n_2}}\big|b(x_1,y_1)-b(x_2,y_1)|\pi(\mathrm{d}y_1,\mathrm{d}y_2)\notag
		\\&~~~+\int_{\mathbb{R}^{n_1}\times\mathbb{R}^{n_2}}\big|b(x_2,y_1)-b(x_2,y_2)|\pi(\mathrm{d}y_1,\mathrm{d}y_2),
	\end{align*}
	{where} $\pi \in C(\mu^{x_1},\mu^{x_2})$ is arbitrary. Then for any $R>0$  and  $x_1,x_2\in \mathbb{R}^{n_1}$ with $|x_1|\vee|x_2|\leq R$, by the H\"older inequality it follows from $({\bf S1})$ and $({\bf S2})$ that
	\begin{align*}
		|\bar{b}(x_1)-\bar{b}(x_2)|&\leq L_{R}|x_{1}-x_{2}|\int_{\mathbb{R}^{n_2}} (1+|y_1|^{\theta_{1}})\mu^{x_1}(\mathrm{d}y_1)\notag
		\\&~~~+K_{1}\int_{\mathbb{R}^{n_1}\times \mathbb{R}^{n_2}}|y_{1}-y_{2}|\big(1+|x_2|^{\theta_2}
		+|y_{1}|^{\theta_2}+|y_{2}|^{\theta_2}\big)\pi(\mathrm{d}y_1,\mathrm{d}y_2)\notag
		\\& \leq L_{R}|x_{1}-x_{2}|\int_{\mathbb{R}^{n_2}} (1+|y_1|^{\theta_{1}})\mu^{x_1}(\mathrm{d}y_1)\notag
		\\&~~~+K_{1}\Big(\int_{\mathbb{R}^{n_1}\times \mathbb{R}^{n_2}}|y_{1}-y_{2}|^2\pi(\mathrm{d}y_1,\mathrm{d}y_2)\Big)^{\frac{1}{2}}\notag
		\\&~~~~~~~~~\times\Big(\int_{\mathbb{R}^{n_1}\times \mathbb{R}^{n_2}}\big(1+|x_2|^{2\theta_2}
		+|y_{1}|^{2\theta_2}+|y_{2}|^{2\theta_2}\big)\pi(\mathrm{d}y_1,\mathrm{d}y_2)\Big)^{\frac{1}{2}}.
	\end{align*}
	Since $\pi$ is arbitrary, under  $({\bf F1})$-$({\bf F3})$ with $k\geq\theta_1\vee2\theta_2$,   applying Lemma \ref{Lcyp2.1} yields that
	\begin{align*}
		|\bar{b}(x_1)-\bar{b}(x_2)|&\leq C_{R}|x_{1}-x_{2}|(1+|x_1|^{\theta_1})+C\mathbb{W}_{2}(\mu^{x_1},\mu^{x_2})(1+|x_1|^{\theta_2}+|x_2|^{\theta_2})\notag
		\\& \leq C_{R}|x_{1}-x_{2}|+C_{R}\mathbb{W}_{2}(\mu^{x_1},\mu^{x_2})\leq C_{R}|x_{1}-x_{2}|
	\end{align*}
for any $x_1,x_2\in \mathbb{R}^{n_1}$ with $|x_1|\vee|x_2|\leq R$, which implies the desired result.
\end{proof}
\par   By virtue of Lemma 2.1 we yield that the averaged coefficient $\bar{b}$  satisfies the Khasminskii-like condition, similarly to  ({\bf{S5}}),
which implies that  the solution $\bar{x}(t)$ of the averaged equation  {\color{red}has bounded  moments}. 
To avoid duplication we omit the proof. 
\begin{lem}\label{le3.1}
	{\rm If $({\bf S4})$, $({\bf S5})$ and $({\bf F1})$-$({\bf F3})$ hold with $k\geq{2}\vee\theta_4$, then 
		\begin{align*}
			x^{T}\bar{b}(x)\leq C(1+|x|^{2}),~~~~x\in \mathbb{R}^{n_1}.
	\end{align*}}
\end{lem}
{\color{red}
\begin{proof}\noindent\textbf {Proof.}
Under $({\bf S4})$ and  $({\bf F1})$-$({\bf F3})$ with $k\geq \theta_4$,
it follows from  \eqref{eqN2.4}  that $\bar{b}(x)$ is well defined for any $x\in \mathbb{R}^{n_{1}}$. Then by Assumption ({\bf{S5}}) one obtains that
	\begin{align*}
		x^{T}\bar{b}(x)&= x^{T}\int_{\mathbb{R}^{n_{2}}} b(x,y)\mu^{x}(\mathrm{d}y)\notag
		\leq K_4(1+|x|^{2})+ \lambda\int_{\mathbb{R}^{n_{2}}}|y|^{2}\mu^{x}(\mathrm{d}y).
	\end{align*}
	Thus, applying  Lemma \ref{Lcyp2.1} yields
	\begin{align*}
		x^{T}\bar{b}(x)\leq C(1+|x|^{2}),~~~~x\in \mathbb{R}^{n_1}. 
	\end{align*}
The proof is complete.
\end{proof}}
\begin{lem}[{{\cite[Lemma 3.11]{MR4047972}}}]\label{L2}
{\rm If $({\bf S1})$-$({\bf S5})$ and $({\bf F1})$-$({\bf F3})$ hold with $k\geq {2}\vee\theta_{1}\vee2\theta_2\vee\theta_4$,
then for any $x_0\in \RR^{n_1}$, the averaged equation (\ref{eq2.4}) has a unique global solution $\bar{x}(t)$ satisfying
\begin{align*}
\mathbb{E}\Big(\sup\limits_{0\leq t\leq T}|\bar{x}(t)|^{p}\Big)\leq {C_{x_0,T,p}},~~~~\forall ~p>0,~T>0.
\end{align*}}
\end{lem}

\section{The construction of  explicit multiscale scheme}\label{s-c3}

With the help of the strong averaging principle,  this section is devoted to constructing an easily implementable multiscale numerical scheme for the slow component of original SFSDE \eqref{e1}.  One notices from $({\bf S4})$ that for any $y\in \mathbb{R}^{n_2}$,
\begin{align*}
	|b(x,y)|&\leq K_3(1+|x|)(1+|x|^{\theta_{3}-1})+K_3|y|^{\theta_4}.
\end{align*}
Then for any $u\geq1$ and $x\in \RR^{n_1}$ with $|x|\leq u$, one has
\begin{align}\label{a15}
	|b(x,y)|\leq K_3\varphi(u)(1+|x|)+K_3|y|^{\theta_4},~~\forall y\in \RR^{n_2},
\end{align}
where $\varphi(u)=1+u^{(\theta_3\vee\theta_4-1)}$, and $\theta_3,\theta_4$ are given in $({\bf S4})$. Thus, $\varphi^{-1}(u)=(u-1)^{\frac{1}{\theta_3\vee\theta_4-1}}, u\geq 1$.
Then for any step size $\Delta_{1}\in(0,1]$,  define 
\begin{align}\label{eqN3.2}
	T_{\Delta_1}(x):=\Big(|x|\wedge \varphi^{-1}\big(K\Delta_1^{-\frac{1}{2}}\big)\Big)\frac{x}{|x|},~~~~~~x\in\mathbb{R}^{n_1},
\end{align}
where $x/|x|={\bf 0}\in \mathbb{R}^{n_1}$ if $x={\bf 0}$, and $K$
is a constant satisfying $K\geq 1+{\varphi(|x_0|)}$.  
Combining \eqref{a15} and \eqref{eqN3.2} implies that for any $x\in \mathbb{R}^{n_1}$,
\begin{align}\label{cyp3.3}
	|b(T_{\Delta_1}(x),y)| \leq C\Delta_1^{-\frac{1}{2}}(1+|T_{\Delta_1}(x)|)+K_3|y|^{\theta_4}.
\end{align}
Moreover, under ({\bf S4}), ({\bf F1})-({\bf F3}) with $k\geq{2\vee\theta_4}$ by the definition \eqref{f11} we derive from the above inequality and \eqref{cyp2.4} that
\begin{align}\label{c3.12} |\bar{b}(T_{\Delta_1}(x))|&=\Big|\int_{\mathbb{R}^{n_{2}}}b(T_{\Delta_1}(x),y)
	\mu^{T_{\Delta_1}(x)}(\mathrm{d}y)\Big|\leq  \int_{\mathbb{R}^{n_{2}}}|b(T_{\Delta_1}(x),y)|
	\mu^{T_{\Delta_1}(x)}(\mathrm{d}y)\notag
	\\&\leq C\Delta_1^{-\frac{1}{2}}(1+|T_{\Delta_1}(x)|)+K_3\int_{\mathbb{R}^{n_2}}|y|^{\theta_4}\mu^{T_{\Delta_1}(x)}(\mathrm{d}y)\notag
	\\&\leq C\Delta_1^{-\frac{1}{2}}(1+|T_{\Delta_1}(x)|)+C(1+|T_{\Delta_1}(x)|^{\theta_4-1})(1+|T_{\Delta_1}(x)|)\notag
	\\&\leq C\Delta_1^{-\frac{1}{2}}(1+|T_{\Delta_1}(x)|) ,~~~~x\in \RR^{n_1},
\end{align}
where $\mu^{T_{\Delta_1}(x)}$ is the unique invariant probability measure of the equation \eqref{eq2.5} with frozen parameter $T_{\Delta_1}(x)$, and the last step follows from $\varphi$ being increasing.
{\rm \begin{rem}
Obviously,  the truncation mapping $T_{\Delta_1}$ depends not only on $\Delta_1$  but also  $K$ and  initial value $(x_0,y_0)$. But for short, we omit  $K$ and $(x_0,y_0)$  from the notation $T_{\Delta_1}$. 
\end{rem} }
\par Because the close form of $\bar{b}(x)$ is not known in general, we need construct an estimator $\tilde{b}(x)$
to approximate it. Using the ergodicity of the equation \eqref{eq2.5}, we  construct estimator  $\tilde{b}(x)$ by the time average of $b(x,\cdot)$ with respect to the numerical solution of the  equation \eqref{eq2.5} with the frozen parameter $x$. For convenience, for an integer $M>0$,  define
\begin{align}\label{3.9}
	B_{M}(x,h)=\frac{1}{M}\sum_{m=1}^{M}b(x,h_m), ~~~\forall x\in \RR^{n_1},~
\end{align}
where $h=\{h_m\}_{m=1}^{\infty}$ is an $\RR^{n_2}$-valued sequence.
Within the framework of HMM, we design an easily implementable  multiscale numerical scheme involving a macro solver and  a micro solver as well as  an estimator. For clarity, we describe it as follows.
Let $\Delta_1$ and $\Delta_2$ denote macro time step size and micro time  step size, respectively.
\begin{itemize}
	\item[(1)]
	Macro solver:
	For the known $X_{n}$,  since the drift coefficient $\bar{b}(x)$ of the averaged equation may be super-linear, the truncated EM scheme \cite{MR3370415}  is selected as macro solver to evolve the averaged equation \eqref{eq2.4} as follows:
	\begin{equation*}
		X_{n+1}=X_{n}+\tilde{b}(T_{\Delta_1}(X_{n}))\Delta_1+\sigma(X_{n})\Delta W^{1}_{n},
	\end{equation*}
	where $\tilde{b}(T_{\Delta_1}(X_{n}))$  given in  \eqref{NNq3.6} is an  estimator of the truncated coefficient $\bar{b}(T_{\Delta_1}(X_{n}))$ 
	and $\Delta W^{1}_{n}=W^{1}((n+1)\Delta_1)-W^{1}(n\Delta_1)$.
	\item[(2)] Micro solver: To obtain the  approximation data  of constructing estimator $\tilde{b}(T_{\Delta_1}(X_{n}))$ at each macro time step, for the known $X_{n}\in \RR^{n_1}$, use the EM method to solve the equation \eqref{eq2.5} with frozen parameter $x=T_{\Delta_1}(X_{n})$.  Therefore,  the micro solver is given by
\begin{equation*}
\begin{cases}
\!\!Y^{T_{\Delta_1}(X_{n}),y_0}_{m+1}\!=\!Y^{T_{\Delta_1}(X_{n}),y_0}_{m} \!+\!f\big(T_{\Delta_1}(X_n)\!,\!Y^{T_{\Delta_1}(X_n),y_0}_{m}\big)\Delta_2\!+\!g\big(T_{\Delta_1}(X_n)\!,\!Y^{T_{\Delta_1}(X_n),y_0}_{m}\big)\Delta W^{2}_{n,m},\\
\!\!Y^{T_{\Delta_1}(X_{n}),y_0}_{0}=y_0, ~~m=0,1,\cdot\cdot\cdot,
\end{cases}
\end{equation*}
	where $\{W^{2}_{n}(\cdot)\}_{n\geq 0}$ is a mutually independent Brownian motion sequence and also independent of $W^{1}(t)$, and $\Delta W^{2}_{n,m}=W^{2}_n((m+1)\Delta_{2})-W^{2}_n(m\Delta_{2})$.
	
	\item[(3)] Estimator: For the known $X_{n}$ and $Y^{T_{\Delta_1}(X_{n}),y_0}:= \big\{ Y^{T_{\Delta_1}(X_n),y_0}_{m}\big\}_{m=1}^{\infty}$, define  	
\begin{align}\label{NNq3.6}
\tilde{b}(T_{\Delta_1}(X_{n}))=B_{M}(T_{\Delta_1}(X_{n}),Y^{T_{\Delta_1}(X_{n}),y_0})
\end{align} as the {\color{red}estimator} of $\bar{b}(T_{\Delta_1}(X_{n}))$, where $B_{M}(\cdot,\cdot)$ is defined by \eqref{3.9} and $M$ denotes the number of  micro time steps used for this approximation.
\end{itemize}
\begin{rem}
In describing the construction process of the multiscale scheme, we introduce the symbol $\tilde{b}$ to represent the estimator of $\bar{b}$. Furthermore, we define the specific form of the estimator of $\bar{b}(T_{\Delta_1}(X_n))$ as $\tilde{b}(T_{\Delta_1}(X_{n}))=B_{M}(T_{\Delta_1}(X_{n}),Y^{T_{\Delta_1}(X_{n}),y_0})$.
 Since symbol $B_{M}(T_{\Delta_1}(X_{n}),Y^{T_{\Delta_1}(X_{n}),y_0})$ contains more comprehensive estimation information, we directly use it
to denote the estimator of $\bar{b}(T_{\Delta_1}(X_{n}))$ from this point onward.
  \end{rem}
  
Overall, for any given $\Delta_1, \Delta_2\in (0,1]$ and integer $M\geq1$,  define multiscale TEM (MTEM) scheme as follows: for any $n\geq0$,
\begin{subequations}\label{e5}
\begin{numcases}{}
X_{0}=x_{0},~T_{\Delta_1}(X_{n})=\Big(|X_{n}|\wedge \varphi^{-1}\big(K\Delta_1^{-\frac{1}{2}}\big)\Big)\frac{X_{n}}{|X_{n}|},~Y^{T_{\Delta_1}(X_{n}),y_0}_{0}=y_{0},\\
Y^{T_{\Delta_1}(X_{n}),y_0}_{m+1}\!=\!Y^{T_{\Delta_1}(X_{n}),y_0}_{m}+f\big(T_{\Delta_1}(X_{n}),Y^{T_{\Delta_1}(X_{n}),y_0}_{m}\big)\Delta_{2}\nn\
\\~~~~~~~~~~~~~~~~~+g\big(T_{\Delta_1}(X_{n}),Y^{T_{\Delta_1}(X_{n}),y_0}_{m}\big)\Delta W^{2}_{n,m},~~~m=0,1,\cdot\cdot \cdot,M-1,\label{e6}\\
X_{n+1}=X_{n}+ B_{M}(T_{\Delta_1}(X_{n}),Y^{T_{\Delta_1}(X_{n}),y_0})\Delta_{1}+\sigma(X_{n})\Delta W^{1}_{n}. \label{e8}
\end{numcases}
\end{subequations}
By this scheme we define the  continuous-time approximation processes:
\begin{align*}
	X(t)&=X_{n},~~~~~~~ t\in [n\Delta_1,(n+1)\Delta_1),
\end{align*}
\begin{align}\label{a22} \bar{X}(t)&=x_{0}+\int_{0}^{t} B_{M}(T_{\Delta_1}(X(s)),Y^{T_{\Delta_1}(X(s)),y_0})\mathrm{d}s
	+\int_{0}^{t}\sigma(X(s))\mathrm{d}W^{1}(s).
\end{align}
Note that $\bar{X}(n\Delta_1)= {X}(n\Delta_1)=X_{n}$,
that is, $\bar{X}(t)$ and $X(t)$ coincide with the discrete
solution at the grid points, respectively.
\section{Some preliminary estimates}\label{s-c4}

In order to prove  the strong convergence of the MTEM scheme, we need to give some properties for the estimator $B_{M}(x,Y^{x,y_0}_{n})$ (where $Y^{x,y_0}_{n}$ defined in \eqref{ee13} later) of the averaged coefficient $\bar{b}(x)$. This section pays attention to  some pre-estimates for $B_{M}(x,Y^{x,y_0}_{n})$.

\par{For any fixed $x\in \mathbb{R}^{n_1}$, $y_0\in \RR^{n_2}$,  and integer $n\geq 0$, define an auxiliary process $y^{x,y_0}_{n}(t)$ described by
	\begin{align}\label{e16}
		\mathrm{d}y^{x,y_0}_{n}(t)=f(x,y^{x,y_0}_{n}(t))
		\mathrm{d}t+g(x,y^{x,y_0}_{n}(t))\mathrm{d}W^{2}_{n}(t)
	\end{align}
	on $t\geq0$ with initial value $y^{x,y_0}_{n}(0)=y_0$. } Thanks to the weak uniqueness of the solution  of the equation \eqref{eq2.5},    for any $t\geq 0$,  the distribution of $y^{x,y_0}_{n}(t)$ coincides with that of $y^{x,y_0}(t)$ for any $n\geq0$. Consequently, according to Lemma \ref{Lcyp2.1}, $\mu^{x}$ is also the unique invariant probability measure of transition semigroup of $y^{x,y_0}_{n}(t)$ for any $ n\geq0$.
Then use the  EM scheme for \eqref{e16}
\begin{equation}\label{ee13}
	\begin{cases}
		Y^{x,y_0}_{n,0}=y_0,~~~~~~~~ \\
		Y^{x,y_0}_{n,m+1}=Y^{x,y_0}_{n,m}+
		f(x,Y^{x,y_0}_{n,m})\Delta_2+g(x,Y^{x,y_0}_{n,m})\Delta W^{2}_{n,m}, ~m=0, 1,\cdot\cdot\cdot.\\
	\end{cases}
\end{equation}
Furthermore, define
\begin{align}\label{a3.30}
	Y^{x,y_0}_{n}(t) &=Y^{x,y_0}_{n,m},~~~~~~~~t\in[m\Delta_2,(m+1)\Delta_2),\notag\\
	\bar{Y}^{x,y_0}_{n}(t)&=y_{0}+\int_{0}^{t}f(x,Y^{x,y_0}_{n}(s))\mathrm{d}s+\int_{0}^{t}
	g(x,Y^{x,y_0}_{n}(s))\mathrm{d}W^{2}_{n}(s).
\end{align}
Let  $Y^{x,y_0}_{n}$  denote the discrete EM solution  sequence generated by \eqref{ee13}.  Then,  one observes that $Y^{T_{\Delta_1}(X_{n}),y_0}=Y^{T_{\Delta_1}(X_{n}),y_0}_{n}~\mathrm{a.s.}$ Thus,
{\begin{align}\label{cyp4.7}
		B_{M}\big(T_{\Delta_1}(X_{n}),Y^{T_{\Delta_1}(X_{n}),y_0}\big)=B_{M}\big(T_{\Delta_1}(X_{n}),Y^{T_{\Delta_1}(X_{n}),y_0}_{n}\big)~~\mathrm{a.s.}
\end{align}}

\par Next, we  give several properties of  $Y^{x,y_0}_{n}$ in order for the estimation of   $B_{M}(x,Y^{x,y_0}_{n})$. The uniform moment bound  result of  $Y^{x,y_0}_{n}$ can be obtained  from \cite[Lemmas 3.7]{MR2102646}, but it is not accurate to prove the desired properties of the estimator  $B_{M}(x,Y^{x,y_0}_{n})$. For the sake of completeness, we would like to provide more details here.
\begin{lem}\label{Lb.3.10}
	{\rm If $({\bf F1})$ and $({\bf F3})$  hold {with  $k\geq2$}, then there exists a $\hat\Delta_2\in (0,1]$ such that for any $x\in \RR^{n_1}$, {$y_0\in \RR^{n_2}$},  integer $n\geq 0$ and $\Delta_2\in (0,\hat\Delta_2]$,
		\begin{align*}
			&\sup_{m\geq0}\mathbb{E}|Y^{x,y_0}_{n,m}|^{k}\leq C(1+|y_0|^{k}+|x|^{k}),
		\end{align*}
		and
		\begin{align*}
			\sup_{t\geq 0} \mathbb{E}|\bar{Y}^{x,y_0}_{n}(t)-{Y}^{x,y_0}_{n}(t)|^{k}
			\leq C(1+|y_0|^{k}+|x|^{k})\Delta_2^{\frac{k}{2}}.
	\end{align*}}
\end{lem}
\begin{proof} {\bf Proof}.
For  any $t>0$, using the $\mathrm{It\hat{o}}$ formula,  we derive from \eqref{a3.30} that
\begin{align}\label{a3.31}
\mathbb{E}\Big(e^{\frac{k\alpha_k t}{8} }|\bar{Y}^{x,y_0}_{n}(t)|^{k}\Big)&\leq |y_0|^{k}+\mathbb{E}\int_{0}^{t}e^{\frac{k\alpha_k s}{8} }|\bar{Y}^{x,y_0}_{n}(s)|^{k-2}\Big[\frac{k\alpha_k}{8}|\bar{Y}^{x,y_0}_{n}(s)|^{2}
\nn\
\\&~~+k(\bar{Y}^{x,y_0}_{n}(s))^{T}f(x,Y^{x,y_0}_{n}(s))+\frac{k(k-1)}{2}|g(x,Y^{x,y_0}_{n}(s))|^{2}\Big]\mathrm{d}s.
\end{align}
 Invoking the Young inequality, ({\bf F1}) and  ({\bf F3}) implies that
\begin{align*}
 &k(\bar{Y}^{x,y_0}_{n}(s))^{T}f(x,Y^{x,y_0}_{n}(s))+\frac{k(k-1)}{2}|g(x,Y^{x,y_0}_{n}(s))|^{2}\nn\
\\ \leq& k\Big[\big(\bar{Y}^{x,y_0}_{n}(s)\big)^{T}f(x,\bar{Y}^{x,y_0}_{n}(s))
+\frac{(k-1)}{2}|g(x,\bar{Y}^{x,y_0}_{n}(s))|^{2}\Big]\nn\
\\&~~~
+k\big(\bar{Y}^{x,y_0}_{n}(s)\big)^{T}\big(f(x,Y^{x,y_0}_{n}(s))
-f(x,\bar{Y}^{x,y_0}_{n}(s))\big)\nn\
\\&~~~+\frac{k(k-1)}{2}|g(x,Y^{x,y_0}_{n}(s))
-g(x,\bar{Y}^{x,y_0}_{n}(s))|^2\nn\
\\&~~~
+k(k-1)|g(x,\bar{Y}^{x,y_0}_{n}(s))||g(x,Y^{x,y_0}_{n}(s))-g(x,\bar{Y}^{x,y_0}_{n}(s))|\nn\
\\ \leq& C(1+|x|^2)-\frac{k\alpha_k}{2}|\bar{Y}^{x,y_0}_{n}(s)|^2\nn\
+C|Y^{x,y_0}_{n}(s)-\bar{Y}^{x,y_0}_{n}(s)|^{2}.
\end{align*}
Substituting the above inequality into (\ref{a3.31}) and using the Young inequality, we get
\begin{align}\label{e3.31}
\mathbb{E}\Big(e^{\frac{k\alpha_k t}{8} }|\bar{Y}^{x,y_0}_{n}(t)|^{k}\Big)&\leq |y_0|^{k}+C(1+|x|^{2})^{\frac{k}{2}} e^{\frac{k\alpha_k t}{8}  } { -\frac{k\alpha_k}{8}}\int_{0}^{t}e^{\frac{k\alpha_k s}{8} }
\mathbb{E}|\bar{Y}^{x,y_0}_{n}(s)|^{k}\mathrm{d}s
\nn\\&~~~+C\int_{0}^{t} e^{\frac{k\alpha_k s }{8} }\mathbb{E}|Y^{x,y_0}_{n}(s )-\bar{Y}^{x,y_0}_{n}(s)|^{k}\mathrm{d}s.
\end{align}
Moreover, it follows from $({\bf F1})$ and (\ref{a3.30}) that
\begin{equation}\label{c3.28}
\begin{aligned}
 \mathbb{E}|{Y}^{x,y_0}_{n}(s)-\bar{Y}^{x,y_0}_{n}(s)|^{k} \leq & 2^{k-1}\Delta_{2}^{k}\mathbb{E}|f(x,Y^{x,y_0}_{n}(s))|^{k}
 +2^{k-1}\Delta_{2}^{\frac{k}{2}}\mathbb{E}|g(x,Y^{x,y_0}_{n}(s))|^{k}\nn\
\\ \leq&C\Delta_{2}^{\frac{k}{2}}\mathbb{E}\big(1+|x|^{k} +|Y^{x,y_0}_{n}(s)|^{k}\big)\nn\
\\ \leq & C\Delta_{2}^{\frac{k}{2}}\Big(1+|x|^{k}+\mathbb{E}|\bar{Y}^{x,y_0}_{n}(s)|^{k}
+\mathbb{E}|Y^{x,y_0}_{n}(s)- \bar{Y}^{x,y_0}_{n}(s) |^{k} \Big).
\end{aligned}
\end{equation}
Choose a constant $ {\Delta}'_{2} \in (0, 1]$ small enough such that
 $ C( { {\Delta}}_{2}')^{\frac{k}{2}}\leq 1/2$. Then, for any $\Delta_2\in (0, {\Delta}_2']$,
\begin{align}\label{c3.64}
\mathbb{E}|Y^{x,y_0}_{n}(s)-\bar{Y}^{x,y_0}_{n}(s)|^{k}\leq C\Delta_2^{\frac{k}{2}}(1+|x|^{k}+\mathbb{E}|\bar{Y}^{x,y_0}_{n}(s)|^k).
\end{align}
Inserting (\ref{c3.64}) into (\ref{e3.31}) leads to
\begin{align*}
\mathbb{E}\Big(e^{\frac{k\alpha_k  t}{8}  }|\bar{Y}^{x,y_0}_{n}(t)|^{k}\Big)\nn\
 \leq   |y_{0}|^{k} +C(1+|x|^{k})   e^{\frac{k\alpha_k t}{8}}-\Big(\frac{k\alpha_{k}}{8}-C\Delta_2^{\frac{k}{2}}\Big)
  \int_{0}^{t}e^{\frac{k\alpha_k s}{8} }\mathbb{E}|\bar{Y}^{x,y_0}_{n}(s)|^{k}\mathrm{d}s.
\end{align*}
Furthermore, choose $ \hat\Delta_2 \in (0,  {\Delta}_{2}']$ small enough such that  $C\big( \hat\Delta_2 \big)^{\frac{k}{2}}\leq q\alpha_{k}/8 $.  Then, for any $\Delta_{2}\in(0,\hat\Delta_2 ]$,  we derive that
\begin{align*}
\mathbb{E}\Big(e^{\frac{k\alpha_k t}{8} }|\bar{Y}^{x,y_0}_{n}(t )|^{k}\Big)&\leq |y_0|^{k}+ C(1+|x|^{k}) e^{\frac{k\alpha_k t}{8}  } .
\end{align*}
Then a direct computation gives that
\begin{align} \label{f17}
\mathbb{E}|\bar{Y}^{x,y_0}_{n}(t)|^{k}\leq |y_0|^{k} e^{-\frac{k\alpha_{k}t}{8}}+C(1+|x|^{k}).
\end{align}
Thus, one obtains
\begin{align*}
\sup_{t\geq 0}\mathbb{E}|\bar{Y}^{x,y_0}_{n}(t)|^{k}\leq C(1+|y_0|^{k}+|x|^{k}),
\end{align*}
which implies the first desired result. Then  substituting the above inequality into \eqref{c3.64} gives the another desired result.
The proof is complete.
\end{proof}
\begin{lem}\label{CL3.9}
	{\rm Under $({\bf F1})$ and $({\bf F2})$,  there exists a constant $\bar\Delta_2\in (0,\hat{\Delta}_2]$ such that for any $\Delta_2\in (0,\bar\Delta_2)$, {$y, z\in \mathbb{R}^{n_2}$}, $x\in \mathbb{R}^{n_1}$, integers $n\geq 0$ and $m\geq0$,
		\begin{align*}
			\mathbb{E}|Y^{x,y}_{n,m}-Y^{x,z}_{n,m}|^2\leq |y-z|^2e^{\frac{-\beta m \Delta_2}{2}}.
	\end{align*}}
\end{lem}
\begin{proof}{\bf Proof.}
For notation brevity, set $v^{x}_{n,m}:=Y^{x,y}_{n,m}-Y^{x,z}_{n,m}$,
$$F(x,Y^{x,y}_{n,m},Y^{x,z}_{n,m}):=f(x,Y^{x,y}_{n,m})-f(x,Y^{x,z}_{n,m}),$$
$$G(x,Y^{x,y}_{n,m},Y^{x,z}_{n,m}):=g(x,Y^{x,y}_{n,m})-g(x,Y^{x,z}_{n,m}).$$
In view of \eqref{ee13}, for any integer $m\geq0$ we have
\begin{align*}
v^{x}_{n,m+1}=v^{x}_{n,m}+F(x,Y^{x,y}_{n,m},Y^{x,z}_{n,m})\Delta_2+G(x,Y^{x,y}_{n,m},Y^{x,z}_{n,m})\Delta W^{2}_{n,m}.
\end{align*}
Then we derive that
\begin{align}\label{c3.35}
|v^{x}_{n,m+1}|^2&=|v^{x}_{n,m}|^2+2(v^{x}_{n,m})^{T}F(x,Y^{x,y}_{n,m},Y^{x,z}_{n,m})\Delta_2+|G(x,Y^{x,y}_{n,m},Y^{x,z}_{n,m})\Delta W^{2}_{n,m}|^2\nn\
\\&~~~+|F(x,Y^{x,y}_{n,m},Y^{x,z}_{n,m})|^2\Delta_2^2+2(v^{x}_{n,m})^{T}G(x,Y^{x,y}_{n,m},Y^{x,z}_{n,m})\Delta W^{2}_{n,m}\nn\
\\&~~~+2F^{T}(x,Y^{x,y}_{n,m},Y^{x,z}_{n,m})G(x,Y^{x,y}_{n,m},Y^{x,z}_{n,m})\Delta W^{2}_{n,m}\Delta_2.
\end{align}
For any integers $n\geq 1$, $l\geq 1$, denote by $\mathcal{F}^2_{n,m}$ the $\sigma$-algebra generated  by $$\Big\{ W^{2}_n(s): 0\leq s\leq m\Delta_{2}\Big\}.$$
The fact that
$\Delta W^2_{n,m}$ is independent of $\mathcal{F}^2_{n,m}$ implies that
\begin{align}\label{cyp4.44}
\mathbb{E}\big(A\Delta W^{2}_{n,m}|\mathcal{F}^2_{n,m}\big)=0,~\mathbb{E}\big(|A\Delta W^{2}_{n,m}|^{2}|\mathcal{F}^{2}_{n,m}\big)=C\Delta_2,~\forall A\in \mathbb{R}^{n_2\times d_2}.
\end{align}
Then taking  expectation on both sides for \eqref{c3.35} and using ({\bf F1}), ({\bf F2}) and \eqref{cyp4.44} imply that
\begin{align*}
\mathbb{E}|v^{x}_{n,m+1}|^2=\mathbb{E}|v^{x}_{n,m}|^2-\beta \Delta_2 \mathbb{E}|v^{x}_{n,m}|^2+L^2\Delta_2^2\mathbb{E}|v^{x}_{n,m}|^2.
\end{align*}
Choosing $\bar\Delta_2\leq \hat\Delta_2 \wedge (\beta/2L^2)\wedge (2/\beta)$, for any $\Delta_2\in (0,\bar\Delta_2)$, one obtains   that
\begin{align*}
L^2\Delta_2^2\leq \frac{\beta \Delta_2}{2}, ~~~\frac{\beta\Delta_2}{2}\leq 1.
\end{align*}
Thus, one derives that
\begin{align*}
\mathbb{E}|v^{x}_{n,m+1}|^2&\leq \Big(1-\frac{\beta \Delta_2}{2}\Big)\mathbb{E}|v^{x}_{n,m}|^2\nn\
\\&\leq \cdots\leq |y-z|^2 \Big(1-\frac{\beta \Delta_2}{2}\Big)^{m+1}
\leq |y-z|^2e^{-\frac{\beta m \Delta_2}{2}},
\end{align*}
 where the last step used the inequality that $1-\beta \Delta_2/2\leq e^{-\frac{\beta \Delta_2}{2}}$, which implies the desired result. 
\end{proof}
	\begin{lem}\label{CL3.7}
		{\rm If $({\bf F1})$-$({\bf F3})$ hold {with some $k\geq2$}, then for any fixed $x\in \mathbb{R}^{n_1}$, {$y_0\in \RR^{n_2}$,}  integer $n\geq 0$ and $\Delta_2\in (0,\bar\Delta_2]$, {$Y^{x,y_0}_{n}$ determined by \eqref{ee13} admits a unique invariant measure  $\mu^{x,\Delta_2}\in \mathcal{P}_{k}(\mathbb{R}^{n_2})$,  which is independent of $y_0$ and $n$,} and  satisfies
			\begin{align*}
				\int_{\mathbb{R}^{n_2}}|y|^{k}\mu^{x,\Delta_2}(\mathrm{d}y)\leq C(1+|x|^{k}).
		\end{align*}}
	\end{lem}
\begin{proof}\textbf {Proof.}
		Since the EM  numerical solutions~$Y^{x,y_0}_{n},~n=1,\cdots,\infty$ ~are i.i.d and have Markov property,
		for~any $\Delta_2\in(0,1)]$,  we use $\mathbb{P}^{x,\Delta_2}_{m\Delta_2}$ to denote the same discrete Markov semigroup of $Y^{x,y_0}_{n}$. Under $({\bf F1})$-$({\bf F3})$, {with the help of Lemmas \ref{Lb.3.10}-\ref{CL3.9}, proceeding a similar argument to \cite[Theorem 3,1]{Bao} we derive that for any $y_0\in \RR^{n_2}$ and integer $n\geq 0$,  $Y^{x,y_0}_{n}$ has a unique invariant measure  denoted by $\mu^{x,\Delta_2}$, which is independent of $y_0$ and $n$.}
		Furthermore, applying \eqref{f17} yields that
		\begin{align*}
			\int_{\mathbb{R}^{n_2}}(|y|^{k}\wedge N)\mu^{x,\Delta_2}(\mathrm{d}y)&=\int_{\mathbb{R}^{n_2}}\mathbb{E}(|Y^{x,y}_{n,m}|^{k}\wedge N)\mu^{x,\Delta_2}(\mathrm{d}y)\notag
			\\&\leq \int_{\mathbb{R}^{n_2}}(\mathbb{E}|Y^{x,y}_{n,m}|^{k}\wedge N)\mu^{x,\Delta_2}(\mathrm{d}y)\notag
			\\&\leq\int_{\mathbb{R}^{n_2}}\Big(|y|^{k} e^{-\frac{q\alpha_{k}m\Delta_2}{8}}\wedge N\Big)\mu^{x,\Delta_2}(\mathrm{d}y)+C(1+|x|^{k}),
		\end{align*}
		where the identity is due to the invariance of invariant measure $\mu^{x,\Delta_2}$ and the first inequality holds  by Jensen's inequality since $x\mapsto N\wedge x,x\in \mathbb{R}$ is a convex  function. Then, taking $m\rightarrow\infty$ and using the dominated convergence theorem, we deduce that
		\begin{align*}
			\int_{\mathbb{R}^{n_2}}(|y|^{k}\wedge N)\mu^{x,\Delta_2}(\mathrm{d}y)\leq C(1+|x|^k).
		\end{align*}
		Letting $N\rightarrow\infty$  and applying the monotone convergence theorem, we get
		\begin{align*}
			\int_{\mathbb{R}^{n_2}}|y|^{k}\mu^{x,\Delta_2}(\mathrm{d}y)\leq C(1+|x|^k).
		\end{align*}
		The proof is complete.
	\end{proof}
	\begin{lem}\label{Lb.6}
		{\rm If $({\bf F1})$-$({\bf F3})$ hold {with some $k\geq2$}, then for any fixed $x\in \RR^{n_1},{ y_0\in \RR^{n_2}}$, integer $n\geq
			0$ and $\Delta_2\in (0,\hat\Delta_2]$,
			{\begin{align*}
					\sup_{m\geq0}\mathbb{E}\big|Y^{x,y_0}_{n,m}-y^{x,y_0}_{n}(m\Delta_2)\big|^{2}\leq C(1+ |x|^2)\Delta_2.
		\end{align*}}}
	\end{lem}
	\begin{proof}\textbf {Proof.}
		In view of (\ref{e16}) and (\ref{a3.30}), define $\bar{v}^{x,y_0}_{n}(t):=\bar{Y}_{n}^{x,y_0}(t)-y^{x,y_0}_{n}(t)$ described by
		\begin{align*}
			\!\mathrm{d}\bar{v}^{x,y_0}_{n}(t)=\Big(f(x,Y^{x,y_0}_{n}(t))\!-\!f(x,y^{x,y_0}_{n}(t))\Big)\mathrm{d}t+\Big(g(x,Y^{x,y_0}_{n}(t))-g(x,y^{x,y_0}_{n}(t))\Big)\mathrm{d}W^{2}_{n}(t).\notag
		\end{align*}
		Using the $\mathrm{It\hat{o}}$ formula one arrives at
		\begin{align}\label{c3.38}
			\mathbb{E}\Big(e^{\frac{\beta t}{4} }|\bar{v}^{x,y_0}_{n}(t)|^{2}\Big)
			\leq& \mathbb{E}\int_{0}^{t}\bigg[\frac{\beta}{4} e^{\frac{\beta s}{4} }\big|\bar{v}^{x,y_0}_{n}(s)\big|^{2}+e^{\frac{\beta s}{4} }\Big(2(\bar{v}^{x,y_0}_{n}(s))^{T}\big[f(x,Y^{x,y_0}_{n}(s))\notag
			\\&-f(x,y^{x,y_0}_{n}(s))\big]+\big|g(x,Y^{x,y_0}_{n}(s))-g(x,y^{x,y_0}_{n}(s))\big|^{2}\Big)\Bigg]\mathrm{d}s.
		\end{align}
		Invoking $({\bf F1})$, $({\bf F2})$ and the Young inequality yields that
		\begin{align*}
			&2(\bar{v}^{x,y_0}_{n}(s))^{T}\big[f(x,Y^{x,y_0}_{n}(s))-f(x,y^{x,y_0}_{n}(s))\big]+\big|g(x,Y^{x,y_0}_{n}(s))-g(x,y^{x,y_0}_{n}(s))\big|^{2}\notag
			\\ \leq& 2(\bar{v}^{x,y_0}_{n}(s))^{T}\big[f(x,\bar{Y}^{x,y_0}_{n}(s))-f(x,y^{x,y_0}_{n}(s))\big]+\big|g(x,\bar{Y}^{x,y_0}_{n}(s))-g(x,y^{x,y_0}_{n}(s))\big|^{2}\notag
			\\ &~~~+ 2(\bar{v}^{x,y_0}_{n}(s))^{T}\big[f(x,Y^{x,y_0}_{n}(s))-f(x,\bar{Y}^{x,y_0}_{n}(s))\big]+\big|g(x,Y^{x,y_0}_{n}(s))-g(x,\bar{Y}^{x,y_0}_{n}(s))\big|^{2}\notag
			\\&~~~+2\big|g(x,\bar{Y}^{x,y_0}_{n}(s))-g(x,y^{x,y_0}_{n}(s))\big|\big|g(x,Y^{x,y_0}_{n}(s))-g(x,\bar{Y}^{x,y_0}_{n}(s))\big|\notag
			\\ \leq& -\beta |\bar{v}^{x,y_0}_{n}(s)|^2+C|\bar{v}^{x,y_0}_{n}(s)||Y^{x,y_0}_{n}(s)-\bar{Y}^{x,y_0}_{n}(s)|+C|Y^{x,y_0}_{n}(s)-\bar{Y}^{x,y_0}_{n}(s)|^2\notag
			\\ \leq& -\frac{\beta}{2} |\bar{v}^{x,y_0}_{n}(s)|^2+C|Y^{x,y_0}_{n}(s)-\bar{Y}^{x,y_0}_{n}(s)|^2.
		\end{align*}
		Then inserting the above inequality into (\ref{c3.38}) and using $({\bf F1})$ and $({\bf F3})$, we derive from the result of lemma \ref{Lb.3.10} that
		\begin{align*}
			e^{\frac{\beta t}{4}}\mathbb{E}|\bar{v}^{x,y_0}_{n}(t)|^{2}&\leq C\int_{0}^{t}e^{\frac{\beta s}{4} }\mathbb{E}|Y^{x,y_0}_{n}(s)-\bar{Y}^{x,y_0}_{n}(s)|^2\mathrm{d}s
			\leq   C(1+|x|^2)\Delta_2e^{\frac{\beta t}{4}},
		\end{align*}
		which yields the desired result.
	\end{proof}
	\par By  virtue of Lemma \ref{Lb.6}, we obtain the  convergence rate between numerical invariant measure $\mu^{x,\Delta_2}$ and the underlying invariant measure  $\mu^{x}$ in  $\mathbb{W}_2$-distance.
	\begin{lem}\label{cL3.10}
		{\rm Under $({\bf F1})$-$({\bf F3})$ {with some $k\geq2$}, for any fixed $x\in \mathbb{R}^{n_1}$ and $\Delta_2\in (0,\bar\Delta_2]$,
			\begin{align*}
				\mathbb{W}_2(\mu^{x},\mu^{x,\Delta_2})\leq C(1+|x|)\Delta_2^{\frac{1}{2}}.
		\end{align*}}
	\end{lem}
	\begin{proof}\textbf{Proof.}
		From the proofs of Lemmas \ref{Lcyp2.1} and \ref{CL3.7}, we know that
		\begin{align*}
			\lim_{m\rightarrow\infty}\mathbb{W}_{2}(\delta_0\mathbb{P}^{x}_{m\Delta_2},\mu^{x})=0
		\end{align*}
		and
		\begin{align*}
			\lim_{m\rightarrow\infty}\mathbb{W}_{2}(\delta_0\mathbb{P}^{x,\Delta_2}_{m\Delta_2},\mu^{x,\Delta_2})=0.
		\end{align*}
		The above inequalities, together with  Lemma \ref{Lb.6}, imply that
		\begin{align}
			\mathbb{W}_2(\mu^{x},\mu^{x,\Delta_2})&\leq \lim_{m\rightarrow\infty}\mathbb{W}_{2}(\mu^{x},\delta_0\mathbb{P}^{x}_{m\Delta_2}) +\lim_{m\rightarrow\infty}\mathbb{W}_{2}(\delta_0\mathbb{P}^{x}_{m\Delta_2},\delta_{0}\mathbb{P}^{x,\Delta_2}_{m\Delta_2})\notag
\\&~~~+\lim_{m\rightarrow\infty}\mathbb{W}_{2}(\delta_0\mathbb{P}^{x,\Delta_2}_{m\Delta_2},\mu^{x,\Delta_2})\notag
			\\&\leq{ \lim_{m\rightarrow\infty}\big(\mathbb{E}|y^{x,0}_{n}(m\Delta_2)-Y^{x,0}_{n,m}|^2\big)^{\frac{1}{2}}\leq C(1+|x|)\Delta^{\frac{1}{2}}}.\notag
		\end{align}
	The proof is complete.
	\end{proof}
\par Now we turn to analyze the property of the estimator $B_{M}(x,Y^{x,y_0}_{n})$.
\begin{lem}\label{L4.1}
If $({\bf S4})$, $({\bf F1})$ and $({\bf F3})$ hold, then for any $0\leq p\leq k/\theta_4$, there exists a constant $C_{y_0,p,K}$ such that
\begin{align*}
\mathbb{E}\Big|B_{M}\Big(T_{\Delta_1}(x),Y^{T_{\Delta_1}(x),y_0}_{n}\Big)\Big|^{p}\leq C_{y_0,p,K}\Delta_{1}^{-\frac{p}{2}}(1+|x|^{p}),
\end{align*}
where  $K\geq ~1+{\varphi(|x_0|)}$ is defined in \eqref{eqN3.2}.
\end{lem}
\begin{proof}\noindent\textbf {Proof.}
Under $({\bf S4})$, it follows from \eqref{cyp3.3} and \eqref{3.9} that 
		\begin{align}\label{eqqN4.8} \mathbb{E}\Big|B_{M}\Big(T_{\Delta_1}(x),Y^{T_{\Delta_1}(x),y_0}_{n}\Big)\Big|^{p}=&\mathbb{E}\Big|\frac{1}{M}\sum_{m=1}^{M}b\big(T_{\Delta_1}(x),Y^{T_{\Delta_1}(x),y_0}_{n,m}\big)\Big|^{p} \nn\
\\ \leq & \frac{1}{M}\sum_{m=1}^{M}\mathbb{E}\Big|b\big(T_{\Delta_1}(x),Y^{T_{\Delta_1}(x),y_0}_{n,m}\big)\Big|^{p}\notag
			\\ \leq  & \frac{1}{M}\sum_{m=1}^{M}\mathbb{E}\Big[\Big(C\Delta_1^{-\frac{1}{2}}(1+|T_{\Delta_1}(x)|)
			+K_3\Big|Y^{T_{\Delta_1}(x),y_0}_{n,m}\Big|^{\theta_4}\Big)\Big]^{p}\notag
			\\ \leq & C_{p}\Delta_1^{-\frac{p}{2}}\big(1+|T_{\Delta_1}(x)|\big)^{p}+\frac{C_{p}}{M}\sum_{m=1}^{M}\mathbb{E}\big|Y^{T_{\Delta_1}(x),y_0}_{n, m}\big|^{ p\theta_4}.
		\end{align}
		 Under $({\bf F1})$ and $({\bf F3})$, using the H\"older inequality  and Lemma \ref{Lb.3.10} yields that for any  $0\leq p\leq k/\theta_4$,
		\begin{align*}
			\mathbb{E}\big|Y^{T_{\Delta_1}(x),y_0}_{n, m}\big|^{p\theta_4}\leq \Big(\mathbb{E}\big|Y^{T_{\Delta_1}(x),y_0}_{n, m}\big|^{k}\Big)^{\frac{p\theta_4}{k}} \leq C\big(1+|y_0|^{p\theta_4}+|T_{\Delta_1}(x)|^{p\theta_4}\big) \leq C_{y_0,p}\big(1+|T_{\Delta_1}(x)|^{p\theta_4}\big).
		\end{align*}
Since $|T_{\Delta_1}(x)|\leq |x|\wedge\varphi^{-1}(K\Delta_1^{-\frac{1}{2}}), ~\forall \Delta\in (0,1], \forall x\in \RR^{n_1}$,
one derives that
		\begin{align}\label{eqNN4.9}
				\mathbb{E}\big|Y^{T_{\Delta_1}(x),y_0}_{n, m}\big|^{p\theta_4}  
&\leq  C_{y_0,p}\big(1+|\varphi^{-1}(K\Delta^{-\frac{1}{2}})|^{p(\theta_4-1)}|T_{\Delta_1}(x)|^{p}\big)\nn\
\\&
\leq C_{y_0,p,K}\Delta_1^{-\frac{p}{2}}(1+|T_{\Delta_1}(x)|^{p})\leq C_{y_0,p,K}\Delta_1^{-\frac{p}{2}}(1+|x|^{p}).
		\end{align}
		Inserting \eqref{eqNN4.9} into \eqref{eqqN4.8} implies that
		\begin{align*} \mathbb{E}\Big|B_{M}\Big(T_{\Delta_1}(x),Y^{T_{\Delta_1}(x),y_0}_{n}\Big)\Big|^{p}&\leq C_{y_0,p,K}\Delta_1^{-\frac{p}{2}}\big(1+|x|^{p}\big).
		\end{align*}
The proof is complete.
\end{proof}
	
	The error between  $\bar{b}(x)$ and  $B_{M}(x,Y^{x,y_0}_{n})$ is the key  to obtain  the convergence of the  numerical solution of the MTEM scheme. By introducing an  auxiliary function as below
	\begin{align}\label{cyp3.35}
		\bar{b}^{\Delta_2}(x)=\int_{\mathbb{R}^{n_2}}b(x,y)\mu^{x,\Delta_2}(\mathrm{d}y),~~x\in \RR^{d},
	\end{align}
	we use $|\bar{b}(x)-\bar{b}^{\Delta_2}(x)|^2$ and $\E |\bar{b}^{\Delta_2}(x)-B_{M}(x,Y^{x,y_0}_{n})|^2$ to estimate $\E|\bar{b}(x)-B_{M}(x,Y^{x,y_0}_{n})|^2$. In fact,	under $({\bf S4})$ and ({\bf F1})-({\bf F3}) with $k\geq \theta_4$,  by virtue of Lemma \ref{CL3.7}, for any fixed $x\in \mathbb{R}^{n_1}$ and $\Delta_2\in (0,\bar\Delta_2]$,
	\begin{align}\label{f13}
		|\bar{b}^{\Delta_2}(x)|\leq\int_{\mathbb{R}^{n_2}}|b(x,y)|\mu^{x,\Delta_2}(\mathrm{d}y)&\leq K_3\int_{\mathbb{R}^{n_2}}(1+|x|^{\theta_3}+|y|^{\theta_4})\mu^{x,\Delta_2}(\mathrm{d}y)\notag
\\&\leq C(1+|x|^{\theta_3\vee\theta_4})< \infty.
	\end{align}
	Thus, $\bar{b}^{\Delta_2}(x)$ is well-posed  under $({\bf S4})$ and ({\bf F1})-({\bf F3}) with $k\geq \theta_4$. Next,  we estimate  $|\bar{b}(x)-\bar{b}^{\Delta_2}(x)|^2$ and $\E |\bar{b}^{\Delta_2}(x)-B_{M}(x,Y^{x,y_0}_{n})|^2$, respectively.
	\begin{lem}\label{L3.13}
		{\rm Under $({\bf S2})$, $({\bf S4})$ and $({\bf F1})$-$(\bf {F3})$ with $k\geq{ 2\vee2\theta_2\vee\theta_4}$, for any $x\in \mathbb{R}^{n_1}$ and $\Delta_2\in (0,\bar\Delta_2]$,
			\begin{align*}
				|\bar{b}(x)-\bar{b}^{\Delta_2}(x)|\leq C(1+|x|^{\theta_2+1})\Delta_2^{\frac{1}{2}}.
		\end{align*}}
	\end{lem}
	\begin{proof}\textbf {Proof.}
		Under $({\bf S4})$, $({\bf F1})$-$({\bf F3})$ with $k\geq {2\vee\theta_4}$, in view of \eqref{f11} and \eqref{cyp3.35}, using $({\bf S2})$ and the H\"older inequality yields that
		\begin{align}
			|\bar{b}(x)-\bar{b}^{\Delta_2}(x)|&=\Big|\int_{\mathbb{R}^{n_2}\times \mathbb{R}^{n_2}} \big(b(x,y_1)-b(x,y_2)\big)\pi(\mathrm{d}y_1,\mathrm{d}y_2)\Big|\notag
			\\&\leq \int_{\mathbb{R}^{n_2}\times \mathbb{R}^{n_2}} \big|b(x,y_1)-b(x,y_2)\big|\pi(\mathrm{d}y_1,\mathrm{d}y_2)\notag
			\\&\leq  C\Big(\int_{\mathbb{R}^{n_2}\times \mathbb{R}^{n_2}}|y_1-y_2|^2\pi(\mathrm{d}y_1,\mathrm{d}y_2)\Big)^{\frac{1}{2}}\notag
			\\&~~~\times \Big(\int_{\mathbb{R}^{n_2}\times \mathbb{R}^{n_2}}(1+|x|^{2\theta_2}+|y_1|^{2\theta_2}+|y_2|^{2\theta_2})\pi(\mathrm{d}y_2,\mathrm{d}y_2)\Big)^{\frac{1}{2}},\notag
		\end{align}
		where $\pi\in \mathcal{C}(\mu^{x},\mu^{x,\Delta_2})$ is arbitrary. Thus,  we derive that
		\begin{align*}
			|\bar{b}(x)-\bar{b}^{\Delta_2}(x)|&\leq C\mathbb{W}_{2}(\mu^{x},\mu^{x,\Delta_2})
			\times\Big(1+|x|^{2\theta_2}+\int_{\mathbb{R}^{n_2}}|y_1|^{2\theta_2}\mu^{x}(\mathrm{d}y_1)\notag
\\&~~~+\int_{\mathbb{R}^{n_2}}|y_2|^{2\theta_2}\mu^{x,\Delta_2}(\mathrm{d}y_2)\Big)^{\frac{1}{2}}.
		\end{align*}
		Then due to $({\bf F1})$-$({\bf F3})$ with $k\geq {2\vee2\theta_2}$, applying Lemmas \ref{Lcyp2.1}, \ref{CL3.7} and \ref{cL3.10} implies that
		\begin{align*}
			|\bar{b}(x)-\bar{b}^{\Delta_2}(x)|&\leq C(1+|x|^{\theta_2+1})\Delta_2^{\frac{1}{2}}.
		\end{align*}
		The proof is complete.
	\end{proof}
	
	Before estimate  $\mathbb{E}|\bar{b}^{\Delta_2}(x)-B_{M}(x,Y^{x,y_0}_{n})|^2$,  we  prepare a useful result.
	\begin{lem}\label{L3.14}
		{\rm Under $({\bf S2})$, $({\bf S4})$  and  $({\bf F1})$-$({\bf F3})$ with $k\geq {2\vee2\theta_2\vee(\theta_2+1)\vee\theta_4}$, for any $x\in \mathbb{R}^{n_1}$, {$y\in \mathbb{R}^{n_2}$}, $\Delta_2\in (0,\bar\Delta_2]$ and integers $n\geq 0$, $M\geq1$,
			\begin{align*}
				|\bar{b}^{\Delta_2}(x)-\mathbb{E}b(x,Y^{x,y}_{n,m}) |\leq C(1+|x|^{\theta_2+1}+|y|^{\theta_2+1})e^{\frac{-\beta m \Delta_2}{4}}.
		\end{align*}}
	\end{lem}
	\begin{proof}\textbf {Proof.}
		Under $({\bf S4})$ and $({\bf F1})$-$({\bf F3})$ with $k\geq {2\vee\theta_4}$,
		according to \eqref{cyp3.35} and the invariance of invariant probability measure $\mu^{x,\Delta_2}$, we have
		\begin{align}\label{cyp3.43}
			\bar{b}^{\Delta_2}(x)&=\lim_{N\rightarrow\infty}\int_{\mathbb{R}^{n_2}}b(x,z)I_{\{|z|\leq N\}}\mu^{x,\Delta_2}(\mathrm{d}z)\notag
			\\&\leq \lim_{N\rightarrow\infty}\int_{\mathbb{R}^{n_2}}\mathbb{E}\Big(b(x,Y^{x,z}_{n,m})I_{\{|Y^{x,z}_{n,m}|\leq N\}}\Big)\mu^{x,\Delta_2}(\mathrm{d}z).
		\end{align}
Note that $\lim\limits_{N\rightarrow\infty}b(x,Y^{x,z}_{n,m})I_{\{|Y^{x,z}_{n,m}|\leq N\}}=b(x,Y^{x,z}_{n,m}),~\mathrm{a.s.}$ for any  $z\in \mathbb{R}^{n_2}$. In addition, by $({\bf S4})$ and $({\bf F1})$-$({\bf F3})$ with $k\geq{2\vee\theta_4}$, using  Lemma \ref{CL3.7} yields that
		\begin{align*}
			\int_{\mathbb{R}^{n_2}}\mathbb{E}|b(x,Y^{x,z}_{n,m})|\mu^{x,\Delta_2}(\mathrm{d}z)&\leq C\Big(1+|x|^{\theta_3}+\int_{\mathbb{R}^{n_2}}\mathbb{E}|Y^{x,z}_{n,m}|^{\theta_4}\mu^{x,\Delta_2}(\mathrm{d}z)\Big)
			\\ &\leq C\Big(1+|x|^{\theta_3\vee \theta_4}+\int_{\mathbb{R}^{n_2}}|z|^{\theta_4} \mu^{x,\Delta_2}(\mathrm{d}z)\Big)
			\\ &\leq C(1+|x|^{\theta_3\vee\theta_4})<\infty.
		\end{align*}
		Then applying the dominated convergence theorem for \eqref{cyp3.43} we derive that
		\begin{align*}
			\bar{b}^{\Delta_2}(x)=\int_{\mathbb{R}^{n_2}}\mathbb{E}b(x,Y^{x,z}_{n,m})\mu^{x,\Delta_2}(\mathrm{d}z).
		\end{align*}
		As a result,  we have
		\begin{align}
			| \bar{b}^{\Delta_2}(x)-\mathbb{E}b(x,Y^{x,y}_{n,m})|&= \Big|\mathbb{E}b(x,Y^{x,y}_{n,m})-\int_{\mathbb{R}^{n_2}}\mathbb{E}b(x,Y^{x,z}_{n,m})\mu^{x,\Delta_2}(\mathrm{d}z)\Big|\notag
			\\&\leq \int_{\mathbb{R}^{n_2}}\mathbb{E}\big|b(x,Y^{x,y}_{n,m})-b(x,Y^{x,z}_{n,m})\big|\mu^{x,\Delta_2}(\mathrm{d}z).\notag
		\end{align}
		Further using  $({\bf S2})$  and the H\"older inequality gives that
		\begin{align*}
			&| \bar{b}^{\Delta_2}(x)-\mathbb{E}b(x,Y^{x,y}_{n,m})| 
			\\ \leq&K_1\int_{\mathbb{R}^{n_2}}\mathbb{E}\Big(|Y^{x,y}_{n,m}-Y^{x,z}_{n,m}|(1+|x|^{\theta_2}+|Y^{x,y}_{n,m}|^{\theta_2}+|Y^{x,z}_{n,m}|^{\theta_2})\Big)\mu^{x,\Delta_2}(\mathrm{d}z)
			\\ \leq&  C\int_{\mathbb{R}^{n_2}}\Big[\big(\mathbb{E}|Y^{x,y}_{n,m}-Y^{x,z}_{n,m}|^2\big)^{\frac{1}{2}}\big(\mathbb{E}(1+|x|^{2\theta_2}+|Y^{x,y}_{n,m}|^{2\theta_2}+|Y^{x,z}_{n,m}|^{2\theta_2})\big)^{\frac{1}{2}}\Big]\mu^{x,\Delta_2}(\mathrm{d}z).
		\end{align*}
		Under $({\bf F1})$-$({\bf F3})$ with $k\geq {2\vee2\theta_2\vee(\theta_2+1)}$, utilizing Lemmas \ref{Lb.3.10}-\ref{CL3.7} one gets
		\begin{align*}
			| \bar{b}^{\Delta_2}(x)-\mathbb{E}b(x,Y^{x,y}_{n,m})|&\leq Ce^{\frac{-\beta m \Delta_2}{4}}\int_{\mathbb{R}^{n_2}}|y-z|(1+|x|^{\theta_2}+|y|^{\theta_2}+|z|^{\theta_2})\mu^{x,\Delta_2}(\mathrm{d}z)\notag
			\\&\leq Ce^{\frac{-\beta m \Delta_2}{4}}\int_{\mathbb{R}^{n_2}}(1+|x|^{\theta_2+1}+|y|^{\theta_2+1}+|z|^{\theta_2+1})\mu^{x,\Delta_2}(\mathrm{d}z)\notag
			\\&\leq Ce^{\frac{-\beta m \Delta_2}{4}}(1+|x|^{\theta_2+1}+|y|^{\theta_2+1}).
		\end{align*}
		The proof is complete.
	\end{proof}
	\begin{lem}\label{L3.15}
		{\rm Under $({\bf S2})$, $({\bf S4})$ and $({\bf F1})$-$({\bf F3})$ with $k\geq 2\theta_2\vee 2\theta_4\vee(\theta_2+\theta_4+1)$, for any $x\in \mathbb{R}^{n_1}$, {$y_0\in \RR^{n_2}$}, $\Delta_2\in (0,\bar\Delta_2]$ and integers $n\geq 0$, $M\geq1$,
			\begin{align*}
				{\mathbb{E}|\bar{b}^{\Delta_2}(x)-B_{M}(x,Y^{x,y_0}_{n})|^2\leq C(1+|x|^{2\theta_3\vee2\theta_4\vee(\theta_2 +\theta_3\vee\theta_4+1)}+|y_0|^{2\theta_4\vee(\theta_2 +\theta_3\vee\theta_4+1)}) \frac{1}{M\Delta_{2}}}.
		\end{align*}}
	\end{lem}
	\begin{proof}\textbf {Proof.}
		In light of \eqref{3.9}, we derive that
		for any $x\in \RR^{n_1}$,
		\begin{align}\label{cyp3.44}
			\mathbb{E}\Big| \bar{b}^{\Delta_2}(x)-B_{M}(x, Y^{x,y_0}_{n})\Big|^{2}
			&=\frac{1}{M^{2}}\sum_{m, l=1}^{M}\mathbb{E} U_{m,l}
			=\frac{1}{M^{2}}\sum_{m=1}^{M}\mathbb{E} U_{m,m}+ \frac{2}{M^{2}}\sum_{l=1}^{M}\sum_{m=l+1}^{M}\mathbb{E}U_{m,l},
		\end{align}
		where
		\begin{align*}
			U_{m,l}=\Big( \bar{b}^{\Delta_2}(x)-b\big(x,Y^{x,y_0}_{n,m}\big)\Big)
			\Big(\bar{b}^{\Delta_2}(x)-b\big(x,Y^{x,y_0}_{n,l}\big)\Big).
		\end{align*}
		By $({\bf S4})$, $({\bf F1})$ and $({\bf F3})$ with $k\geq2\theta_4$, invoking   Lemma \ref{Lb.3.10} and the $\mathrm{H\ddot{o}lder}$ inequality,  one obtains that
		\begin{align*}
			\mathbb{E}\big|b(x,Y^{x,y_0}_{n,m})\big|^2  &\leq C \mathbb{E}\Big( 1+|x|^{2\theta_3}+|Y^{x,y_0}_{n,m}|^{2\theta_4}  \Big)\leq C(1+|x|^{2\theta_3}) +C\Big(\mathbb{E}|Y^{x,y_0}_{n,m}|^{k} \Big)^{\frac{2\theta_4}{k}}\notag
			\\& \leq    {C(1+|x|^{2(\theta_3\vee\theta_4)}+|y_0|^{2\theta_4}}).
		\end{align*}
		Then using the triangle inequality  along with the above inequality and \eqref{f13},  for any $m, l\geq 1$, we yield that for any $x\in \mathbb{R}^{n_1}$,
		\begin{align}\label{cyp3.48}
			\mathbb{E}|U_{m,l}|&\leq  \mathbb{E}\big|b(x,Y^{x,y_0}_{n,m})\big|^2+\mathbb{E}\big|b(x,Y^{x,y_0}_{n,l})\big|^2+2\mathbb{E}|\bar{b}^{\Delta_2}(x)|^2 \notag
			\\&\leq {C(1+|x|^{2(\theta_3\vee\theta_4)}+|y_0|^{2\theta_4})}<\infty,
		\end{align}
		which implies that $|U_{m,l}|$ is integrable with respect to $\mathbb{P}$. To compute  precisely, let $\mathcal{G}^2_{n,l}$ denote the $\sigma$-algebra generated  by  $$\Big\{ W^{2}_n(s)- W^{2}_n( l\Delta_{2}), s\geq l\Delta_{2} \Big\}$$
and $\mathcal{F}^{2}_{n,l}$ denote the $\sigma$-algebra generated by $\{W^{2}_{n}(s), 0\leq s\leq l\Delta_2\}$.
		Note that $\mathcal{F}^2_{n,l}$ and $\mathcal{G}^2_{n,l}$ are mutually independent.
		Since $Y^{x,y_0}_{n,l}$ is $\mathcal{F}^2_{n,l}$-measurable and 
		independent of $\mathcal{G}^2_{n,l}$, using the result of {\cite[p.221]{MR1368405}}, we derive that for any $x\in \mathbb{R}^{n_1}$ and  $1\leq l<m\leq M$,
		\begin{align}\label{c3.43}
			\mathbb{E}U_{m,l}
			=&\mathbb{E}\Big[\big(\bar{b}^{\Delta_2}(x)-b (x,Y^{x,y_0}_{n,l} )\big)
			\times
			\mathbb{E}\Big(\big(\bar{b}^{\Delta_2}(x)-b (x,Y^{x,y_0}_{n,m} )
			\big)\Big|\mathcal{F}^{2}_{n,l}\Big)
			\Big]\notag \\
			\leq&\mathbb{E}\bigg[\big|\bar{b}^{\Delta_2}(x)-b (x,Y^{x,y_0}_{n,l} )\big|
			\times  \Big|\bar{b}^{\Delta_2}(x)-\mathbb{E} b \big (x,Y^{x,z}_{n,m-l} \big)\Big|_{z=Y^{x,y_0}_{n,l}}  \bigg].
		\end{align}
		For any $x\in \mathbb{R}^{n_{1}}$ and $y\in \mathbb{R}^{n_{2}}$, it follows from  $({\bf S4})$ and \eqref{f13} that
		\begin{align}\label{4.21}
			|\bar{b}^{\Delta_2}(x)-b(x,y)|&=|\bar{b}^{\Delta_2}(x)|+|b(x,y)|\leq  C(1+|x|^{\theta_3\vee\theta_4}+|y|^{\theta_4}).
		\end{align}
		Owing to $({\bf S2})$, $({\bf S4})$ and $({\bf F1})$-$({\bf F3})$  with $k\geq {2\vee2\theta_2\vee(\theta_2+1)\vee\theta_4}$,
		using Lemma \ref{L3.14} derives that
		\begin{align*}
			\Big|\bar{b}^{\Delta_2}(x)-\mathbb{E}b\big(x,Y^{x,z}_{n,m-l})\Big| \leq Ce^{-\frac{\beta(m-l)\Delta_{2}}{4}} \big(1+|x|^{\theta_2+1}+|z|^{ \theta_2+1 }\big).
		\end{align*}
		Using \eqref{4.21} and substituting the above inequality  into (\ref{c3.43}) lead to that for any $x\in \mathbb{R}^{n_1}$ and $1\leq l<m\leq M$,
		\begin{align*}
			\mathbb{E}U_{m,l}&\leq  C e^{-\frac{\beta(m-l)\Delta_{2}}{4}}\E\Big[\big( 1+|x|^{\theta_3\vee\theta_4}+|Y^{x,y_0}_{n,l}|^{\theta_4}\big)\notag
			\\&~~~~~~~~~~~~~~~~~~~~~~~\times \big(1+|x|^{\theta_2+1}+|Y^{x,y_0}_{n,l}|^{\theta_2+1 }\big)\Big]\notag
			\\&\leq C e^{-\frac{\beta(m-l)\Delta_{2}}{4}}\E\Big[\big( 1+|x|^{\theta_2+\theta_3\vee\theta_4+1}+(1+|x|^{\theta_3\vee\theta_4})|Y^{x,y_0}_{n,l}|^{\theta_2+1 }\notag
			\\&~~~+(1+|x|^{\theta_2+1})|Y^{x,y_0}_{n,l}|^{\theta_4}
			+|Y^{x,y_0}_{n,l}|^{\theta_2+\theta_4+1 }\big)\Big].
		\end{align*}
		Due to $k\geq \theta_2 + \theta_4+1$, using  Lemma \ref{Lb.3.10}  we deduce that for any $1\leq l<m\leq M$,
		\begin{align}\label{cyp3.53}
			\mathbb{E}U_{m,l} \leq Ce^{-\frac{\beta(m-l)\Delta_{2}}{4}}\big(1+|x|^{\theta_2 +\theta_3\vee\theta_4+1}+{|y_0|^{\theta_2 +\theta_3\vee\theta_4+1}}\big).
		\end{align}
		Hence, inserting \eqref{cyp3.48}  with $m=l$ and \eqref{cyp3.53} into \eqref{cyp3.44} yields that
		\begin{align*}
			&\mathbb{E}\Big|\bar{b}^{\Delta_2}(x)-B_{M}(x, Y^{x,y_0}_{n})\Big|^{2}
			\leq \frac{C(1+|x|^{2(\theta_3\vee\theta_4)}+|y_0|^{2\theta_4})}{M } \notag
			\\&~~~~~+ \frac{C(1+|x|^{\theta_2 +\theta_3\vee\theta_4+1}+|y_0|^{\theta_2 +\theta_3\vee\theta_4+1})}{M^{2}}\sum_{l=1}^{M}\sum_{m=l+1}^{M}e^{-\frac{\beta(m-l)\Delta_{2}}{4}}\notag
			\\ \leq & \frac{C(1+|x|^{2(\theta_3\vee\theta_4)}+|y_0|^{2\theta_4})}{M } +\frac{C(1+|x|^{\theta_2 +\theta_3\vee\theta_4+1}+|y_0|^{\theta_2 +\theta_3\vee\theta_4+1})}{M( e^{ \beta \Delta_2/4}-1)}\notag
			\\ \leq& C\big(1+|x|^{2\theta_3\vee2\theta_4\vee(\theta_2 +\theta_3\vee\theta_4+1)}+|y_0|^{2\theta_3\vee2\theta_4\vee(\theta_2 +\theta_3\vee\theta_4+1)}\big)\Big( \frac{1}{M}+\frac{1}{M\Delta_{2}}\Big)\notag
			\\ \leq&  C\big(1+|x|^{2\theta_3\vee2\theta_4\vee(\theta_2 +\theta_3\vee\theta_4+1)}+|y_0|^{2\theta_4\vee(\theta_2 +\theta_3\vee\theta_4+1)}\big)\frac{1}{M\Delta_{2}},
		\end{align*}
		where the second to last inequality used the fact $e^{x}-1\geq x, \forall x\geq0$ and the last inequality used the fact  $1/M\leq 1/M\Delta, \forall \Delta_2\in (0,1]).$
		The proof is complete.
	\end{proof}
	
	Combining Lemmas \ref{L3.13} and \ref{L3.15},  we obtain the estimate of  $\mathbb{E}|\bar{b}(x)-B_{M}(x,Y^{x}_{n})|^2$ directly.
	\begin{lem}\label{Lb.7}
		{\rm Under $({\bf S2})$, $({\bf S4})$ and $({\bf F1})$-$({\bf F3})$ with $k\geq 2\theta_2\vee 2\theta_4\vee(\theta_2+\theta_4+1)$, for any $x\in \mathbb{R}^{n_1}$, $y_0\in \RR^{n_2}$, $\Delta_2\in (0,\bar\Delta_2]$ and integers $n\geq 0$, $M\geq1$,
			\begin{align*}				\mathbb{E}\Big|\bar{b}(x)-B_{M}(x,Y^{x,y_0}_{n})\Big|^{2}\leq C\big(1+|x|^{2\theta_3\vee2\theta_4\vee(\theta_2+\theta_3\vee\theta_4+1)}+|y_0|^{2\theta_4\vee(\theta_2+\theta_3\vee\theta_4+1)}\big)\Big(\Delta_{2}+\frac{1}{M\Delta_{2}}\Big).
		\end{align*}}
	\end{lem}

\section{Strong convergence in $p$th moment}\label{s-c5}
	\par With the help of the averaging principle, this section aims to prove the strong convergence between the slow component $x^{\varepsilon}(t)$ of original system \eqref{e1} and the numerical solution $X(t)$ generated by the MTEM scheme.
	\begin{lem}\label{la3.8}
		{\rm If $({\bf S3})$-$({\bf S5})$, $({\bf F1})$ and $({\bf F3})$  hold  with $k\geq2(\theta_4\vee2)$, then for any  $x_0\in \RR^{n_1}$, $y_0\in \RR^{n_2}$, $0< p\leq k/(\theta_4\vee2)$, $T>0$ and $M\geq1$, there exists a constant $C_{x_0,y_0,T,p}$ such that
			\begin{align*}
				\sup_{\Delta_{1}\in(0,1], \Delta_2\in (0,\hat\Delta_2]}\mathbb{E}|\bar{X}(t)|^{p}\leq C_{x_0,y_0,T,p},
			\end{align*}
			and
			\begin{align*}
				\mathbb{E}|\bar{X}(t)-X(t)|^{p}\leq C_{x_0,y_0,T,p}\Delta_1^\frac{p}{2}.
		\end{align*}}
	\end{lem}

\begin{proof}\noindent\textbf {Proof.}
For $2\leq p\leq k/(\theta_4\vee2)$, using the $\mathrm{It\hat{o}}$ formula, we deduce from \eqref{a22} that for any $0\leq t\leq T$,
		\begin{align*}
			|\bar{X} (t)|^{p} =&|x_{0}|^{p}+ p\int_{0}^{t} |\bar{X}(s)|^{p-2}\Big[ \bar{X}^{T}(s){B_{M}\Big(T_{\Delta_1}(X(s)),Y^{T_{\Delta_1}(X(s)),y_0}\Big)}
			\notag
			\\&~~~+\frac{ p-1 }{2}|\sigma(X(s))|^{2}\Big]\mathrm{d}s
			+ p \int_{0}^{t} |\bar{X}(s)|^{p-2}  \bar{X}^{T}(s)\sigma(X(s))\mathrm{d}W^{1}(s)\nn\
\\=&|x_{0}|^{p}+ p\int_{0}^{t} |\bar{X}(s)|^{p-2}\Big[ X^{T}(s)B_{M}\Big(T_{\Delta_1}(X(s)),Y^{T_{\Delta_1}(X(s)),y_0}\Big)\nn\
\\&~~~+(\bar{X}(s)-X(s))^{T}B_{M}\Big(T_{\Delta_1}(X(s)),Y^{T_{\Delta_1}(X(s)),y_0}\Big)
			\notag
			\\&~~~+\frac{ p-1 }{2}|\sigma(X(s))|^{2}\Big]\mathrm{d}s
			+ p \int_{0}^{t} |\bar{X}(s)|^{p-2}  \bar{X}^{T}(s)\sigma(X(s))\mathrm{d}W^{1}(s).
		\end{align*}
 Utilizing ({\bf{S3}}) and the  Young inequality implies that 
 	\begin{align}\label{NN5.1}
			\mathbb{E}|\bar{X}(t)|^{p}  &\leq|x_{0}|^{p}+ C_{p}\int_{0}^{t}  \mathbb{E}|\bar{X}(s)|^{p}\mathrm{d}s+C_{p}\int_{0}^{t}\mathbb{E}|X(s)|^{p}\mathrm{d}s+\mathcal{A}_1+\mathcal{A}_2.
		\end{align}
where 
\begin{align*}
\mathcal{A}_1=p\int_{0}^{t} \mathbb{E}\Big[|\bar{X}(s)|^{p-2} X^{T}(s)B_{M}\Big(T_{\Delta_1}(X(s)),Y^{T_{\Delta_1}(X(s)),y_0}\Big)\Big]\mathrm{d}s
\end{align*}
and
\begin{align*}
\mathcal{A}_2=C_p\int_{0}^{t}\mathbb{E}\Big(|\bar{X}(s)-X(s)|^{\frac{p}{2}}\big|B_{M}(T_{\Delta_1}(X(s)),Y^{T_{\Delta_1}(X(s)),y_0})\big|^{\frac{p}{2}}\Big)\mathrm{d}s.\nn\
\end{align*}
{\color{red}One observes that  for any $s\geq0$,
\begin{align*} 
X(s)=\frac{|X(s)|}{|T_{\Delta_1}(X(s))|}T_{\Delta_1}(X(s)).
\end{align*}
Furthermore, according to \eqref{3.9} it follows that
\begin{align*}
\mathcal{A}_1&=p\int_{0}^{t} \mathbb{E}\Big[|\bar{X}(s)|^{p-2} \frac{|X(s)|}{|T_{\Delta_1}(X(s))|}(T_{\Delta_1}(X(s)))^{T}B_{M}\Big(T_{\Delta_1}(X(s)),Y^{T_{\Delta_1}(X(s)),y_0}\Big)\Big]\mathrm{d}s\nn\
\\&=\frac{p}{M}\sum_{m=1}^{M}\int_{0}^{t} \mathbb{E}\Big[|\bar{X}(s)|^{p-2} \frac{|X(s)|}{|T_{\Delta_1}(X(s))|}(T_{\Delta_1}(X(s)))^{T}b\Big(T_{\Delta_1}(X(s)),Y^{T_{\Delta_1}(X(s)),y_0}_{m}\Big)\Big]\mathrm{d}s.\nn\
\end{align*} 
To simplify the notation for estimating $\mathcal{A}_1$, for any $u\geq 0$, let {$n_{\Delta_1}(u) =\lfloor u/\Delta_1\rfloor$}, which is the integer part of $u/\Delta_1$. Using ({\bf{S5}}) and the Young inequality yields that
\begin{align}\label{eqNN5.2}
\mathcal{A}_{1}&\leq \frac{p}{M}\sum_{m=1}^{M}\int_{0}^{t} \mathbb{E}\Big[|\bar{X}(s)|^{p-2} \frac{|X(s)|}{|T_{\Delta_1}(X(s))|}\Big(K_4\big(1+|T_{\Delta_1}(X(s))|^2\big) +\lambda\big|Y^{T_{\Delta_1}(X(s)),y_0}_{n_{\Delta_1}(s),m}\big|^2\Big)\Big]\mathrm{d}s.
\end{align}
One observes that
\begin{align}\label{equ1} 
|T_{\Delta_1}(X(s))|=\left\{
\begin{array}{lcl}
|X(s)|,~~~~&
  \omega\in A_{s},\\
 \varphi^{-1}(K\Delta_1^{-\frac{1}{2}}) , ~~~~&\omega\in A^{c}_{s},
  \end{array}\right.\end{align}
  where $A_{s}=\{\omega:|X(s)|\leq \varphi^{-1}(K\Delta_1^{-\frac{1}{2}})\}$.  
Thus, one has
\begin{align}\label{eqNNN5.3}
\mathcal{A}_1&\leq\frac{p}{M}\sum_{m=1}^{M}\int_{0}^{t} \mathbb{E}\Big[|\bar{X}(s)|^{p-2} \Big(K_4\big(1+|T_{\Delta_1}(X(s))|^2\big)+\lambda\big|Y^{T_{\Delta_1}(X(s)),y_0}_{n_{\Delta_1}(s),m}\big|^2\Big)I_{A_{s} }\Big]\mathrm{d}s\nn\
\\&~~~+\frac{p}{M}\sum_{m=1}^{M}\int_{0}^{t} \mathbb{E}\Big[|\bar{X}(s)|^{p-2} \frac{|X(s)|}{\varphi^{-1}(K\Delta_1^{-\frac{1}{2}})}\Big(K_4\big(1+|T_{\Delta_1}(X(s))|^2\big)+\lambda\big|Y^{T_{\Delta_1}(X(s)),y_0}_{n_{\Delta_1}(s),m}\big|^2\Big)I_{A^{c}_{s}}\Big]\mathrm{d}s\nn\
\\&\leq \mathcal{A}_{11}+\mathcal{A}_{12},
\end{align}
where
\begin{align*}
\mathcal{A}_{11}=\frac{p}{M}\sum_{m=1}^{M}\int_{0}^{t} \mathbb{E}\Big[|\bar{X}(s)|^{p-2} \Big(K_4\big(1+|T_{\Delta_1}(X(s))|^2\big)+\lambda\big|Y^{T_{\Delta_1}(X(s)),y_0}_{n_{\Delta_1}(s),m}\big|^2\Big) I_{A_{s}}\Big]\mathrm{d}s
\end{align*}
and 
\begin{align}\label{eqNN5.3}
\mathcal{A}_{12}&=\frac{p}{M}\sum_{m=1}^{M}\int_{0}^{t} \mathbb{E}\Big[|\bar{X}(s)|^{p-2} \frac{|X(s)|}{\varphi^{-1}(K\Delta_1^{-\frac{1}{2}})}\Big(K_4\big(1+|T_{\Delta_1}(X(s))|^2\big) +\lambda\big|Y^{T_{\Delta_1}(X(s)),y_0}_{n_{\Delta_1}(s),m}\big|^2\Big)I_{A^{c}_{s}}\Big]\mathrm{d}s.
\end{align}
Next, we estimate $\mathcal{A}_{11}$ and $\mathcal{A}_{22}$, respectively.
Utilizing the Young inequality and the fact $|T_{\Delta_1}(x)|\leq |x|, \forall x\in \RR^{n_1}$ one derives that
\begin{align} \label{eqNN5.2}
\mathcal{A}_{11} \leq C_{p}\Big(1+\int_{0}^{t}\mathbb{E}|\bar{X}(s)|^{p}\mathrm{d}s+ \int_{0}^{t}\mathbb{E}|X(s)|^{p}\mathrm{d}s\Big)+\frac{C_{p}}{M}\sum_{m=1}^{M}\int_{0}^{t}\mathbb{E}\big|Y^{T_{\Delta_1}(X(s)),y_0}_{n_{\Delta_1}(s),m}\big|^{p}\mathrm{d}s.
\end{align}
Since $2\leq p\leq k$, using the property of condition expectation and Lemma \ref{Lb.3.10} shows that
\begin{align*}
\mathbb{E}\big|Y^{T_{\Delta_1}(X(s)),y_0}_{n_{\Delta_1}(s),m}\big|^{p}&= \mathbb{E}\Big[\mathbb{E}\Big(\big|Y^{T_{\Delta_1}(X(s)),y_0}_{n_{\Delta_1}(s),m}\big|^{p}|
X(s) \Big)\Big]\nn\
\\&\leq
C\big(1+|y_0|^{p}+\mathbb{E}|T_{\Delta_1}(X(s))|^{p}\big)\nn\
\\&\leq C\big(1+|y_0|^{p}+\mathbb{E}|X(s)|^{p}\big).
\end{align*}
Inserting the above inequality into \eqref{eqNN5.2} implies that
\begin{align}\label{eqNN5.5}
\mathcal{A}_{11}=C_{y_0,p}t+C_{p}\int_{0}^{t}\mathbb{E}|\bar{X}(s)|^{p}\mathrm{d}s+C_{y_0,p}\int_{0}^{t}\mathbb{E}|X(s)|^{p}\mathrm{d}s.
\end{align}
Now we start estimating $\mathcal{A}_{12}$. It follows from \eqref{equ1} and \eqref{eqNN5.3} that
\begin{align*}
\mathcal{A}_{12}&\leq K_4\int_{0}^{t} \mathbb{E}\Big[|\bar{X}(s)|^{p-2} |X(s)|\Big(\frac{1}{\varphi^{-1}(K\Delta_{1}^{-\frac{1}{2}})}+|T_{\Delta_1}(X(s))|\Big)\Big]\mathrm{d}s\nn\
\\&~~~+\frac{p\lambda}{M}\sum_{m=1}^{M}\int_{0}^{t} \mathbb{E}\Big(|\bar{X}(s)|^{p-2} \frac{|X(s)|}{\varphi^{-1}(K\Delta_{1}^{-\frac{1}{2}})} \big|Y^{T_{\Delta_1}(X(s)),y_0}_{n_{\Delta_1}(s),m}\big|^2I_{A^{c}_{s}}\Big)\mathrm{d}s.\nn\
\end{align*}
Recalling the definition of  $\varphi^{-1}$, one obtains that
\begin{align}\label{eqp3.2}
		1\leq\varphi^{-1}(2)\leq   \varphi^{-1}(K)\leq \varphi^{-1}(K\Delta_1^{-\frac{1}{2}}),~~~\forall \Delta_1\in (0,1].
\end{align}
This, together with  the Young inequality, implies that
\begin{align*}
\mathcal{A}_{12}&\leq K_4\int_{0}^{t} \mathbb{E}\Big[|\bar{X}(s)|^{p-2} |X(s)|\Big(1+|T_{\Delta_1}(X(s))|\Big)\nn\
\\&~~~+\frac{p\lambda}{M}\sum_{m=1}^{M}\int_{0}^{t} \mathbb{E}\Big[|\bar{X}(s)|^{p-2} \frac{|X(s)|}{\varphi^{-1}(K\Delta_1^{-\frac{1}{2}})} \big|Y^{T_{\Delta_1}(X(s)),y_0}_{n_{\Delta_1}(s),m}\big|^2I_{A^{c}_{s}}\Big)\mathrm{d}s\nn\
\\&\leq \frac{3K_4p}{2}\int_{0}^{t}\mathbb{E}\Big[|\bar{X}(s)|^{p-2}\big(1+|X(s)|^2\big)\Big]\mathrm{d}s\nn\
\\&~~~+\frac{p\lambda}{M}\sum_{m=1}^{M}\int_{0}^{t}\mathbb{E}
\Big[|\bar{X}(s)|^{p-2}\frac{|X(s)|}{\varphi^{-1}(K\Delta_1^{-\frac{1}{2}})}\big|
Y^{T_{\Delta_1}(X(s)),y_0}_{n_{\Delta_1}(s),m}\big|^2I_{A^{c}_{s}}\Big]\mathrm{d}s.
\end{align*}
Furthermore, utilizing the Young inequality and the H\"older inequality yields that
\begin{align}\label{eq5.7}
\mathcal{A}_{12}&\leq C_{p}t+C_{p}\int_{0}^{t}\mathbb{E}|\bar{X}(s)|^{p}\mathrm{d}s+C_{p}\int_{0}^{t}\mathbb{E}|X(s)|^{p}\mathrm{d}s\nn\
\\&~~~+\frac{C_{p}\lambda}{[\varphi^{-1}(K\Delta_1^{-\frac{1}{2}})]^{\frac{p}{2}}}
\frac{1}{M}\sum_{m=1}^{M}\int_{0}^{t}\big(\mathbb{E}|X(s)|^{p}\big)^{\frac{1}{2}}\Big(\mathbb{E}
\big(\big|Y^{T_{\Delta_1}(X(s)),y_0}_{n_{\Delta_1}(s),m}\big|^{2p}I_{A^{c}_{s}}\big)\Big)^{\frac{1}{2}}\mathrm{d}s.
\end{align}
Owing to $2p\leq k$, using the property of condition expectation and Lemma \ref{Lb.3.10} leads to
\begin{align*}
\mathbb{E}\Big(\big|Y^{T_{\Delta_1}(X(s)),y_0}_{n_{\Delta_1}(s),m}\big|^{2p}I_{A^{c}_{s}}\Big)&= \mathbb{E}\Big[\mathbb{E}\big(\big|Y^{T_{\Delta_1}(X (s) ),y_0}_{n_{\Delta_1}(s),m}
\big|^{2p}I_{A^{c}_{s}}\big|X (s) \big)\Big]\nn\ \\ 
&\leq C\Big(1+|y_0|^{2p}+\mathbb{E}\big(|T_{\Delta_1}(X (s) )|^{2p}I_{A^{c}_{s}}\big)\Big).\nn\
\end{align*}
Inserting the above inequality into \eqref{eq5.7} and using the Young inequality and \eqref{equ1} imply  that
\begin{align}\label{eqNN5.8}
\mathcal{A}_{12}&\leq C_{y_0,p}t+C_{p}\int_{0}^{t}\mathbb{E}|\bar{X}(s)|^{p}\mathrm{d}t+C_{y_0,p}\int_{0}^{t}\mathbb{E}|X(s)|^{p}\mathrm{d}t\nn\
\\&~~~+\frac{C_{y_0,p}\lambda}{[\varphi^{-1}(K\Delta_1^{-\frac{1}{2}})]^{\frac{p}{2}}}
\int_{0}^{t}\big(\mathbb{E}|X(s)|^{p}\big)^{\frac{1}{2}}\big(\mathbb{E}\big(|T_{\Delta_1}(X(s))|^{2p}I_{A^{c}_{s}}\big)\big)^{\frac{1}{2}}\mathrm{d}s\nn\
\\&\leq C_{y_0,p}t+C_{p}\int_{0}^{t}\mathbb{E}|\bar{X}(s)|^{p}\mathrm{d}t+C_{y_0,p}\int_{0}^{t}\mathbb{E}|X(s)|^{p}\mathrm{d}t\nn\
\\&~~~+C_{y_0,p}\lambda\int_{0}^{t}\big(\mathbb{E}|X(s)|^{p}\big)^{\frac{1}{2}}
\big(\mathbb{E}|T_{\Delta_1}(X(s))|^{p}\big)^{\frac{1}{2}}\mathrm{d}s\nn\
\\&\leq C_{y_0,p}t+C_{p}\int_{0}^{t}\mathbb{E}|\bar{X}(s)|^{p}\mathrm{d}t+
C_{y_0,p}\int_{0}^{t}\mathbb{E}|X(s)|^{p}\mathrm{d}s.
\end{align}
Inserting  \eqref{eqNN5.5} and \eqref{eqNN5.8} into \eqref{eqNNN5.3} one obtains that
\begin{align}\label{eqNN5.9}
\mathcal{A}_{1}=C_{y_0,p}t+C_{p}\int_{0}^{t}\mathbb{E}|\bar{X}(s)|^{p}\mathrm{d}s+
C_{y_0,p}\int_{0}^{t}\mathbb{E}|X(s)|^{p}\mathrm{d}s.
\end{align}
}
 Then by \eqref{a22}, using the H\"older inequality yields that for any $0\leq s\leq T$,
\begin{align}\label{eqN5.12}
&\mathbb{E}|\bar{X}(s)-X(s)|^{p}=\mathbb{E}|\bar{X}(s)-X_{n_{\Delta_1}(s)}|^{p}\nn\
\\=&C_{p}\Delta_1^{p}\mathbb{E}\big|B_{M}(T_{\Delta_1}(X_{n_{\Delta_1}(s)}),Y^{T_{\Delta_1}(X_{n_{\Delta_1}(s)}),y_0})\big|^{p}+C_p\Delta_1^{\frac{p}{2}}\mathbb{E}|\sigma(X_{n_{\Delta_1}(s)})|^{p}.
\end{align}
Thanks to $2\leq p\leq k/\theta_4$, by virtue of Lemma \ref{L4.1} and  ({\bf{S3}}) one derives that
\begin{align}\label{eqN5.13}
&\mathbb{E}\Big|B_{M}\big(T_{\Delta_1}(X_{n_{\Delta_1}(s)}),Y^{T_{\Delta_1}(X_{n_{\Delta_1}(s)}),y_0}\big)\Big|^{p}=\mathbb{E}\Big[\mathbb{E}\Big|B_{M} \big(T_{\Delta_1}(x),Y^{T_{\Delta_1}(x),y_0}_{n_{\Delta_1}(s)}\big)\Big|^{p}\Big|_{x=X_{n_{\Delta_1}(s)}}\Big]\nn\
\\ \leq& C_{y_0,p,K}\Delta_{1}^{-\frac{p}{2}}(1+\mathbb{E}|X_{n_{\Delta_1}(s)}|^{p})=C_{y_0,p,K}\Delta_{1}^{-\frac{p}{2}}(1+\mathbb{E}|X(s)|^{p}),
\end{align}
and 
\begin{align}\label{eqqN5.4}
\mathbb{E}\big|\sigma(X_{n_{\Delta_1(s)}})\big|^{p}\leq C_{p}\big(1+\mathbb{E}|X_{n_{\Delta_1(s)}}|^{p}\big)=C_{p}\big(1+\mathbb{E}|X(s)|^{p}\big).
\end{align}
Then combining \eqref{eqN5.12}-\eqref{eqqN5.4} implies that
\begin{align}\label{eqNN5.4}
\mathbb{E}|\bar{X}(s)-X(s)|^{p}=C_{y_0,p,K}\Delta_{1}^{\frac{p}{2}}(1+\mathbb{E}|X(s)|^{p}).
\end{align}
Using the H\"older inequality and then employing \eqref{eqN5.13} and \eqref{eqNN5.4} one deduces that
\begin{align}\label{NN5.5}
\mathcal{A}_2\leq& \int_{0}^{t}\mathbb{E}\Big(|\bar{X}(s)-X(s)|^{p}\Big)^{\frac{1}{2}}\Big(\mathbb{E}\big|B_{M}(T_{\Delta_1}(X(s)),Y^{T_{\Delta_1}(X(s)),y_0})\big|^{p}\Big)^{\frac{1}{2}}\mathrm{d}s\nn\
\\=&\int_{0}^{t}\mathbb{E}\Big(|\bar{X}(s)-X(s)|^{p}\Big)^{\frac{1}{2}}\Big(\mathbb{E}\big|B_{M}(T_{\Delta_1}(X_{n_{\Delta}(s))},Y^{T_{\Delta_1}(X_{n_{\Delta_1}(s)}),y_0})\big|^{p}\Big)^{\frac{1}{2}}\mathrm{d}s\nn\
\\\leq& C_{y_0,p,K}\int_{0}^{t}(1+\mathbb{E}|X(s)|^{p})\mathrm{d}s.
\end{align}
Inserting \eqref{eqNN5.9} and \eqref{NN5.5} into \eqref{NN5.1} gives that for any $T>0$
\begin{align}\label{eq5.12}
\sup_{0\leq t\leq T}\mathbb{E}|\bar{X}(t)|^{p}  &\leq|x_{0}|^{p}+ C_{y_0,p}\int_{0}^{T}  \sup_{0\leq r\leq s}\mathbb{E}|\bar{X}(r)|^{p}\mathrm{d}s+C_{y_0,p}\int_{0}^{T}\sup_{0\leq r\leq s}\mathbb{E}|X(r)|^{p}\mathrm{d}s+C_{y_0,p}T\nn\
\\&\leq |x_{0}|^{p}+ C_{y_0,p}\int_{0}^{T}  \sup_{0\leq r\leq s}\mathbb{E}|\bar{X}(r)|^{p}\mathrm{d}s+C_{y_0,p}T.
\end{align}
By virtue of the Gronwall inequality one derives that there exists a constant $C_{x_0,y_0,p,T}$ such that
\begin{align*}
\sup_{0\leq t\leq T}\mathbb{E}|\bar{X}(t)|^{p}\leq  C_{x_0,y_0,p,T}.
\end{align*}
Then the second assertion holds directly by substituting the above result into \eqref{eqNN5.4}. The case $0<p<2$ follows   directly by  using the H\"older inequality. 
\end{proof}

To prove the strong convergence of the MTEM scheme \eqref{e5}, we introduce  an auxiliary TEM numerical scheme
	for the averaged equation \eqref{eq2.4} 
\begin{equation*}
	\begin{cases}
		Z_{0}=x_{0},\\
		Z_{n+1}=Z_{n}+\bar{b}(T_{\Delta_1}(Z_{n}))\Delta_{1}+\sigma(Z_{n})\Delta W^{1}_{n},
	\end{cases}
\end{equation*}
and the  corresponding continuous-time   processes
$$
Z(t) =Z_{n},~~~~~~~ t\in[n\Delta_1,(n+1)\Delta_1),
$$
and~~~~ 
\begin{align}\label{EQN5.9}
\bar{Z}(t) =x_{0}+\int_{0}^{t}\bar{b}(T_{\Delta_1}(Z(s)))\mathrm{d}s
		+\int_{0}^{t}\sigma(Z(s))\mathrm{d}W^{1}(s).
\end{align}
	One observes that $\bar{Z}(n\Delta_1)= {Z}(n\Delta_1)=Z_{n} $. In what follows,  we analyze the strong   errors  $\mathbb{E}\Big(\sup_{0\leq t\leq T}|\bar{x}(t)-\bar{Z}(t)|^{2}\Big)$ and $\mathbb{E}\Big(\sup_{0\leq t\leq T}|\bar{Z}(t)-X(t)|^{2}\Big)$, respectively. To proceed we give the bound   of the $p$th moment of $\bar{Z}(t)$.
	\par  Now we show  that the Khasminskii-like condition is preserved for the modified coefficient $\bar{b}(T_{\Delta_1}(x))$, 
which is used to obtain   the moment bound of an important  auxiliary process $\bar{Z}(t)$ defined  by \eqref{EQN5.9}.
\begin{lem}\label{le3.1}
	{\rm If $({\bf S4})$, $({\bf S5})$ and $({\bf F1})$-$({\bf F3})$ hold with $k\geq{2\vee\theta_4}$, then for any $\Delta_{1}\in(0,1]$,
		\begin{align*}
			x^{T}\bar{b}(T_{\Delta_1}(x))\leq C(1+|x|^{2}),~~~~x\in \mathbb{R}^{n_1}.
	\end{align*}}
\end{lem}
\begin{proof}\textbf {Proof.}
	For $x\in \mathbb{R}^{n_{1}}$ with $|x|\leq \varphi^{-1}(K\Delta_1^{-1/2})$, $x=T_{\Delta_1}(x)$. 
	Using $({\bf S5})$ implies that
	\begin{align*}
		x^{T}\bar{b}(T_{\Delta_1}(x))&= x^{T}\int_{\mathbb{R}^{n_{2}}} b(x,y)\mu^{x}(\mathrm{d}y)\notag
		\leq K_4(1+|x|^{2})+ \lambda\int_{\mathbb{R}^{n_{2}}}|y|^{2}\mu^{x}(\mathrm{d}y).
	\end{align*}
	Using the H\"older inequality and Lemma \ref{Lcyp2.1}, one obtains
	\begin{align}\label{f8}
		x^{T}\bar{b}(T_{\Delta_1}(x))\leq C(1+|x|^{2}).
	\end{align}
	For any $x\in \mathbb{R}^{n_{1}}$ with $|x|> \varphi^{-1}\big(K\Delta_1^{-1/2}\big)$, it follows from the definition of $T_{\Delta_1}(x)$ that $x= \Big({|x|}/{\varphi^{-1}(K\Delta_1^{-{1}/{2}})}\Big)T_{\Delta_1}(x)$. This, together with
	\eqref{f8}, implies that
	\begin{align*}	x^{T}\bar{b}(T_{\Delta_1}(x))&=\frac{|x|}{\varphi^{-1}(K\Delta_1^{-\frac{1}{2}})}
		(T_{\Delta_1}(x))^{T}\bar{b}(T_{\Delta_1}(x))\leq \frac{|x|}{\varphi^{-1}(K\Delta_1^{-\frac{1}{2}})}C(1+|T_{\Delta_1}(x)|^{2}).
	\end{align*}
Thus, applying \eqref{eqp3.2} and the Young inequality leads to
\begin{align*}
x^{T}\bar{b}(T_{\Delta_1}(x))&\leq C|x|\Big(\big(\varphi^{-1}(K)\big)^{-1}+|x|\Big)\leq C|x|\Big(\big(\varphi^{-1}(2)\big)^{-1}+|x|\Big)\nn\
\\&\leq C(1+|x|^2).
\end{align*}
This, together with \eqref{f8}, implies the desired result.
\end{proof}
	\begin{lem}\label{L3.2}
		{\rm If $({\bf S3})$-$({\bf S5})$ and $({\bf F1})$-$({\bf F3})$ hold with {$k\geq 2\vee\theta_4$}, then for any  $x_0\in \RR^{n_1}$,  $p>0$ and $T>0$,
			\begin{align*}
				\sup_{\Delta_1\in(0,1]}\mathbb{E}\Big(\sup_{0\leq t\leq T}|\bar{Z}(t)|^{p}\Big)\leq {C_{x_0,T,p}},
			\end{align*}
			and
			\begin{align*}
				\sup_{0\leq t\leq T}\mathbb{E}|\bar{Z}(t)-Z(t)|^2\leq {C_{x_0,T,p}}\Delta_1.
		\end{align*}}
	\end{lem}
	\begin{proof}\textbf {Proof.}
We begin with dealing with the case $p\geq 2$. Applying the It\^o formula yields that
\begin{align*}
|\bar{Z}(t)|^{p}&=|x_0|^{p}+p\int_{0}^{t} |\bar{Z}(s)|^{p-2}\Big[ \bar{Z}^{T}(s)\bar{b}(T_{\Delta_1}(Z(s)))+\frac{ p-1 }{2}|\sigma(Z(s))|^{2}\Big]\mathrm{d}s\nn\
\\&~~~+ p \int_{0}^{t} |\bar{Z}(s)|^{p-2}  \bar{Z}^{T}(s)\sigma(Z(s))\mathrm{d}W^{1}(s)\nn\
\\&=|x_{0}|^{p}+ p\int_{0}^{t} |\bar{Z}(s)|^{p-2}\Big[Z^{T}(s)\bar{b}(T_{\Delta_1}(Z(s)))+\frac{ p-1 }{2}|\sigma(Z(s))|^{2}\nn\
\\&~~~+(\bar{Z}(s)-Z(s))^{T}\bar{b}(T_{\Delta_1}(Z(s)))\Big]\mathrm{d}s
			+ p \int_{0}^{t} |\bar{Z}(s)|^{p-2}  \bar{Z}^{T}(s)\sigma(Z(s))\mathrm{d}W^{1}(s).
\end{align*}
  Under $({\bf S5})$ and $({\bf F1})$-$({\bf F3})$ with $k\geq 2\vee\theta_4$,  employing the result of Lemma \ref{le3.1}  and the Burkholder-Davis-Gundy inequality \cite[ p.40, Theorem7.2]{MR2380366}, one derives that for $p\geq2$ and $T>0$,
		\begin{align}
			\mathbb{E}\Big(\sup_{0\leq t\leq T}|\bar{Z}(t)|^{p}\Big)&\leq |x_{0}|^{p}+ \frac{pC}{2}\mathbb{E}\int_{0}^{T}\big|\bar{Z}(t)\big|^{p-2}\big(1+|Z(t)|^{2}\big)\mathrm{d}t\notag
			\\&~~~+p\mathbb{E}\int_{0}^{T}|\bar{Z}(t)|^{p-2}|\bar{Z}(t)-Z(t)| |\bar{b}(T_{\Delta_1}(Z(t))|\mathrm{d}t\notag
			\\&~~~+4\sqrt{2}p\mathbb{E}\Big(\int_{0}^{T}\big|\bar{Z}(t)\big|^{2p-2}\big|\sigma(Z(t))|^{2}\mathrm{d}t\Big)^{\frac{1}{2}}.\notag
		\end{align}
		Then by the Young inequality  we obtain that for any $T>0$,
		\begin{align}\label{f1}
			&\mathbb{E}\Big(\sup_{0\leq t\leq T}|\bar{Z}(t)|^{p}\Big)\notag
\\ \leq& |x_0|^{p}+C\int_{0}^{T}\mathbb{E}\Big(\sup_{0\leq s\leq t}|\bar{Z}(s)|^{p}\Big)\mathrm{d}t+C\int_{0}^{T}\mathbb{E}\Big(|\bar{Z}(t)-Z(t)|^{\frac{p}{2}}|\bar{b}(T_{\Delta_1}(Z(t)))|^{\frac{p}{2}}\Big)\mathrm{d}t\notag
			\\&~~~+\frac{1}{2}\mathbb{E}\Big(\sup_{0\leq t\leq T}|\bar{Z}(t)|^{p}\Big)+C\mathbb{E}\Big(\int_{0}^{T}|\sigma(Z(t))|^{2}\mathrm{d}t\Big)^{\frac{p}{2}}.
		\end{align}
		For any $0\leq t\leq T$, due to $({\bf S4})$, $({\bf F1})$-$({\bf F3})$ with $k\geq \theta_4$, (\ref{c3.12}) hold.  Then using \eqref{c3.12} and $({\bf S3})$ yields that
		\begin{align}\label{a3.20}
			&\mathbb{E}\big|\bar{Z}(t)-Z(t)\big|^{p}=\mathbb{E}\big|\bar{Z}(t)-Z_{n_{\Delta_1}(t)}\big|^{p}\notag
			\\ \leq& 2^{p-1}\Big(\mathbb{E}\Big|\int_{n_{\Delta_1}(t)\Delta_1}^{t}\bar{b}(T_{\Delta_1}(Z(s)))\mathrm{d}s\Big|^{p}+\mathbb{E}\Big|\int_{n_{\Delta_1}(t)\Delta_1}^{t}\sigma(Z(s))\mathrm{d}W^{1}(s)\Big|^{p}\Big)\notag
			\\ \leq& 2^{p-1}\Big(\Delta_{1}^{p}\mathbb{E}\big|\bar{b}(T_{\Delta_1}(Z_{n_{\Delta_1}(t)}))\big|^{p}+\Delta_{1}^{\frac{p}{2}}\mathbb{E}|\sigma(Z_{n_{\Delta_1}(t)})|^{p}\Big)\notag
			\\ \leq& C\Delta_{1}^{\frac{p}{2}}\big(1+\mathbb{E}|Z_{n_{\Delta_1}(t)}|^{p}\big)\leq C\Delta_{1}^{\frac{p}{2}}(1+\mathbb{E}|Z(t)|^{p}).
		\end{align}
		Then utilizing  \eqref{c3.12} again and the $\mathrm{H\ddot{o}lder}$ inequality implies that
		\begin{align}\label{cypp3.15}
			\mathbb{E}\Big(|\bar{Z}(t)-Z(t)|^{\frac{p}{2}}|\bar{b}
			(T_{\Delta_1}(Z(t)))|^{\frac{p}{2}}\Big)
			& \leq\big(\mathbb{E}|\bar{Z}(t)-Z(t)|^{p}\big)^{\frac{1}{2}}\big(\mathbb{E}|\bar{b}
			(T_{\Delta_1}(Z(t)))|^{p}\big)^{\frac{1}{2}}\notag
			\\&\leq C\Delta_{1}^{\frac{p}{4}}\big(1+\mathbb{E}|Z(t)|^{p}\big)^{\frac{1}{2}}\Delta_1^{-\frac{p}{4}}\big(1+\mathbb{E}|Z(t)|^{p}\big)^{\frac{1}{2}} \notag
			\\&\leq C + \mathbb{E}\Big(\sup_{0\leq s\leq t}|\bar{Z}(s)|^{p}\Big).
		\end{align}
		Applying $({\bf S3})$ and the H\"older inequality we get
		\begin{align}\label{f9}
			\mathbb{E}\Big(\int_{0}^{T}\big|\sigma(Z(t))|^{2}\mathrm{d}t\Big)^{\frac{p}{2}}
			\leq C_{T,p}+C_{T,p}\int_{0}^{T}\mathbb{E}\Big(\sup_{0\leq s\leq t}|\bar{Z}(s)|^{p}\Big)\mathrm{d}t.
		\end{align}
		Hence, substituting \eqref{cypp3.15} and (\ref{f9}) into \eqref{f1}  yields that
		\begin{align}
			&\mathbb{E}\Big(\sup_{0\leq t\leq T}|\bar{Z}(t)|^{p}\Big)
			\leq  |x_0|^{p}+C_{T,p}+C_{T,p}\int_{0}^{T}\mathbb{E}\Big(\sup_{0\leq s\leq t}|\bar{Z}(s)|^{p}\Big)\mathrm{d}t .\notag
		\end{align}
		An application of the $\mathrm{Gronwall}$ inequality gives that
		\begin{align*}
			\mathbb{E}\Big(\sup_{0\leq t\leq T}|\bar{Z}(t)|^{p}\Big)\leq C_{x_0,T,p}.
		\end{align*}
		Then inserting the above inequality into \eqref{a3.20} implies that the another desired assertion holds. The corresponding results of case $0< p<2$ follows directly from that of the case  $2\leq p$ by the H\"older inequality.  The proof is complete.
	\end{proof}

{\rm To prove the strong convergence between $\bar{x}(t)$ and $\bar{Z}(t)$, as well as between $\bar{Z}(t)$ and $\bar{X}(t)$, for any $R>|x_{0}|$ and $\Delta_1\in (0,1]$, define the stopping times for the processes $\bar{x}(t)$,  $\bar{Z}(t)$ and $\bar{X}(t)$, respectively.}
\begin{align}\label{eqNN3.22}
\tau_{R}=\inf\{t\geq 0:|\bar{x}(t)|\geq R\},
\end{align}
\begin{align}\label{3.22}
				\rho_{\Delta_1,R}:=\inf\{t\geq0: |\bar{Z}(t)|\geq R\},
			\end{align} and
\begin{align}\label{3.47}
\bar{\rho}_{\Delta_1,R}= \inf\{t\geq 0:|\bar{X}(t)|\geq R\}.
\end{align}
\begin{rem}\label{rem5.1}
By virtue of  Lemma \ref{L2} one  deduces that for any $T>0$,
\begin{align*}
R^{p}\mathbb{P}(\tau_{R}\leq T)\leq \mathbb{E}|\bar{x}(T\wedge\tau_{R})|^{p}\leq \mathbb{E}\Big(\sup_{0\leq t\leq T}|\bar{x}(t)|^{p}\Big)\leq C_{x_0,T,p},
\end{align*}
which implies that
\begin{align*}
\mathbb{P}(\tau_{R}\leq T)\leq{ \frac{C_{x_0,T,p}}{R^{p}}},~~~\forall T>0.
\end{align*}
\end{rem}
\begin{rem}\label{rem5.2}
By  an argument similar to that of Remark \ref{rem5.1},  applying Lemma \ref{L3.2} one derives that
			\begin{align*}
			\mathbb{P}(\rho_{\Delta_1,R}\leq T)\leq  \frac{C_{x_0,T,p}}{R^{p}},~~~\forall T>0.
			\end{align*}
\end{rem}

\begin{rem}\label{rem5.4}

By an argument similar to \eqref{eq5.12} we deduce that for any $T>0$ and $R>0$,
\begin{align}
\mathbb{E}|\bar{X}(t\wedge\bar{\rho}_{\Delta_1,R})|^{p}  &\leq|x_{0}|^{p}+ C_{y_0,p}\mathbb{E}\int_{0}^{T\wedge\bar{\rho}_{\Delta_1,R}}  |\bar{X}(s)|^{p}\mathrm{d}s+C_{y_0,p}\mathbb{E}\int_{0}^{T\wedge\bar{\rho}_{\Delta_1,R}}|X(s)|^{p}\mathrm{d}s+C_{y_0,p}T\nn\
\\&\leq |x_{0}|^{p}+ C_{y_0,p}\int_{0}^{T}  \mathbb{E}|\bar{X}(s\wedge\bar{\rho}_{\Delta_1,R})|^{p}\mathrm{d}s+C_{y_0,p}\int_{0}^{T}\mathbb{E}|X(s)|^{p}\mathrm{d}s+C_{y_0,p}T.\nn\
\end{align}
Employing the result of Lemma \ref{la3.8} yields that
\begin{align*}
\mathbb{E}|\bar{X}(t\wedge\bar{\rho}_{\Delta_1,R})|^{p}  &\leq|x_{0}|^{p}+ C_{y_0,p}\int_{0}^{T}  \mathbb{E}|\bar{X}(s\wedge\bar{\rho}_{\Delta_1,R})|^{p}\mathrm{d}s+C_{y_0,p}T.\nn\
\end{align*}
Then applying the Gronwall inequality gives  that there exists a constant $C_{y_0,p,T}$ such that
\begin{align*}
\mathbb{E}|\bar{X}(t\wedge\bar{\rho}_{\Delta_1,R})|^{p}  &\leq C_{x_0,y_0,p,T},
\end{align*}
which implies that
\begin{align*}
\mathbb{P}(\bar{\rho}_{\Delta_1,R}\leq T)\leq \frac{C_{x_0,y_0,T,p}}{R^{p}}.
\end{align*}
\end{rem}

\begin{lem}\label{3.8}
		{\rm If $({\bf S1})$-$({\bf S5})$ and $({\bf F1})$-$({\bf F3})$ hold with $k>{2\vee\theta_1\vee2\theta_2\vee\theta_4}$,  then for any $T>0$,
			\begin{align*}
				\lim_{\Delta_{1}\rightarrow 0}\mathbb{E}\Big(\sup_{0\leq t\leq T}|\bar{x}(t)-\bar{Z}(t)|^{2}\Big)=0.
	\end{align*}}
	\end{lem}}
\begin{proof}\noindent\textbf {Proof.}
Let $e(t)=\bar{x}(t)-\bar{Z}(t)$ for any $t\geq 0$. Define the stopping time $$\theta_{\Delta_1,R}=\tau_{R}\wedge\rho_{\Delta_1, R},$$ where $\tau_{R}$ and $\rho_{\Delta_1, R}$ are defined in \eqref{eqNN3.22} and \eqref{3.22}. Employing the young inequality, we derive that for any $p>2$ and $\delta>0$,
\begin{align*}
\mathbb{E}\Big[\sup_{0\leq t\leq T}|e(t)|^2\Big]&=\mathbb{E}\Big[\sup_{0\leq t\leq T}|e(t)|^{2}I_{\theta_{\Delta_1,R}>T}\Big]+\mathbb{E}\Big[\sup_{0\leq t\leq T}|e(t)|^{2}I_{\theta_{\Delta_1,R}\leq T}\Big]\nn\
\\&\leq \mathbb{E}\Big[\sup_{0\leq t\leq T}|e(t\wedge \theta_{\Delta_1,R})|^{2}\Big]+\frac{2\delta}{p}\mathbb{E}\Big[\sup_{0\leq t\leq T}|e(t)|^{p}\Big]+\frac{p-2}{p\delta^{\frac{p-2}{2}}}\mathbb{P}(\theta_{\Delta_1,R}\leq T).\nn\
\end{align*}
Under {\color{red}Assumptions} $({\bf S1})$-$({\bf S3})$, $({\bf S5})$ and $({\bf F1})$-$({\bf F3})$ with $k\geq {2}\vee\theta_{1}\vee2\theta_2\vee\theta_4$, using Lemmas \ref{L2} and \ref{L3.2}, we derive that
\begin{align*}
\mathbb{E}\Big[\sup_{0\leq t\leq T}|e(t)|^{p}\Big]\leq \mathbb{E}\Big[\sup_{0\leq t\leq T}|\bar{x}(t)|^{p}\Big]+\mathbb{E}\Big[\sup_{0\leq t\leq T}|\bar{Z}(t)|^{p}\Big]\leq C_{x_0,T,p}.
\end{align*}
Furthermore, applying Remarks \ref{rem5.1} and \ref{rem5.2} one gets
\begin{align*}
\mathbb{P}(\theta_{\Delta_1,R}\leq T)\leq \mathbb{P}(\tau_{R}\leq T)+\mathbb{P}(\rho_{\Delta_1,R}\leq T)\leq \frac{C_{x_0,T,p}}{R^{p}}.
\end{align*}
Hence, one obtains that
\begin{align}\label{eqN5.26}
\mathbb{E}\Big[\sup_{0\leq t\leq T}|e(t)|^{p}\Big]\leq \mathbb{E}\Big[\sup_{0\leq t\leq T}|e(t\wedge \theta_{\Delta_1,R})|^{2}\Big]+\frac{2C_{x_0,T,p}\delta}{p}+\frac{C_{x_0,T,p}(p-2)}{pR^{p}\delta^{\frac{p-2}{2}}}.
\end{align}
		Fix any constant $R>|x_{0}|$, define the truncated coefficients
		$$\bar{b}_{R}(x)=\bar{b}\Big((|x|\wedge R)\frac{x}{|x|}\Big),~~~\sigma_{R}(x)=\sigma\Big((|x|\wedge R)\frac{x}{|x|}\Big).$$
		Consider the following SDE
		\begin{align}\label{ee9}
			\mathrm{d}\bar{u}_{R}(t)=\bar{b}_{R}(\bar{u}_{R}(t))\mathrm{d}t
			+\sigma_{R}(\bar{u}_{R}(t))\mathrm{d}W^1(t)
		\end{align}
		with initial value $\bar{u}(0)=x_{0}$.  Under $({\bf S1})$, $({\bf S2})$ and $({\bf F1})$-$({\bf F3})$ with  $k\geq\theta_1\vee2\theta_2$, one observes from $({\bf S1})$ and the result of Lemma \ref{L3.3} that both $\bar{b}_{R}(x)$ and $\sigma_{R}(x)$ are global Lipschitz continuous with the Lipschitz constant dependent on $R$. Thus equation (\ref{ee9}) has a unique global solution $\bar u_{R}(t)$ on $t\geq0$. Let $ \bar{U}_{R}(t) $ denote the continuous extension of the EM  numerical solution of  (\ref{ee9}). It is well known \cite{MR1949404,Kloeden} that
		\begin{align}\label{eq5.27}
			\mathbb{E}\Big(\sup_{t\in[0,T]}|\bar{u}_{R}(t)- \bar{U}_{R}(t)|^{2}\Big)\leq C_{T,R}\Delta_{1},~~~\forall~T>0.
		\end{align}
		Furthermore, choose a constant $\bar\Delta_1\in (0,1]$  sufficiently small such that $\varphi^{-1}(K\bar\Delta_1^{-1/2}) \geq R$. One observes that for any $\Delta_1\in (0,\bar\Delta_1]$
		\begin{align}
			\bar{b}_{R}(x)=\bar{b}(x)=\bar{b}(x^*),~~~~\forall ~x\in \mathbb{R}^{n_1}~\hbox{with}~|x|\leq R.\notag
		\end{align}
		Then it is straightforward to see that that for any $t\geq0$
		\begin{align}
			\bar{x}(t\wedge \tau_{R})=\bar{u}_{R}(t\wedge\tau_{R}),~~~~ \bar{Z}(t\wedge\rho_{\Delta_1, R})= \bar{U}_{R}(t\wedge\rho_{ \Delta_1,R}),~~\mathrm{a.s.}.\notag
		\end{align}
Thus, we also have
$$\bar{x}(t\wedge \theta_{\Delta_1,R})=\bar{u}_{R}(t\wedge\theta_{\Delta_1,R}),~~~~ \bar{Z}(t\wedge\theta_{\Delta_1, R})= \bar{U}_{R}(t\wedge\theta_{ \Delta_1,R}),~~\mathrm{a.s.}.$$
This, together with \eqref{eq5.27}, implies that
\begin{align}\label{eqN5.28}
\mathbb{E}\Big[\sup_{0\leq t\leq T}|e(t\wedge \theta_{\Delta_1,R})|^{2}\Big]&=\mathbb{E}\Big[\sup_{0\leq t\leq T}|\bar{u}_{R}(t\wedge \theta_{\Delta_1,R})-\bar{U}_{R}(t\wedge \theta_{\Delta_1,R})|^{2}\Big]\nn\
\\&
=\mathbb{E}\Big(\sup_{t\in[0,T]}|\bar{u}_{R}(t)- \bar{U}_{R}(t)|^{2}\Big)\leq C_{T,R}\Delta_{1},~~~\forall~T>0.
\end{align}
Inserting \eqref{eqN5.28} into \eqref{eqN5.26} leads to
\begin{align*}
\mathbb{E}\Big[\sup_{0\leq t\leq T}|e(t)|^{p}\Big]\leq C_{T,R}\Delta_1+\frac{2C_{x_0,T,p}\delta}{p}+\frac{C_{x_0,T,p}(p-2)}{pR^{p}\delta^{\frac{p-2}{2}}}.
\end{align*}
Now, for any  $\epsilon>0$, choose $\delta>0$  sufficiently small such that $ {C_{x_0,T,p}\delta }/{p}\leq  {\epsilon}/{3}$. For this $\delta$, choose $R>0$ large enough such that
		$
		C_{x_0,T,p}(p-2)/ (pR^{p}\delta^{\frac{2}{p-2}}) \leq {\epsilon}/{3}.
		$
Then for the fixed $\delta>0$ and $R>0$, choose $\hat{\Delta}_1$  sufficiently small such that $C_{x_0,T,p}\hat{\Delta}_1\leq  {\epsilon}/{3}$. Hence, one derives that for the chosen $\delta>0$, $R>0$ and $\Delta_1\in (0,\hat{\Delta}_1]$,
$$\mathbb{E}\Big[\sup_{0\leq t\leq T}|e(t)|^{p}\Big]\leq \epsilon,$$
which implies the desired result.
	\end{proof}
	
	Then we turn to prove the strong convergence  of the   auxiliary process $\bar{Z}(t)$ and the MTEM numerical solution $X(t)$.
	By virtue of  Lemma \ref{la3.8}, we only need to prove the strong convergence of  $\bar{Z}(t)$ and $\bar{X}(t)$.
	
	\begin{lem}\label{Lb.13+}
		{\rm If $({\bf S1})$-$({\bf S5})$ and $({\bf F1})$-$({\bf F3})$ hold with $k\geq \theta_1\vee2(\theta_4\vee2)\vee2\theta_2\vee(\theta_2+\theta_4+1)$, for any $T>0$ and $\Delta_1\in(0,1]$,
			\begin{align*}
					\lim_{\Delta_{2}\rightarrow0}\lim_{M\rightarrow\infty}\mathbb{E}\Big(\sup_{0\leq t\leq T}|\bar{Z}(t)-\bar{X}(t)|^{2}\Big)=0.
		\end{align*}}
	\end{lem}
	\begin{proof}\textbf {Proof.}
		Define  
\begin{align}\label{NN5.28}
\bar{e}(t)=\bar{Z}(t)-\bar{X}(t)
\end{align} 
for any $t\geq 0$ and
		$
		\beta_{\Delta_1,R}= \rho_{\Delta_1,R}\wedge\bar{\rho}_{\Delta_1,R}
		$ for any $R>0$,
		where $\rho_{\Delta_1,R}$ and $\bar{\rho}_{\Delta_1,R}$ are given by (\ref{3.22}) and (\ref{3.47}), respectively. Due to $k>2(\theta_4\vee2)$, let  $2 < p\leq k/(\theta_4\vee2)$. For any $\delta>0$, using the $\mathrm{Young}$ inequality yields that for any  $2 < p\leq k/(\theta_4\vee2)$,
		\begin{align*}
			&\mathbb{E}\Big(\sup_{0\leq t\leq T}|\bar{e}(t)|^{2}\Big)=\mathbb{E}\Big(\sup_{0\leq t\leq T}|\bar{e}(t)|^{2}I_{\{\beta_{\Delta_1,R}>T\}}\Big)+\mathbb{E}\Big(\sup_{0\leq t\leq T}|\bar{e}(t)|^{2}I_{\{\beta_{\Delta_1,R}\leq T\}}\Big)\notag
			\\ \leq&\mathbb{E}\Big(\sup_{0\leq t\leq T}|\bar{e}(t)|^{2}I_{\{\beta_{\Delta_1,R}>T\}}\Big)+\frac{2\delta}{p}\mathbb{E}\Big(\sup_{0\leq t\leq T}|\bar{e}(t)|^{p}\Big)+\frac{p-2}{p\delta^{\frac{2}{p-2}}}\mathbb{P}(\beta_{\Delta_1,R}\leq T).
		\end{align*}
		Owing to $({\bf S3})$-$({\bf S5})$ and $({\bf F1})$-$({\bf F3})$ with $k> 2(\theta_4\vee2)$, it follows from the results of Lemmas  \ref{la3.8} and \ref{L3.2} that for any $2<p\leq k/(\theta_4\vee2)$
		\begin{align}
			\mathbb{E}\Big(\sup_{0\leq t\leq T}|\bar{e}(t)|^{p}\Big)\leq 2^{p-1}\mathbb{E}\Big(\sup_{0\leq t\leq T}|\bar{Z}(t)|^{p}\Big)+2^{p-1}\mathbb{E}\Big(\sup_{0\leq t\leq T}|\bar{X}(t)|^{p}\Big)\leq C_{x_0,y_0,T,p}.\notag
		\end{align}
		Furthermore,  Remarks \ref{rem5.2} and \ref{rem5.4} imply that
		\begin{align*}
			\mathbb{P}(\beta_{\Delta_1,R}\leq T)\leq \mathbb{P}(\rho_{\Delta_1,R}\leq T)+\mathbb{P}(\bar{\rho}_{\Delta_1,R}\leq T)\leq \frac{C_{x_0,y_0,T,p}}{R^{p}}.
		\end{align*}
		Consequently, we have
		\begin{align*}
			\mathbb{E}\Big(\sup_{0\leq t\leq T}|\bar{e}(t)|^{2}\Big)\leq \mathbb{E}\Big(\sup_{0\leq t\leq T}|\bar{e}(t)|^{2}I_{\{\beta_{\Delta_1,R}>T\}}\Big)+\frac{C_{x_0,y_0,T,p}\delta }{p}+\frac{ C_{x_0,y_0,T,p}(p-2)}{p\delta^{\frac{2}{p-2}}R^{p}}.
		\end{align*}
		Now, for any  $\epsilon>0$, choose $\delta>0$ sufficiently  small such that $ {C_{x_0,y_0,T,p}\delta }/{p}\leq  {\epsilon}/{3}$.  Then for this $\delta$, choose $R>0$ large enough such that
		$
		C_{x_0,y_0,T,p}(p-2)/ (p\delta^{\frac{2}{p-2}}R^{p}) \leq {\epsilon}/{3}.
		$
		Hence, for the desired assertion it is sufficient to prove
		\begin{align}
			\label{f7}
			\mathbb{E}\Big(\sup_{0\leq t\leq T}|e(t)|^{2}I_{\{\beta_{\Delta_1,R}>T\}}\Big)\leq \frac{\epsilon}{3}.
		\end{align}
		From (\ref{a22}) and (\ref{EQN5.9}) one derives that for any $0\leq t\leq T$,
		\begin{align*}
			\bar{e}(t\wedge\beta_{\Delta_1, R})
			=& \int_{0}^{t\wedge\beta_{\Delta_1,R}}\Big(\bar{b}(T_{\Delta_1}(Z(s)))-B_{M}\Big(T_{\Delta_1}(X(s)),Y^{T_{\Delta_1}(X(s)),y_0}\Big)\Big)\mathrm{d}s\notag
			\\&~~~+\int_{0}^{t\wedge\beta_{\Delta_1,R}}\big(\sigma(Z(s))-\sigma(X(s))\big)\mathrm{d}W^{1}(s).\notag
		\end{align*}
		Recalling the definition of the stopping time $\beta_{\Delta_1,R}$, it is straightforward to see that
		\begin{align*}
			T_{\Delta_1}(Z(s))=Z(s), ~~~T_{\Delta_1}(X(s))=X(s),~~~\forall s\in [0, t\wedge\beta_{\Delta_1,R}].
		\end{align*}
		Then we have
		\begin{align*}
			\bar{e}(t\wedge\beta_{\Delta_1,R})&=\int_{0}^{t\wedge\beta_{\Delta_1,R}}\Big(\bar{b}(Z(s))-B_{M}\Big(X(s),Y^{X(s),y_0}\Big)\Big)\mathrm{d}s\notag
			\\&~~~+\int_{0}^{t\wedge\beta_{\Delta_1,R}}\big(\sigma(Z(s))-\sigma(X(s))\big)\mathrm{d}W^{1}(s).
		\end{align*}
		Using the $\mathrm{H\ddot{o}lder}$ inequality, the Burkholder-Davis-Gundy inequality \cite[p.40, Theorem 7.2]{MR2380366} and the triangle inequality, we arrive at
		\begin{align}\label{f29}
			\mathbb{E}\Big(\sup_{0\leq t\leq T}|\bar{e}(t\wedge\beta_{\Delta_1,R})|^{2}\Big)
			\leq& 2T\int_{0}^{T}\mathbb{E}\Big(\Big|\bar{b}(Z(s))-B_{M}\Big(X(s),Y^{X(s),y_0}\Big)\Big|^{2}I_{\{s\leq\beta_{\Delta_1,R}\}}\Big)\mathrm{d}s\notag
			\\&~~~+8\int_{0}^{T}\mathbb{E}\big|\sigma(Z(s\wedge\beta_{\Delta_1,R}))-\sigma(X(s\wedge\beta_{\Delta_1,R}))\big|^{2}\mathrm{d}s\notag
			\\ \leq& 4T\int_{0}^{T}\!\mathbb{E}\Big(\Big| \bar{b}(X(s))-B_{M}\Big(X(s),Y^{X(s),y_0}\Big)\Big|^{2}I_{\{s\leq \beta_{\Delta_1,R}\}}\Big)\mathrm{d}s\notag
			\\&~~~+4T\!\!\int_{0}^{T}\!\mathbb{E}\big|\bar{b}(Z(s\wedge\beta_{\Delta_1,R}))-\bar{b}(X(s\wedge\beta_{\Delta_1,R}))\big|^{2}\mathrm{d}s\notag
			\\&~~~+8\int_{0}^{T}\mathbb{E}\big|\sigma(Z(s\wedge\beta_{\Delta_1,R}))-\sigma(X(s\wedge\beta_{\Delta_1,R}))\big|^{2}
			\mathrm{d}s.
		\end{align}
		For any $0\leq s\leq T$, one provides that for any $\omega\in \{\omega\in \Omega:  s\leq\beta_{\Delta_1,R}\}$, $|X(s)|\leq R$. Using this fact and \eqref{cyp4.7} implies that
		\begin{align*}
			&\mathbb{E}\Big(\Big|\bar{b}(X(s))-B_{M}\Big(X(s),Y^{X(s),y_0}\Big)  \Big|^{2}I_{\{s\leq\beta_{\Delta_1,R}\}}\Big)\notag
			\\ \leq& \mathbb{E}\Big(\Big|\bar{b}(X(s))-B_{M}\Big(X(s),Y^{X(s),y_0}\Big) \Big|^{2}I_{\{|X(s)|\leq R\}}\Big)\notag
			\\ =& \mathbb{E}\Big(\Big|\bar{b}(X_{n_{\Delta_1}(s)})-B_{M}\Big(X_{n_{\Delta_1}(s)},Y^{X_{n_{\Delta_1}(s)},y_0}_{n_{\Delta_1}(s)}\Big) \Big|^{2}I_{\{|X_{n_{\Delta_1}(s)}|\leq R\}}\Big)\notag
			\\ = & \mathbb{E}\Bigg(\mathbb{E}\Big[\Big(\Big| \bar{b}(X_{n_{\Delta_1}(s)})-B_{M}\Big(X_{n_{\Delta_1}(s)},Y^{X_{n_{\Delta_1}(s),y_0}}_{n_{\Delta_1}(s)}\Big)\Big|^{2}I_{\{|X_{n_{\Delta_1}(s)}|\leq R\}}\Big)\Big|X_{n_{\Delta_1}(s)}\Big]\Bigg)\notag
			\\ = & \mathbb{E}\Big(\mathbb{E}\Big| \bar{b}(x)-B_{M}\Big(x,Y^{x,y_0}_{n_{\Delta_1}(s)}\Big)\Big|^{2}_{x=X_{n_{\Delta_1}(s)}}I_{\{|X_{n_{\Delta_1}(s)}|\leq R\}}\Big).\notag
		\end{align*}
		By $({\bf S2})$, $({\bf S4})$ and $({\bf F1})$-$({\bf F3})$ with  $k\geq 2\theta_2\vee 2\theta_4\vee(\theta_2+\theta_4+1)$, it follows from the result of  Lemma \ref{Lb.7} that
		\begin{align}\label{f2+}
			&\mathbb{E}\Big(\Big|\bar{b}(X(s))-B_{M}\Big(X(s),Y^{X(s),y_0}\Big) \Big|^{2}I_{\{s\leq\beta_{\Delta_1,R}\}}\Big)\notag
			\\ =& C\mathbb{E}\Big[\big(1+|X_{n_{\Delta_1}(s)}|^{2\theta_3\vee2\theta_4\vee(\theta_2+\theta_3\vee\theta_3+1)}+|y_0|^{2\theta_4\vee(\theta_2+\theta_3\vee\theta_4+1)}\big)I_{\{|X_{n_{\Delta_1}(s)}|\leq R\}}\Big]\Big(\Delta_{2}+\frac{1}{M\Delta_{2}}\Big)\notag
			\\ \leq &{C_{y_0, R}}\Big(\Delta_{2}+\frac{1}{M\Delta_{2}}\Big).
		\end{align}
		Under $({\bf S1})$, $({\bf S2})$, $({\bf S4})$ and $({\bf F1})$-$({\bf F3})$ with $k\geq 2\vee\theta_1\vee2\theta_2\vee\theta_4$, applying Lemma \ref{L3.3} yields  that
		\begin{align}\label{f12}
			&\mathbb{E}\big|\bar{b}(Z(s\wedge\beta_{\Delta_1,}))-\bar{b}(X(s\wedge\beta_{\Delta_1,R}))\big|^{2}\vee\mathbb{E}\big| \sigma(Z(s\wedge\beta_{\Delta_1,R}))-\sigma(X(s\wedge\beta_{\Delta_1,R}))\big|^{2}\notag
			\\ \leq &(\bar {L}_R^2\vee L^2_{R})\mathbb{E}|\bar{e}(s\wedge\beta_{\Delta_1,R})|^{2}
			\leq(\bar {L}_R^2\vee L^2_{R})\mathbb{E}\Big(\sup_{0\leq r\leq s}|\bar{e}(s\wedge\beta_{\Delta_1,R})|^{2}\Big) .
		\end{align}
		Inserting \eqref{f2+} and \eqref{f12} into \eqref{f29} one  derives that
		\begin{align*}
			\mathbb{E}\Big(\sup_{0\leq t\leq T}|\bar{e}(t\wedge\beta_{\Delta_1,R})|^{2}\Big)
			&\leq 4T^2 C_{y_0, R}\Big(\Delta_{2}+\frac{1}{M\Delta_{2}}\Big)\notag
			\\&~~~+(8+4T) (\bar {L}_R^2\vee L^2_{R})
			\int_{0}^{T}\mathbb{E}\Big(\sup_{0\leq r\leq s}|\bar{e}(s\wedge\beta_{\Delta_1,R})|^{2}\Big)\mathrm{d}s.
		\end{align*}
		An application of the $\mathrm{Gronwall}$ inequality implies  that
		\begin{equation}
			\begin{aligned}
				\mathbb{E}\Big(\sup_{0\leq t\leq T}|\bar{e}(t\wedge\beta_{\Delta_1,R})|^{2}\Big)
				\leq  C_{y_0,T,R}\Big(\Delta_{2}+\frac{1}{M\Delta_{2}}\Big).\notag
			\end{aligned}
		\end{equation}
For the  given  $R$, choose  $\hat{\Delta}_{2}$  sufficiently small such that $C_{y_0,T,R}\hat{\Delta}_{2}\leq\epsilon/6$. For any fixed $\Delta_{2}\in (0,\hat{\Delta}_2]$, choose a constant $ M^*_{\Delta_2}\geq 1/\Delta^2_2$. Then for any $M\geq M^*_{\Delta_2}$ one has $1/(M\Delta_2)\leq \Delta_2$. Thus,
$ C_{y_0,T,R}(\Delta_2+1/(M\Delta_2))\leq 2C_{y_0,T,R} \Delta_2 \leq 2C_{y_0,T,R} \hat{\Delta}_2\leq \epsilon/3.$  Therefore, it follows  that
		\begin{align*}
			\mathbb{E}\Big(\sup_{0\leq t\leq T}|\bar{e}(t)|^{2}I_{\{\beta_{\Delta_1,R}>T\}}\Big)&\leq
			\mathbb{E}\Big(\sup_{0\leq t\leq T}|\bar{e}(t\wedge\beta_{\Delta_1,R})|^{2}\Big) \leq \frac{\epsilon}{3},
		\end{align*}
 which implies the desired result.
	\end{proof}
	
	Obviously, combining  the second result of Lemma \ref{la3.8}, Lemmas \ref{3.8} and \ref{Lb.13+} provides the strong convergence between $ \bar{x}(t)$ and $X(t)$.
\begin{thm}\label{Lb.14} 
{\rm If $({\bf S1})$-$({\bf S5})$ and $({\bf F1})$-$({\bf F3})$ hold with $k\geq \theta_1\vee2(\theta_4\vee2)\vee2\theta_2\vee(\theta_2+\theta_4+1)$, then  for any $x_0\in \RR^{n_1}$, $y_0\in \RR^{n_2}$, $0<p<k/(\theta_4\vee2)$ and $T>0$
					\begin{align}\label{f3}
							\lim_{(\Delta_1, \Delta_2)\rightarrow(0,0)}\lim_{M\rightarrow\infty}\mathbb{E}\big|\bar{x}(t)-{X}(t)\big|^{p}=0.
				\end{align}}
			\end{thm}
			\begin{proof}\textbf {Proof.}
				For any $T>0$,
				combining  Lemmas \ref{la3.8}, \ref{3.8}  and  \ref{Lb.13+} implies that the desired assertion holds for $p=2$. Obviously, \eqref{f3} holds for $0<p<2$ due to the $\mathrm{H\ddot{o}lder}$ inequality.
				Next,
				we consider the case $2<p<k/(\theta_4\vee2)$. Choose a constant $\bar{q}$ such that  $ p<\bar{q}<k/(\theta_4\vee2)$.
				Utilizing the $\mathrm{H\ddot{o}lder}$ inequality,  $\mathrm{Lemmas}$ \ref{L2} and \ref{la3.8} we derive that
				\begin{align*}
					\mathbb{E}|\bar{x}(t)&- {X}(t)|^{p}=\mathbb{E}\Big(\big|\bar{x}(t)- {X}(t)\big|^{\frac{2(\bar{q}-p)}{\bar{q}-2}}\big|\bar{x}(t)- {X}(t)\big|^{p-\frac{2(\bar{q}-p)}{\bar{q}-2}}\Big)\notag
					\\&\leq \Big(\mathbb{E}\big|\bar{x}(t)- {X}(t)\big|^{2}\Big)^{\frac{\bar{q}-p}{\bar{q}-2}}\Big(\mathbb{E}\big|\bar{x}(t)- {X}(t)\big|^{\bar{q}}\Big)^{\frac{p-2}{\bar{q}-2}}\notag
					\\&\leq C_{x_0,y_0,T,p}\Big(\mathbb{E}\big|\bar{x}(t)- {X}(t)\big|^{2}\Big)^{\frac{\bar{q}-p}{\bar{q}-2}}.
				\end{align*}
This, together with the case of $p=2$, implies the required assertion. 
			\end{proof}

Incorporating the  results of  Lemmas \ref{L1} and Theorem \ref{Lb.14}, using the triangle inequality, we  obtain the strong convergence  between the slow component $x^{\varepsilon}(t)$ and the numerical solution $X(t)$ generated by MTEM scheme.
			\begin{thm}\label{th2}
				{\rm If $({\bf S1})$-$({\bf S5})$ and $({\bf F1})$-$({\bf F3})$  hold with  $k> 4\theta_1\vee2(\theta_2+1)\vee2\theta_3\vee {\color{red}2(\theta_4\vee2)}$, then for any $x_0\in \RR^{n_1}$, $y_0\in \RR^{n_2}$, $0<p<k/(\theta_4\vee2)$ and $T>0$,
					\begin{align}
							\lim_{(\varepsilon, \Delta_{1}, \Delta_2)\rightarrow(0,0,0)}\lim_{M\rightarrow\infty}\mathbb{E}\Big(\sup_{0\leq t\leq T}|x^{\varepsilon}(t)-X(t)|^{p}\Big)=0.\notag
				\end{align}}
			\end{thm}
			\begin{proof}\textbf {Proof.}
				For any $0< p<k/(\theta_4\vee2)$, using the triangle inequality, by virtue of Lemmas \ref{L1} and Theorem \ref{Lb.14}, yields that
				 \begin{align}					&\lim_{(\varepsilon,\Delta_{1},\Delta_2)\rightarrow(0,0,0)}\lim_{M\rightarrow\infty}\mathbb{E}\Big(\sup_{0\leq t\leq T}|x^{\varepsilon}(t)-X(t)|^{p}\Big)\notag
						\\ \leq& 2^p\lim_{\varepsilon\rightarrow0}\mathbb{E}|x^{\varepsilon}(t)-\bar{x}(t)|^p
						+2^p \lim_{(\Delta_1,\Delta_2)\rightarrow(0, 0)}\lim_{M\rightarrow\infty} \mathbb{E}\Big(\sup_{0\leq t\leq T}|\bar{x}(t)- {X}(t)|^{p}\Big)=0.\notag
				\end{align}
				The proof is complete.
			\end{proof}
			\section{Strong error bounds}\label{s-c6}
			This section  focuses on  the strong error estimate of the MTEM scheme. To obtain the rates of
			strong convergence we need somewhat stronger conditions compared with the strong convergence alone, which are
			stated as follows.
			\begin{itemize}
				\item[{(\bf S1'})] For any $x_{1}, x_2\in \mathbb{R}^{n_{1}}$  and  $y\in \mathbb{R}^{n_{2}}$, there exist constants $\theta_1\geq1$ and  $K>0$ such that
				\begin{align}
					|b(x_{1},y)-b(x_{2},y)|+|\sigma(x_{1})-\sigma(x_{2})|\leq K|x_{1}-x_{2}|(1+|x_1|^{\theta_1}+|x_2|^{\theta_1}+|y|^{\theta_{1}}).\notag
				\end{align}
				\item[{(\bf S6)}]For any $x_1,x_2 \in \mathbb{R}^{n_1}$ and $y_1,y_2 \in \mathbb{R}^{n_{2}}$, there is a constant $K_{5}>0$ such that
				\begin{align*}
					2(x_1-x_2)^{T}(b(x_1,y_1)-b(x_2,y_2))&+|\sigma(x_1)-\sigma(x_2)|^2
					\leq K_5\big(|x_1-x_2|^2+|y_1-y_2|^2\big).
				\end{align*}
			\end{itemize}
			\begin{remark}\label{R5.1}
				{\rm It follows from $(\text{\bf S1'})$ and $({\bf S2})$ that for any $(x,y)\in \mathbb{R}^{n_1}\times \mathbb{R}^{n_2}$,
					\begin{align}
						|b(x,y)|&\leq |b(x,y)-b(x,0)|+|b(x,0)-b(0,0)|+|b(0,0)|\notag
						\\&\leq  K_{1} |y|(1+|x|^{\theta_2}+|y|^{\theta_{2}})+K|x|\big(1
						+|x|^{\theta_1}\big)+|b(0,0)|\notag
						\\&\leq  C(1+|x|^{(\theta_1\vee\theta_2)+1}+|y|^{\theta_2+1}),\notag
					\end{align}
					namely, combining $(\text{\bf S1'})$ and $({\bf S2})$  leads to $({\bf S4})$ with $\theta_3=\theta_1\vee\theta_2+1$  and $\theta_4=\theta_2+1$.}
			\end{remark}
			\begin{remark}\label{rem4.2}
				{\rm According to  Remark \ref{R5.1},   choose $\varphi(u)=1+u^{\theta_1\vee\theta_2}$ and then $\varphi^{-1}(u)=(u-1)^{\frac{1}{\theta_1\vee\theta_2}},~\forall u\geq 1$.  Then for any $u\geq1$ and $|x|\leq u$,
					\begin{align*}
						|b(x,y)|
						\leq C\varphi(u)(1+|x|)+|y|^{\theta_2+1},~~~\forall y\in \RR^{n_2}.
				\end{align*}}
			\end{remark}
			
			Using the similar techniques to that of Lemma \ref{L3.3}, we derive that the averaged coefficient $\bar{b}$ keeps the property of polynomial growth. To avoid duplication we omit the  proof.
			\begin{lem}\label{Lcyp2.2}
				{\rm If $(\text{\bf S1'})$, $({\bf S2})$ and $({\bf F1})$-$({\bf F3})$  hold with $k\geq {2\vee\theta_1\vee2\theta_2}$, then for any $x_1,x_2\in \mathbb{R}^{n_1}$,  there is a constant $\bar{L}>0$ such that
					\begin{align*}
						\big|\bar{b}(x_1)-\bar{b}(x_2)\big|
						\leq \bar{L}|x_1-x_2|(1+|x_1|^{\theta_1\vee\theta_2}+|x_2|^{\theta_1\vee\theta_2}).
				\end{align*}}
			\end{lem}
			
			\begin{lem}\label{Lcyp2.3}
				{\rm If $(\text{\bf S1'})$, $({\bf S2})$, $(\text{\bf S6})$ and $({\bf F1})$-$(\bf{F3})$ hold with  ${k\geq2\vee(\theta_2+1)}$, then for any $x_1, x_2\in \mathbb{R}^{n_1}$,
					\begin{align*}
						2(x_1-x_2)^{T}\big(\bar{b}(x_1)-\bar{b}(x_2)\big)+|\sigma(x_1)-\sigma(x_2)|^{2}\leq C|x_1-x_2|^{2}.
				\end{align*}}
			\end{lem}
			\begin{proof}\textbf {Proof.}
				Due to $(\text{\bf S1'})$, $({\bf S2})$ and  $({\bf F1})$-$({\bf F3})$ with $k\geq\theta_2+1$, it follows from the definition of $\bar{b}(x)$ and $(\text{\bf S4'})$ that
				\begin{align*}
					&2(x_1-x_2)^{T}(\bar{b}(x_1)-\bar{b}(x_2))+|\sigma(x_1)-\sigma(x_2)|^{2}\notag
					\\ =&\int_{\mathbb{R}^{n_2}\times \mathbb{R}^{n_2}}\Big[2(x_1-x_2)^{T}\big(b(x_1,y_1)-b(x_2,y_2)\big)+|\sigma(x_1)-\sigma(x_2)|^2\Big]\pi(\mathrm{d}y_1\times\mathrm{d}y_2)\notag
					\\ \leq& K_5|x_1-x_2|^{2}+K_5\int_{\mathbb{R}^{n_1}\times \mathbb{R}^{n_2}}|y_1-y_2|^2\pi(\mathrm{d}y_1,\mathrm{d}y_2),
				\end{align*}
				where $\pi\in \mathcal{C}(\mu^{x_1},\mu^{x_2})$ is arbitrary. Then owing to the  arbitrariness of $\Pi\in \mathcal{C}(\mu^{x_1},\mu^{x_2})$,
				\begin{align*}
					&2(x_1-x_2)^{T}(\bar{b}(x_1)-\bar{b}(x_2))+|\sigma(x_1)-\sigma(x_2)|^{2}\leq K_5|x_1-x_2|^2+K_5 \mathbb{W}^2_{2}(\mu^{x_1},\mu^{x_2}).
				\end{align*}
				Under $({\bf F1})$-$({\bf F3})$, we deduce from \eqref{cyp2.5} that
				\begin{align*}
					2(x_1-x_2)^{T}(\bar{b}(x_1)-\bar{b}(x_2))+|\sigma(x_1)-\sigma(x_2)|^{2}\leq C|x_1-x_2|^2.
				\end{align*}
				The proof is complete.
			\end{proof}
			\par According to Remark \ref{R5.1} and Lemma \ref{Lb.7}, we give the bound of  $\mathbb{E}\big|\bar{b}(x)-B_{M}(x,Y^{x,y_0}_{n})\big|^{2}$.
			\begin{lem}\label{L3.20}
				{\rm If $(\text{\bf S1'})$, $({\bf S2})$ and $({\bf F1})$-$({\bf F3})$ with $k\geq2(\theta_2+1)$ hold, then for any $x\in \mathbb{R}^{n_1}$, $y_0\in \RR^{n_2}$, $\Delta_2\in (0,\bar\Delta_2]$, and integers $n\geq0$, $M\geq1$,
					\begin{align*}
						\mathbb{E}\big|\bar{b}(x)-B_{M}(x,Y^{x,y_0}_{n})\big|^{2}\leq  C\big(1+|x|^{2(\theta_1\vee\theta_2+1)} {+|y_0|^{(\theta_2+\theta_1\vee\theta_2+2)}}\big)\Big( \Delta_{2}+\frac{1}{M\Delta_{2}}\Big).
				\end{align*}}
			\end{lem}
			\par    By the similar arguments as  the strong convergence   in Section \ref{s-c5}, we  give the error estimates of  $\mathbb{E}|\bar{x}(T)-\bar{Z}(T)|^{2}$ and $\mathbb{E}| \bar{Z}(T)-X(T)|^{2}$, respectively.
			\begin{lem}\label{L5.4}
				{\rm If $(\text{\bf S1'})$, $({\bf S2})$, $({\bf S3})$, $(\text{\bf S4'})$, $({\bf S5})$ and $({\bf F1})$-$({\bf F3})$  hold with $k\geq{2\vee\theta_1\vee 2\theta_2\vee(\theta_2+1)}$, then for any $x_0\in \RR^{n_1}$, $T>0$ and $\Delta_1\in (0,1]$,
					\begin{align*}
						\mathbb{E}|\bar{x}(T)-\bar{Z}(T)|^{2}\leq C_{T,x_0}\Delta_1.
				\end{align*}}
			\end{lem}
			\begin{proof}\textbf {Proof.}
				Let $e(t)=\bar{x}(t)-\bar{Z}(t)$ for any $t\geq 0$.
				Define the stopping time
				$$\theta_{\Delta_1}=\tau_{\varphi^{-1}(K\Delta_1^{-1/2})}\wedge\rho_{\Delta_1,\varphi^{-1}(K\Delta_1^{-1/2})}.$$
				Using the Young inequality for any $p>2$, we derive that for any $T>0$,
				\begin{align}\label{cyp5.6}
					\mathbb{E}|e(T)|^{2}&=\mathbb{E}\big(|e(T)|^{2}I_{\{\theta_{\Delta_1}>T\}}\big)+\mathbb{E}\big(|e(T)|^2I_{\{\theta_{\Delta_1}\leq T\}}\big)\notag
					\\&\leq \mathbb{E}\big(|e(T)|^2I_{\{\theta_{\Delta_1}> T\}}\big)+\frac{2\Delta_1\mathbb{E}|e(T)|^{p}}{p}+\frac{(p-2)\mathbb{P}(\theta_{\Delta_1}\leq T)}{p\Delta_1^{\frac{2}{p-2}}}.
				\end{align}
				Under $(\text{\bf S1'})$, $({\bf S2})$, $({\bf S3})$, $({\bf S5})$ and $({\bf F1})$-$({\bf F3})$ with $k\geq2\vee\theta_1\vee2\theta_2\vee(\theta_2+1)$, it follows from the results of Lemmas \ref{L2} and \ref{L3.2} that
				\begin{align}\label{eqN6.2}
					\mathbb{E}|e(T)|^{p}\leq C\big(\mathbb{E}|\bar{x}(T)|^{p}+\mathbb{E}|\bar{Z}(T)|^{p}\big)\leq C_{x_0,T,p}.
				\end{align}
				Furthermore, by Remarks \ref{rem5.1} and \ref{rem5.2} we deduce that
				\begin{align*}
					\mathbb{P}(\theta_{\Delta_1}\leq T)&\leq \mathbb{P}\big(\tau_{\varphi^{-1}(K\Delta_1^{-1/2})}\leq T\big)+\mathbb{P}\big(\rho_{\Delta_1,\varphi^{-1}(K\Delta_1^{-1/2})}\leq T\big)\notag
					\\&\leq \frac{C_{x_0,T,p}}{\big(\varphi^{-1}(K\Delta_1^{-1/2})\big)^{p}}.
				\end{align*}
 Letting $p\geq 2(\theta_1\vee\theta_2+1)$ and  using the explicit form of $\varphi^{-1} $ given in Remark \ref{rem4.2} yield that
 \begin{align}\label{eqN6.3}
 \frac{(p-2)\mathbb{P}(\theta_{\Delta_1}\leq T)}{p\Delta_1^{\frac{2}{p-2}}}&\leq \frac{C_{x_0,T,p}}{\Delta_1^{\frac{2}{p-2}}\big(\varphi^{-1}(K\Delta_1^{-1/2})\big)^{p}}\leq \frac{C_{x_0,T,p}}{\Delta_1^{\frac{2}{p-2}}(K\Delta_1^{-\frac{1}{2}}-1)^{\frac{p}{\theta_1\vee\theta_2}}}\nn\
 \\&\leq C_{x_0,T,p}\Delta_1.
 \end{align}
Then inserting \eqref{eqN6.2} and \eqref{eqN6.3} into \eqref{cyp5.6} implies that 
				\begin{align*}
					\mathbb{E}|e(T)|^{2}\leq C_{x_0,T,p}\Delta_1+\mathbb{E}|e(T\wedge\theta_{\Delta_1})|^{2}.
				\end{align*}
				Thus for the desired result it is sufficient to prove
				\begin{align*}
					\mathbb{E}|e(T\wedge\theta_{\Delta_1})|^{2}\leq C_{x_0,T,p}\Delta_1.
				\end{align*}
				Recalling
				the definition of the stopping time $\theta_{\Delta_1}$, one observes that  $Z^{*}(t)=Z(t), ~0\leq t\leq T\wedge \theta_{\Delta_1}$. Thus,  using the $\mathrm{It\hat{o}}$ formula for (1.2) and (5.13) and then applying the inequality $(a+b)^2\leq |a|^2+|b|^2+2|a||b|$ yields that
\begin{align*}
					\mathbb{E}|e(T\wedge\theta_{\Delta_1})|^{2}
					=& \mathbb{E}\int_{0}^{T\wedge\theta_{\Delta_1}}\Big[2e^{T}(t)
					\big( \bar{b}(\bar{x}(t))-\bar{b}(Z(t))\big)
					+|\sigma(\bar{x}(t))-\sigma(Z(t))|^{2}\Big]\mathrm{d}t\notag
					\\ \leq& \mathbb{E}\int_{0}^{T\wedge\theta_{\Delta_1}}\Big[2e^{T}(t)\big(\bar{b}(\bar{x}(t))-\bar{b}(\bar{Z}(t))\big)+2e^{T}(t)\big(\bar{b}(\bar{Z}(t))-\bar{b}(Z(t))\big)\Big]\mathrm{d}t\notag
\\&~~~+\mathbb{E}\int_{0}^{T\wedge\theta_{\Delta_1}}|\sigma(\bar{x}(t))-\sigma(\bar{Z}(t))|^2
+|\sigma(\bar{Z}(t))-\sigma(Z(t))|^{2}\nn\
\\&~~~+2|\sigma(\bar{x}(t))-\sigma(\bar{Z}(t))||\sigma(\bar{Z}(t))-\sigma(Z(t))|\mathrm{d}t.\nn\
				\end{align*}
After adjusting the order, one  obtains that
				\begin{align*}
					\mathbb{E}|e(T\wedge\theta_{\Delta_1})|^{2}
					\leq& \mathbb{E}\int_{0}^{T\wedge\theta_{\Delta_1}}\Big[2e^{T}(t)
					\big( \bar{b}(\bar{x}(t))-\bar{b}(\bar{Z}(t))\big)
		+|\sigma(\bar{x}(t))-\sigma(\bar{Z}(t))|^{2}\Big]\mathrm{d}t\notag
					\\&~~~+2\mathbb{E}\int_{0}^{T\wedge\theta_{\Delta_1}}e^{T}(t)
					\big( \bar{b}(\bar{Z}(t))-\bar{b}(Z(t))\big)\mathrm{d}t\notag
\\&~~~+\mathbb{E}\int_{0}^{T\wedge\theta_{\Delta_1}}|\sigma(\bar{Z}(t))-\sigma(Z(t))|^{2}\mathrm{d}t\notag
					\\&~~~+ 2\mathbb{E}\int_{0}^{T\wedge\theta_{\Delta_1}}|\sigma(\bar{x}(t))-\sigma(\bar{Z}(t))||\sigma(\bar{Z}(t))-\sigma(Z(t))|\mathrm{d}t.
				\end{align*}
				Under $({\bf S6})$ and  $({\bf F1})$-$({\bf F3})$ with $k\geq2\vee(\theta_2+1)$, utilizing the Lemma \ref{Lcyp2.3} and the Young inequality we derive that
				\begin{align}\label{cyp5.11}
					\mathbb{E}|e(T\wedge\theta_{\Delta_1})|^2
					\leq C\mathbb{E}\int_{0}^{T\wedge\theta_{\Delta_1}}|e(t)|^{2}\mathrm{d}t+J_1+J_2,
				\end{align}
				where
				\begin{align*}
					&J_1=C\mathbb{E}\int_{0}^{T\wedge\theta_{\Delta_1}}\Big(|\bar{b}(\bar{Z}(t))-\bar{b}(Z(t))|^{2}+|\sigma(\bar{Z}(t))-\sigma(Z(t))|^{2}\Big)\mathrm{d}t,\notag
					\\&J_2=C\mathbb{E}\int_{0}^{T\wedge\theta_{\Delta_1}}|\sigma(\bar{x}(t))-\sigma(\bar{Z}(t))||\sigma(\bar{Z}(t))-\sigma(Z(t))|\mathrm{d}t.
				\end{align*}
				Due to $({\bf S1'})$, $({\bf S2})$, $({\bf S3})$, $({\bf S5})$ and $({\bf  F1})$-$({\bf F3})$ with $k\geq 2\vee\theta_1\vee2\theta_2\vee(\theta_2+1)$, it follows from  the results of Lemmas  \ref{L3.2} and \ref{Lcyp2.2}  that
				\begin{align}\label{cyp5.12}
					J_1&\leq C\int_{0}^{T}\mathbb{E}\Big[|\bar{Z}(t)-Z(t)|^2\big(1+|\bar{Z}(t)|^{2(\theta_1\vee\theta_2)}+|Z(t)|^{2(\theta_1\vee\theta_2)}\big)\Big ]\mathrm{d}t\notag
					\\&\leq C\int_{0}^{T}\Big(\mathbb{E}|\bar{Z}(t)-Z(t)|^4\Big)^{\frac{1}{2}}\Big[\mathbb{E}\big(1+|\bar{Z}(t)|^{4(\theta_1\vee\theta_2)}+|Z(t)|^{4(\theta_1\vee\theta_2)}\big)\Big]^{\frac{1}{2}}\mathrm{d}t\notag
\\&\leq C_{x_0,T,p}\Delta_1.
				\end{align}
				In addition, using the Young inequality and the H\"older inequality yields that
				\begin{align*}
					J_2&\leq C\mathbb{E}\int_{0}^{T\wedge\theta_{\Delta_1}}|e(t)||\bar{Z}(t)-Z(t)|(1+|\bar{x}(t)|^{2\theta_1}+|\bar{Z}(t)|^{2\theta_1}+|Z(t)|^{2\theta_1})\mathrm{d}t\notag
					\\&\leq C\int_{0}^{T}\big(\mathbb{E}|\bar{Z}(t)-Z(t)|^{4}\big)^{\frac{1}{2}}\Big[\mathbb{E}\big(1+|\bar{x}(t)|^{8\theta_1}+|\bar{Z}(t)|^{8\theta_1}+|Z(t)|^{8\theta_1}\big)\Big]^{\frac{1}{2}}\mathrm{d}t\notag
					\\&~~~+C\int_{0}^{T}\mathbb{E}|e(t\wedge\theta_{\Delta_1})|^2\mathrm{d}t.
				\end{align*}
				Similarly to \eqref{cyp5.12}, applying  Lemmas \ref{L2} and \ref{L3.2} we  show that
				\begin{align}\label{cyp5.13}
					J_2&\leq C_{x_0,T,p}\Delta_1+ C\int_{0}^{T}\mathbb{E}|e(t\wedge\theta_{\Delta_1})|^2\mathrm{d}t.
				\end{align}
				Inserting \eqref{cyp5.12} and \eqref{cyp5.13} into \eqref{cyp5.11}  and then using  Gronwall's inequality derive that
				\begin{align}
					\mathbb{E}|e(T\wedge\theta_{\Delta_1})|^{2}\leq C_{x_0,T,p}\Delta_1,\notag
				\end{align}
				which implies the desired result. The proof is complete.
			\end{proof}
			\begin{lem}\label{L5.7}
				If $(\text{\bf S1'})$, $({\bf S2})$,  $({\bf S3})$,  $({\bf S5})$, $(\text{\bf S6})$ and  $({\bf F1})$-$({\bf F3})$ with $k\geq[2(2\theta_1+1)\vee2(\theta_1\vee\theta_2+1)]((\theta_2+1)\vee2)$ hold, then for any $x_0\in \RR^{n_1}$, $y_0\in \RR^{n_2}$, $T>0$, $\Delta_1\in(0,1]$, $\Delta_2\in(0,\bar\Delta_2]$ and $M\geq1$,
				\begin{align*}
					\mathbb{E}|\bar{Z}(T)-\bar{X}(T)|^{2}\leq C_{x_0,y_0,T}\Big( \Delta_1+\Delta_{2}+\frac{1}{M\Delta_{2}}\Big).
				\end{align*}
			\end{lem}
			\begin{proof}\textbf {Proof.}
				Define  the stopping time
				$$\bar{\theta}_{\Delta_1}= \rho_{\Delta_1, \varphi^{-1}(K\Delta_1^{-1/2})}\wedge\bar{\rho}_{\Delta_1,\varphi^{-1}(K\Delta_1^{-1/2})},$$
				where $\rho_{\Delta_1, \varphi^{-1}(K\Delta_1^{-1/2})}$ and $\bar{\rho}_{\Delta_1,\varphi^{-1}(K\Delta_1^{-1/2})}$ are given by \eqref{3.22} and \eqref{3.47}.
 By $({\bf S1'})$, $({\bf S2})$, $({\bf S3})$, $({\bf S5})$ and $({\bf F1})$-$({\bf F3})$ with $k> 2((\theta_2+1)\vee 2)$, using Lemmas  \ref{la3.8} and \ref{L3.2} as well as the H\"older inequality yields that for any $2< p\leq k/((\theta_2+1)\vee2)$,
				\begin{align}\label{cyp3.69}
					\sup_{\Delta_1\in(0,1]}\mathbb{E}\Big(\sup_{0\leq t\leq T}|\bar{Z}(t)|^{p}\Big)\vee\sup_{\Delta_{1}\in(0,1], \Delta_2\in (0,\hat\Delta_2]}\sup_{t\in[0,T]}\mathbb{E}|\bar{X}(t)|^{p}\leq C_{x_0,y_0, T, p}.
				\end{align}
				Then applying the $\mathrm{Young}$ inequality, for any $\Delta_1\in (0,1]$ and  $2< p\leq k/((\theta_2+1)\vee2)$ one obtains that
				\begin{align}\label{cyp5.16}
					&\mathbb{E}|\bar{e}(T)|^{2}=\mathbb{E}\big(|\bar{e}(T)|^{2}I_{\{\bar{\theta}_{\Delta_1}>T\}}\big)+\mathbb{E}\big(|\bar{e}(T)|^{2}I_{\{\bar{\theta}_{\Delta_1}\leq T\}}\big)\notag
					\\ \leq&\mathbb{E}\big(|\bar{e}(T)|^{2}I_{\{\bar{\theta}_{\Delta_1}>T\}}\big)+\frac{2\Delta_1}{p}\mathbb{E}|\bar{e}(T)|^{p}+\frac{p-2}{p\Delta_1^{\frac{2}{p-2}}}\mathbb{P}(\bar{\theta}_{\Delta_1}\leq T),
				\end{align}
where $\bar{e}(T)$ is defined by \eqref{NN5.28}.
				It follows from \eqref{cyp3.69} that
				\begin{align}\label{NNN6.9}
					\mathbb{E}|\bar{e}(T)|^{p}\leq 2^{p-1}\mathbb{E}|\bar{Z}(T)|^{p}+2^{p-1}\mathbb{E}|\bar{X}(t)|^{p}\leq C_{x_0,y_0,T,p}.
				\end{align}
Then applying Remarks \ref{rem5.2} and \ref{rem5.4}  gives that
				\begin{align*}
					\mathbb{P}(\bar{\theta}_{\Delta_1}\leq T)\leq \mathbb{P}\big(\rho_{\Delta_1,\varphi^{-1}(K\Delta_1^{-1/2}) }\leq T\big)+\mathbb{P}\big(\bar{\rho}_{\Delta_1,\varphi^{-1}(K\Delta_1^{-1/2})}\leq T\big)\leq \frac{C_{x_0,y_0,T,p}}{(\varphi^{-1}(K\Delta_1^{-\frac{1}{2}}))^{p}}.
				\end{align*}
				Due to $k\geq 2((\theta_2+1)\vee2)(\theta_1\vee \theta_2+1)$, one further let $2(\theta_1\vee \theta_2+1)\leq p\leq k/((\theta_2+1)\vee2)$. Then using the explicit form of $\varphi^{-1} $ given in Remark \ref{rem4.2} yields that
 \begin{align*}
 \frac{(p-2)\mathbb{P}(\bar{\theta}_{\Delta_1}\leq T)}{p\Delta_1^{\frac{2}{p-2}}}&\leq \frac{C_{x_0,y_0,T,p}}{\Delta_1^{\frac{2}{p-2}}\big(\varphi^{-1}(K\Delta_1^{-1/2})\big)^{p}}\leq \frac{C_{x_0,y_0,T,p}}{\Delta_1^{\frac{2}{p-2}}(K\Delta_1^{-\frac{1}{2}}-1)^{\frac{p}{\theta_1\vee\theta_2}}}\nn\
 \\&\leq C_{x_0,y_0,T,p}\Delta_1.
 \end{align*}
Then inserting the above inequality into \eqref{cyp5.16} and then applying \eqref{NNN6.9} shows that
				\begin{align*}
					\mathbb{E}|\bar{e}(T)|^{2}&\leq \mathbb{E}\big(|\bar{e}(T)|^{2}I_{\{\bar{\theta}_{\Delta_1}>T\}}\big)+\frac{C_{x_0,y_0,T,p} \Delta_1 }{p}+\frac{ C_{x_0,y_0,T,p}}{p\Delta^{\frac{2}{p-2}}(\varphi^{-1}(K\Delta_1^{-\frac{1}{2}}))^{p}}\notag
					\\&\leq \mathbb{E}\big(|\bar{e}(T)|^{2}I_{\{\bar{\theta}_{\Delta_1}>T\}}\big)+C_{x_0,y_0,T,p}\Delta_1.
				\end{align*}
				Hence, for the desired result it remains to prove that
				\begin{align*}
					\mathbb{E}\big(|\bar{e}(T)|^{2}I_{\{\bar{\theta}_{\Delta_1}>T\}}\big)\leq C_{x_0,y_0,T,p}\Delta_1.
				\end{align*}
				Obviously, $T_{\Delta_1}(X(t))=X(t)$ and $T_{\Delta_1}(Z(t))=Z(t)$ for any $0\leq t\leq T\wedge\bar{\theta}_{\Delta_1}$. Using the $\mathrm{It\hat{o}}$ formula for (\ref{a22}) and (\ref{EQN5.9}), applying the inequality $(a+b+c)^{2}\leq |a|^2+2|b|^2+2|c|^2+2|a||b|+2|a||c|$ gives that
\begin{align*}
&\mathbb{E}|\bar{e}(T\wedge\bar{\theta}_{\Delta_1})|^2
\\= & \mathbb{E}\int_{0}^{T\wedge\bar{\theta}_{\Delta_1}}\Big[2\bar{e}^{T}(t)\Big(\bar{b}(Z(t))-B_{M}\Big(X(t),Y^{X(t),y_0}\Big)\Big)
					+|\sigma(Z(t))-\sigma(X(t))|^{2}\Big]\mathrm{d}t\nn\
\\=&\mathbb{E}\int_{0}^{T\wedge\bar{\theta}_{\Delta_1}}2\bar{e}^{T}(t)\Big(\bar{b}(Z(t))-\bar{b}(\bar{Z}(t))\Big)+2\bar{e}^{T}(t)\Big(\bar{b}(\bar{Z}(t))-\bar{b}(\bar{X}(t))\Big)\nn\
\\&~~~+2\bar{e}^{T}(t)\Big(\bar{b}(\bar{X}(t))-\bar{b}(X(t))\Big)+2\bar{e}^{T}(t)\Big(\bar{b}(X(t))-B_{M}(X(t),Y^{X(t),y_0})\Big)\mathrm{d}t\nn\
\\&~~~+\mathbb{E}\int_{0}^{T\wedge\bar{\theta}_{\Delta_1}}2|\sigma(Z(t))-\sigma(\bar{Z}(t))|^2+ |\sigma(\bar{Z}(t))-\sigma(\bar{X}(t))|^{2}+2|\sigma(\bar{X}(t))-\sigma(X(t))|^2\nn\
\\&~~~+2|\sigma(\bar{Z}(t))-\sigma(\bar{X}(t))|\big(|\sigma(Z(t))-\sigma(\bar{Z}(t))|+|\sigma(\bar{X}(t))-\sigma(X(t))|\big)\mathrm{d}t.
\end{align*}
After adjusting the order, under $({\bf S6})$ and $({\bf F1})$-$({\bf F3})$, by Lemma \ref{Lcyp2.3} and using the Young inequality one derives  that for any $ T > 0$,
				\begin{align}\label{cypp5.17}
					&\mathbb{E}|\bar{e}(T\wedge\bar{\theta}_{\Delta_1})|^2
					\notag
\\=& \mathbb{E}\int_{0}^{T\wedge\bar{\theta}_{\Delta_1}}\Big[2\bar{e}^{T}(t)\Big(\bar{b}(Z(t))-B_{M}\Big(X(t),Y^{X(t),y_0}\Big)\Big)
					+|\sigma(Z(t))-\sigma(X(t))|^{2}\Big]\mathrm{d}t\notag
					\\  \leq & \mathbb{E}\int_{0}^{T\wedge\bar{\theta}_{\Delta_1}}\Big[2\bar{e}^{T}(t)\Big(\bar{b}(\bar{Z}(t))-\bar{b}(\bar{X}(t))\Big)
					+|\sigma(\bar{Z}(t))-\sigma(\bar{X}(t))|^{2}\Big]\mathrm{d}t\notag
					\\&~~~+\int_{0}^{T}C\mathbb{E}|\bar{e}(t\wedge\bar{\theta}_{\Delta_1})|^2\mathrm{d}t+I_1+I_2+I_3+I_4\notag
					\\ \leq & \int_{0}^{T}C\mathbb{E}|\bar{e}(t\wedge\bar{\theta}_{\Delta_1})|^2\mathrm{d}t+I_1+I_2+I_3+I_4,
				\end{align}
			where
				\begin{align*}
					&I_1=\int_{0}^{T}\mathbb{E}\Big|\bar{b}(X(t))-B_{M}\Big(X(t),Y^{X(t),y_0}\Big)\Big|^2\mathrm{d}t,\notag
					\\&I_2=C\int_{0}^{T}\mathbb{E}\big(|\bar{b}(\bar{X}(t))-\bar{b}(X(t))|^2+|\sigma(\bar{X}(t))-\sigma(X(t))|^2\big)\mathrm{d}t,\notag
					\\&I_3=C\int_{0}^{T}\mathbb{E}\big(|\bar{b}(Z(t))-\bar{b}(\bar{Z}(t))|^2+|\sigma(Z(t))-\sigma(\bar{Z}(t))|^2\big)\mathrm{d}t,\notag
					\\&I_4=C\mathbb{E}\int_{0}^{T\wedge\bar{\theta}_{\Delta_1}}|\sigma(\bar{Z}(t))-\sigma(\bar{X}(t))|\big(|\sigma(\bar{X}(t))-\sigma(X(t))|+|\sigma(Z(t))-\sigma(\bar{Z}(t))|\big)\mathrm{d}t.
				\end{align*}
				In addition, owing to $({\bf S1'})$, $({\bf S2})$ and $({\bf F1})$-$({\bf F3})$ with $k\geq 2(\theta_2+1)$, applying \eqref{cyp4.7} and Lemma \ref{L3.20} implies that for any $0\leq t\leq T$,
				\begin{align*}
					&\mathbb{E}\Big|\bar{b}(X(t))-B_{M}\Big(X(t),Y^{X(t),y_0}\Big)\Big|^2\notag
					\\=&\mathbb{E}\Big|\bar{b}(X_{n_{\Delta_1}(t)})-B_{M}\Big(X_{n_{\Delta_1}(t)},Y^{X_{n_{\Delta_1}(t)},y_0}\Big)\Big|^2
					\\ =&\mathbb{E}\Big[\mathbb{E}\Big(\Big|\bar{b}(X_{n_{\Delta_1}(t)})-B_{M}\Big(X_{n_{\Delta_1}(t)},Y^{X_{n_{\Delta_1}(t)},y_0}_{n_{\Delta_1}(t)}\Big)\Big|^2\Big|X_{n_{\Delta_1}(t)}\Big)\Big]\notag
					\\=&\mathbb{E}\Big(\mathbb{E}\Big|\bar{b}(x)-B_{M}\Big(x,Y^{x,y_0}_{n_{\Delta_1}(t)}\Big)\Big|^2\Big|_{x=X_{n_{\Delta_1}(t)}}\Big)\notag
					\\  \leq&  C\Big( \Delta_{2}+\frac{1}{M\Delta_{2}}\Big)\Big(1+\mathbb{E}|X_{n_{\Delta_1}(t)}|^{2(\theta_1\vee\theta_2+1)}+|y_0|^{\theta_2+\theta_1\vee\theta_2+2}\Big).
				\end{align*}
Since $2(\theta_1\vee \theta_2+1)\leq p\leq k/((\theta_2+1)\vee2)$,  utilizing \eqref{cyp3.69}  and the H\"older inequality we deduce that
				\begin{align}\label{cyp5.19}
					I_1&\leq C\Big(\Delta_{2} +\frac{1}{M\Delta_{2}}\Big)\int_{0}^{T}\big(1+|y_0|^{2(\theta_2+1)}+\mathbb{E}|X_{n_{\Delta_1}(t)}|^{2(\theta_2+1)} \big)\mathrm{d}t\notag
					\\&\leq C_{y_0}\Big( \Delta_{2}+\frac{1}{M\Delta_{2}}\Big)\int_{0}^{T}\big(1+\big(\mathbb{E}|X_{n_{\Delta_1}(t)}|^{p}\big)^{\frac{2(\theta_2+1)}{p}} \big)\mathrm{d}t\notag
					\\ &\leq C_{x_0,y_0,T,p}\Big(\Delta_{2}+\frac{1}{M\Delta_{2}}\Big).
				\end{align}
Under $({\bf S1'})$, $({\bf S2})$ and $({\bf F1})$-$({\bf F3})$ with $k\geq 2\vee\theta_1\vee2\theta_2$, by Lemma \ref{Lcyp2.2} and the H\"older inequality  we derive that
\begin{align}
					I_2+I_3 &\leq C\int_{0}^{T}\mathbb{E}\Big(|\bar{X}(t)-X(t)|^{2}\big(1+|X(t)|^{2(\theta_1\vee\theta_2)}+|\bar{X}(t)|^{2(\theta_1\vee\theta_2)}\big)\Big)\mathrm{d}t\notag
					\\&~~~+C\int_{0}^{T}\mathbb{E}\Big(|Z(t)-\bar{Z}(t)|^{2}\big(1+|Z(t)|^{2(\theta_1\vee\theta_2)}+|\bar{Z}(t)|^{2(\theta_1\vee\theta_2)}\big)\Big)\mathrm{d}t\notag
					\\&\leq C\int_{0}^{T}\Big(\mathbb{E}|\bar{X}(t)-X(t)|^{p}\Big)^{\frac{2}{p}}\Big(\mathbb{E}\big(1+|X(t)|^{\frac{2p(\theta_1\vee\theta_2)}{p-2}}+|\bar{X}(t)|^{\frac{2p(\theta_1\vee\theta_2)}{p-2}}\big)\Big)^{\frac{p-2}{p}}\mathrm{d}t\notag
					\\&~~~+C\int_{0}^{T}\Big(\mathbb{E}|Z(t)-\bar{Z}(t)|^{p}\Big)^{\frac{2}{p}}\Big(\mathbb{E}\big(1+|Z(t)|^{\frac{2p(\theta_1\vee\theta_2)}{p-2}}+|\bar{Z}(t)|^{\frac{2p(\theta_1\vee\theta_2)}{p-2}}\big)\Big)^{\frac{p-2}{p}}\mathrm{d}t.\notag
				\end{align}
Thanks to $ 2(\theta_1\vee\theta_2+1) \leq p\leq k/((\theta_2+1)\vee2)$, we have $$2p(\theta_1\vee\theta_2)/(p-2)\leq p\leq k/((\theta_2+1)\vee2).$$
				Then applying Lemmas \ref{la3.8} and \ref{L3.2} and the Young inequality yields that
				\begin{align}\label{cyp5.20}
					I_2+I_3&\leq C\int_{0}^{T}\Big(\mathbb{E}|\bar{X}(t)-X(t)|^{p}\Big)^{\frac{2}{p}}\Big(\mathbb{E}(1+|X(t)|^{p}+|\bar{X}(t)|^{p})\Big)^{\frac{p-2}{p}}\mathrm{d}t\notag
					\\&~~~+C\int_{0}^{T}\Big(\mathbb{E}|Z(t)-\bar{Z}(t)|^{p}\Big)^{\frac{2}{p}}\Big(\mathbb{E}(1+|Z(t)|^{p}+|\bar{Z}(t)|^{p})\Big)^{\frac{p-2}{p}}\mathrm{d}t\notag
					\\&\leq C_{x_0,y_0,T,p}\Delta_1.
				\end{align}
In view of $({\bf S1'})$, together with using the Young inequality and the H\"older inequality,  we also obtain that
\begin{align*}
					I_4&\leq C\mathbb{E}\int_{0}^{T\wedge\bar{\theta}_{\Delta_1}}|\bar{e}(t)|\big(|\bar{X}(t)-X(t)|+|Z(t)-\bar{Z}(t)|\big)\notag
					\\&~~~~~~\times\big(1+|X(t)|^{2\theta_1}+|\bar{X}(t)|^{2\theta_1}+|\bar{Z}(t)|^{2\theta_1}+|Z(t)|^{2\theta_1}\big)\mathrm{d}t
					\\&\leq C\int_{0}^{T}\mathbb{E}|\bar{e}(t\wedge\bar{\theta}_{\Delta_1})|^2\mathrm{d}t+C\int_{0}^{T}\Big[\mathbb{E}\big(|\bar{X}(t)-X(t)|^p+|Z(t)-\bar{Z}(t)|^p\big)\Big]^{\frac{2}{p}}\notag
					\\&~~~\times \Big[\mathbb{E}\big(1+|X(t)|^{\frac{4p\theta_1}{p-2}}+|\bar{X}(t)|^{\frac{4p\theta_1}{p-2}}+|Z(t)|^{\frac{4p\theta_1}{p-2}}+|\bar{Z}(t)|^{\frac{4p\theta_1}{p-2}}\big)\Big]^{\frac{p-2}{p}}\mathrm{d}t.
				\end{align*}
				Similarly, owing to  $k\geq[2(2\theta_1+1)\vee2(\theta_1\vee\theta_2+1)]((\theta_2+1)\vee2)$, one further let $2(2\theta_1+1)\vee(2(\theta_1\vee\theta_2)+1)\leq p\leq k/((\theta_2+1)\vee2)$.  Thus, $4p\theta_1/(p-2)\leq p\leq k/((\theta_2+1)\vee2)$ holds. By means of Lemmas  \ref{la3.8} and \ref{L3.2}  and using the Young inequality one deduces that
				\begin{align}\label{cyp5.21}
					I_4&\leq C\int_{0}^{T}\mathbb{E}|\bar{e}(t\wedge\bar{\theta}_{\Delta_1})|^2\mathrm{d}t+C\int_{0}^{T}\Big[\mathbb{E}\big(|\bar{X}(t)-X(t)|^p+|Z(t)-\bar{Z}(t)|^p\big)\Big]^{\frac{2}{p}}\notag
					\\&~~~\times \Big[\mathbb{E}\big(1+|X(t)|^{p}+|\bar{X}(t)|^{p}+|Z(t)|^{p}+|\bar{Z}(t)|^{p}\big)\Big]^{\frac{p-2}{p}}\mathrm{d}t\notag
					\\&\leq  C\int_{0}^{T}\mathbb{E}|\bar{e}(t\wedge\bar{\theta}_{\Delta_1})|^2\mathrm{d}t+C_{x_0,y_0,T,p}\Delta_1.
				\end{align}
				Then inserting \eqref{cyp5.19}-\eqref{cyp5.21} into \eqref{cypp5.17} implies that
				\begin{align*}
					\mathbb{E}|\bar{e}_{\Delta_1}(T\wedge\beta_{\Delta_1})|^2\leq C\int_{0}^{T}\mathbb{E}|\bar{e}(t\wedge\beta_{\Delta_1})|^2\mathrm{d}t+C_{x_0,y_0,T,p}\Big( \Delta_1+\Delta_{2}+\frac{1}{M\Delta_{2}}\Big).
				\end{align*}
				Using the Gronwall inequality shows that
				\begin{align*}
					\mathbb{E}|\bar{e}_{\Delta_1}(T\wedge\beta_{\Delta_1})|^2\leq C_{x_0,y_0,T,p}\Big( \Delta_1+\Delta_{2}+\frac{1}{M\Delta_{2}}\Big),
				\end{align*}
				which implies the desired result.
			\end{proof}
			\par Combining  Lemmas \ref{la3.8}, \ref{L5.4} and
			\ref{L5.7}, we yield the strong  error estimate  of the  MTEM scheme directly.
			\begin{thm}\label{L5.1}
				{\rm If  $(\text{\bf S1'})$, $({\bf S2})$, $({\bf S3})$, $({\bf S5})$, $(\text{\bf S6})$ and $({\bf F1})$-$({\bf F3})$ hold with $k\geq[2(2\theta_1+1)\vee2(\theta_1\vee\theta_2+1)]((\theta_2+1)\vee2)$, then for any $x_0\in \RR^{n_1}$, $y_0\in \RR^{n_2}$, $T>0$, $\Delta_1\in (0,\bar\Delta_1]$, $\Delta_2\in (0,\bar\Delta_2]$ and $M\geq1$,
					\begin{align*}
						\mathbb{E}|\bar{x}(T)- {X}(T)|^{2}\leq C_{x_0,y_0,T}\Big( \Delta_1+\Delta_{2}+\frac{1}{M\Delta_{2}}\Big).
				\end{align*}}
			\end{thm}
\begin{rem}
The averaging principle offers a crucial simplification of the original SFSDEs, thereby significantly reducing the computational complexity when approximating the slow component of SFSDEs using the MTEM scheme (3.6).
Theorem 6.6 gives the strong error bounds between the exact solution of the averaged equation (1.2) and the numerical solution generated by MTEM  scheme (3.6). These error bounds can be decomposed into two primary components: 
\begin{itemize}
\item[(1)] The first part $\mathcal{O}(\Delta_1)$ accounts for the error caused by the TEM scheme during macro time discretization, assuming that the averaged coefficient $\bar{b}$ is known, as given in Lemma 6.4; 
\item[(2)] The second part $\mathcal{O}(\Delta_2)+\mathcal{O}(1/M\Delta_2)$ encompasses  the error resulting from  using the estimator $B_{M}(x,Y^{x,y_0}_{n})$ to replace $\bar{b}(x)$ in the macro time discretization. This component includes errors caused by the error of EM scheme during the micro time discretization, as well
as the   approximation error of  the ergodic limit about $\mu^{x}$, as present in Lemma 6.3.
\end{itemize}
Furthermore, it is noteworthy that Theorem 6.6 establishes the optimal convergence rates for $\Delta_1$, $\Delta_2$ and $M$.  
\end{rem}

			Based on the result of Theorem 6.6,  the determination of the strong convergence rate of the averaging principle further allows us to ascertain the strong error estimate between the slow component of the original system and the MTEM numerical solution. An important case is presented here to illustrate this.  Let us assume that the slow drift term $b=b_1+b_2$ and satisfies that\\
			$({\bf B1})$
			There exist  constants $C_1>0$, $\alpha>0$ and $\theta\geq 2$ such that for any $x\in \mathbb{R}^{n_1}$,
			\begin{align*}
				x^{T}b_{1}(x)\leq -\alpha|x|^{\theta}+ C_1(1+|x|^{2}).
			\end{align*}
			$({\bf B2})$
			There exists a  constants $L>0$  such that for any $x,x_i\in \RR^{n_1}$ and $y_i\in\mathbb{R}^{n_2}$, $i=1,2$,
			\begin{align*}
				|b_1(x)|\leq L(1+|x|^{\theta-1}),
\end{align*}
and
\begin{align*}
|b_2(x_1,y_1)-b_2(x_2,y_2)|+|\sigma(x_1)-\sigma(x_2)|\leq L(|x_1-x_2|+|y_1-y_2|),
			\end{align*}
where the constant $\theta$ is given in $({\bf B1})$.\\
			$({\bf B3})$
			There exists a constant $K>0$ such that for any $x_1,x_2\in \mathbb{R}^{n_1}$,
			\begin{align*}
				(x_1-x_2)^{T}(b_{1}(x_1)-b_1(x_2))\leq K|x_1-x_2|^2.
			\end{align*}		
			Meanwhile, Assumptions $({\bf F1})$-$({\bf F3})$ are preserved without modification.  Subsequently, the subsequent strong averaging principle can be inferred from  \cite[Theorem 2.2]{MR4374850}.
			\begin{lem}[{{\cite[Theorem 2.2]{MR4374850}}}]
				{\rm Suppose that  $({\bf B1})$-$({\bf B3})$ and $({\bf F1})$-$({\bf F3})$ hold. Then for any $(x_0,y_0)\in \RR^{n_1}\times\RR^{n_2}$ and $T>0$,
					\begin{align*}
						\mathbb{E}\Big(\sup_{t\in[0,T]}|x^{\varepsilon}(t)-\bar{x}(t)|^{2}\Big)\leq C\varepsilon^{\frac{1}{3}}.
				\end{align*}}
			\end{lem}
			\begin{thm}\label{L6.9}
				{\rm Suppose that  $({\bf B1})$-$({\bf B3})$ and $({\bf F1})$-$({\bf F3})$ hold with $k\geq 4(2\theta-1)$. Then for any $T>0$, $\Delta\in (0,1]$, $\Delta_2\in (0,\bar{\Delta}_2]$ and $M>1$,
					$$
					\mathbb{E}|x^{\varepsilon}(T)-X(T)|^{2}\leq C_{T}\left(\varepsilon^{\frac{1}{3}}+\Delta_{1}+\Delta_{2}+\frac{1}{M\Delta_{2}}\right).
					$$}
			\end{thm}

			\section{Numerical examples}\label{s-c7}
			This section gives two examples and carries out some numerical experiments by the  MTEM scheme to verify the theoretical results.
			\begin{expl}\label{exp5.1}
				{\rm Recall the SFSDE \eqref{ex1}. 
					The exact solution of the averaged equation with initial value $\bar{x}(0)=x_0$ has the  closed form (see, e.g., \cite{MR2795791,Kloeden})
					\begin{align}
						\bar{x}(t)=\frac{x_0\exp(-\frac{3}{2}t+W^{1}(t))}
						{\sqrt{1+2x_0^{2}\int_{0}^{t}\exp(-3s+2W^{1}(s))\mathrm{d}s}}.\notag
					\end{align}
					It can be verified that $(\text{\bf S1'})$, $({\bf S2})$, $({\bf S3})$,  $({\bf S5})$, $(\text{\bf S6})$ and ${\bf (F1)}$-${\bf (F3)}$ hold with $\theta_{1} = 2,  \theta_2=1$ and any $k\geq 2$. According to Remark \ref{rem4.2}, we can choose $\varphi(u)=1+u^2,~\forall~u\geq1$.
					For the  fixed  $\Delta_1, \Delta_2\in (0,1]$ and integer $M\geq1$, define the MTEM scheme for \eqref{ex1}: for any $n\geq0$,
					\begin{equation}\label{f18}
						\begin{cases}
							X_{0}=x_{0}, T_{\Delta_1}(X_{n})=\Big(|X_{n}|\wedge \big(2\Delta_1^{-\frac{1}{2}}-1\big)^{\frac{1}{2}}\Big)\frac{X_{n}}{|X_{n}|},~Y^{T_{\Delta_1}(X_{n}),y_0}_{0}=y_{0}, \\
							Y^{T_{\Delta_1}(X_{n}),y_0}_{m+1}=Y^{T_{\Delta_1}(X_{n}),y_0}_{m}+(T_{\Delta_1}(X_{n})-Y^{T_{\Delta_1}(X_{n}),y_0}_{m})\Delta_{2}+\Delta W^{2}_{n,m},
							~~~m=0,1, \dots,M-1,\\
							B_{M}(T_{\Delta_1}(X_{n}),Y^{T_{\Delta_1}(X_{n}),y_0})=-\big(T_{\Delta_1}(X_{n})\big)^3-\displaystyle{\frac{1}{M}}\sum_{m=1}^{M}
							Y^{T_{\Delta_1}(X_{n}),y_0}_{m},
							\\X_{n+1}=X_{n}+B_{M}(T_{\Delta_1}(X_{n}),Y^{T_{\Delta_1}(X_{n}),y_0})\Delta_{1}+X_{n}\Delta W^{1}_{n},
						\end{cases}\vspace{-1mm}
					\end{equation}
where $\Delta W^{1}_{n}=W^{1}((n+1)\Delta_1)-W^{1}(n\Delta_1)$ and $\Delta W^{2}_{n,m}=W^{2}_{n}((n+1)\Delta_2)-W^{2}_{n}(m\Delta_2)$.
					Figure \ref{figN2} predicts  the numerical solution generated by the MTEM scheme and the exact solution of the averaged equation \eqref{ex2}. Comparing Figure \ref{figN1} and \ref{figN2} one observes  that the truncation device in the MTEM scheme effectively suppresses the explosive divergence phenomenon of  the PI iteration process.  Correcting the grid points by  using the truncation mapping, the MTEM numerical solution rapidly converges to the exact solution of the averaged equation after going through the initial transient oscillation phase.
					\begin{figure}[htp]
                    \begin{center}
						\includegraphics[width=12cm,height=7cm]{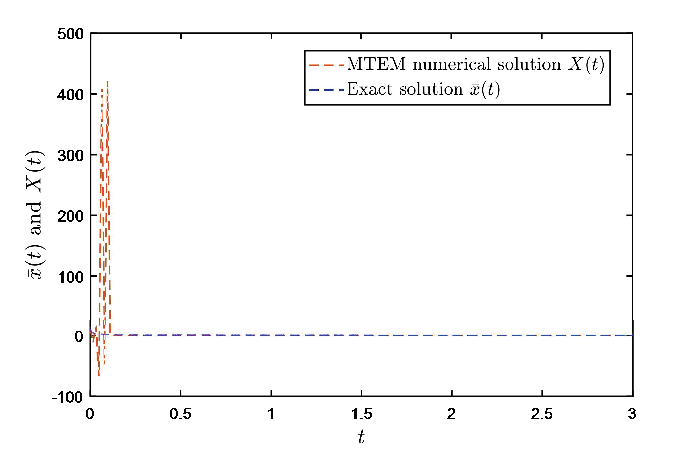}\vspace{-3mm}
						\caption{The sample paths  of the MTEM numerical solution $X(t)$ on $t\in [0,3]$ with $\Delta_1=2^{-6}$, $\Delta_2=2^{-6}$ and $M=2^{18}$. }
						\label{figN2}
					\end{center}\end{figure}
					Owing to Theorem \ref{L1}, one notices that $x^{\varepsilon}(t)$ converges to $\bar{x}(t)$ as $\varepsilon\rightarrow 0$. Next we pay attention to the strong convergence between $\bar{x}(t)$ and the numerical solution $X(t)$  by the MTEM scheme (\ref{f18}) as   $\Delta_1, \Delta_2 \rightarrow0$ and $M\Delta_2\rightarrow\infty$  revealed by Theorem \ref{L5.1}.
					To verify this result, we carry out some numerical experiments by
					the MTEM scheme.  Provided that we want to bound the error by $\mathcal{O}(2^{-q})(q>0)$,  the optimal parameters are derived by Theorem \ref{L5.1} as follows:
					\begin{align*}
						\Delta_1=\mathcal{O}(2^{-q}), ~~\Delta_2=\mathcal{O}(2^{-q}), ~~M=\mathcal{O}(2^{2q}).
					\end{align*}
					In the numerical calculations, using $500$ sample points we compute the sample mean square
					of the  error (SMSE)
					\begin{align*}
						\mathbb{E}|\bar{x}(t)-X(t)|^{2}\approx\frac{1}{500}\sum_{j=1}^{500}|\bar{x}^{(j)}(n\Delta_1)-X^{(j)}_{n}|^2,
					\end{align*}
					where $\bar{x}^{(j)}(n\Delta_1)$ and $X^{(j)}_{n}$ are  sequences of independent copies of $\bar{x}(n\Delta_1)$ and $X_{n}$, respectively. Note that for the fixed $n$ and $j$, $\bar{x}^{(j)}(n\Delta_1)$ and $X^{(j)}_{n}$ are generated
					by a same Brownian motion. Then we carry out numerical experiments by implementing \eqref{f18} using MATLAB.
					In Figure \ref{fig1}, the blue solid line depicts the SMSE for $q=2,3,4,5,6,7$ with $500$ sample points. The red dotted line plots the reference line with
					the slope -1. In addition, we plot 10 groups of sample paths of $\bar{x}(t)$ and $X(t)$ for $t\in [0,5]$ with $(\Delta_1,\Delta_2,M)=(2^{-8},2^{-6},2^{12})$. The Figure \ref{figure2} only depicts four groups of them.
					\begin{figure}[htp]
                    \begin{center}
						\includegraphics[width=10cm,height=5cm]{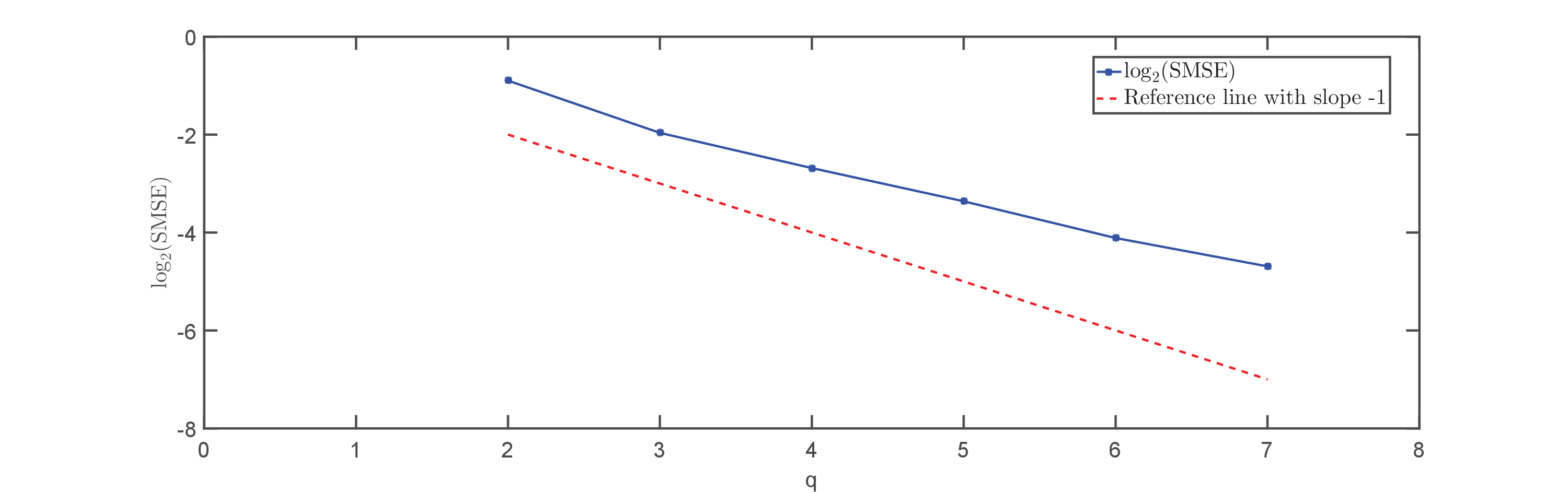}\vspace{-3mm}
						\caption{The SMSE for $q=2,3,4,5,6,7$ with $500$ sample points. The red dashed line is the reference with slope -1.}
						\label{fig1}\vspace{-3mm}
					\end{center}\end{figure}
					\begin{figure}[htp]
                    \begin{center}
						\includegraphics[width=13cm,height=5cm]{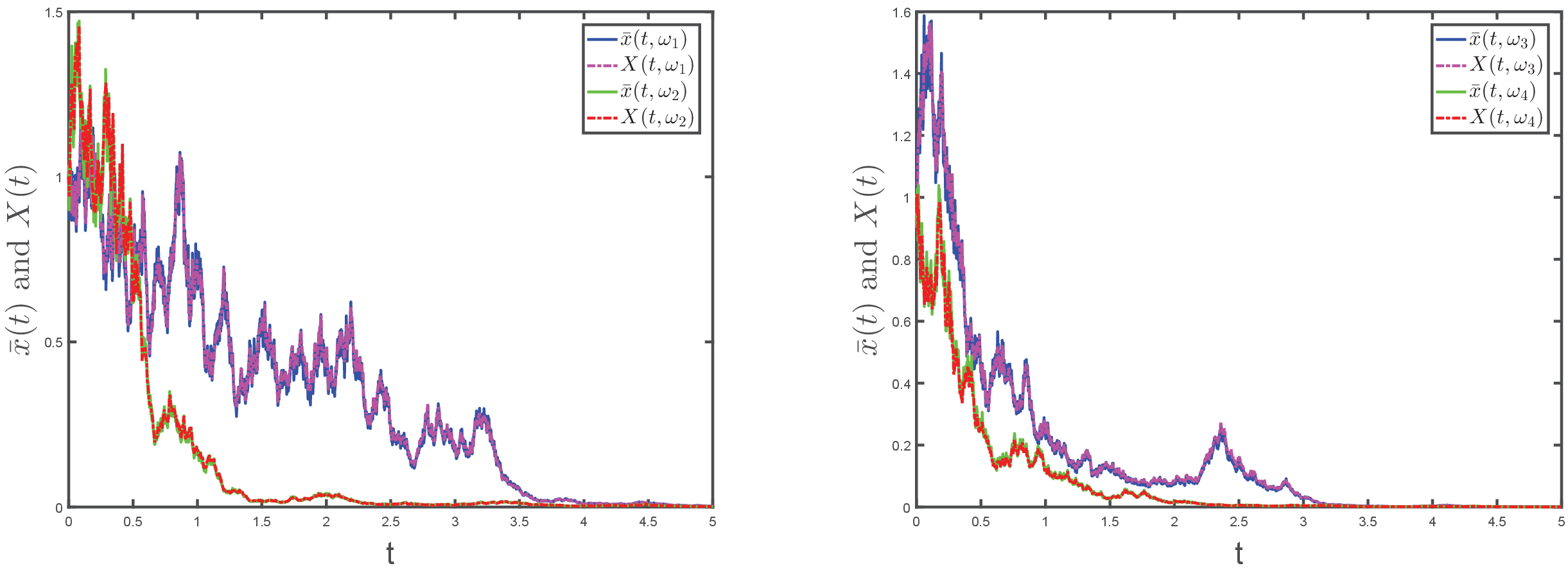}\vspace{-3mm}
						\caption{Four pairs of sample paths of $\bar{x}(t)$  and $X(t)$ for $t\in [0, 5]$ with $(\Delta_1,\Delta_2,M)=(2^{-8},2^{-6},2^{12})$.}
						\label{figure2}\vspace{-3mm}
				\end{center}\end{figure} }
			\end{expl}

\section{Concluding remarks}\label{s-c8}

In this paper, we have developed an explicit numerical scheme tailored for a category of super-linear  SFSDEs  wherein the slow drift coefficient exhibits polynomial growth. An explicit multiscale numerical scheme, termed MTEM, has been proposed through the application of a truncation mechanism. The strong convergence of the numerical solutions yielded by the MTEM scheme has been rigorously established. Furthermore, the convergence rate has been determined under weakly restrictive conditions. The construction of an explicit scheme to approximate the dynamical behaviors of the exact solutions for more generic SFSDEs featuring a super-linear fast component remains an intriguing topic for future investigation. This direction will inform our subsequent research endeavors.

\section*{Acknowledgements}

 The authors would like to thank the associate editor and referees for the
 helpful comments and suggestions.

\end{document}